\documentclass{amsart}

\usepackage[dvipsnames,svgnames,x11names,hyperref]{xcolor}
\usepackage{url,graphicx,verbatim,amssymb,enumerate,stmaryrd}
\usepackage[pagebackref,colorlinks,citecolor=Mahogany,linkcolor=Mahogany,urlcolor=Mahogany,filecolor=Mahogany]{hyperref}
\usepackage{microtype}
\usepackage[all]{xy}
\usepackage[hmargin=3cm,vmargin=3cm]{geometry}
\usepackage{setspace}
\usepackage{amsaddr}

\newcommand{\Q}{\mathbb{Q}}
\newcommand{\Z}{\mathbb{Z}}
\newcommand{\C}{\mathbb{C}}

\newcommand{\R}{\mathbb{R}}
\newcommand{\m}{\longrightarrow}

\newcommand{\N}{\mathbb{N}_0}

\DeclareMathOperator{\im}{im}
\newcommand{\cat}[1]{\mathsf{#1}}
\newcommand{\mr}[1]{{\rm #1}}

\newtheorem{theorem}{Theorem}[section]
\newtheorem{lemma}[theorem]{Lemma}
\newtheorem{proposition}[theorem]{Proposition}
\newtheorem{corollary}[theorem]{Corollary}
\newtheorem{conjecture}[theorem]{Conjecture}
\theoremstyle{definition}
\newtheorem{remark}[theorem]{Remark}
\newtheorem{definition}[theorem]{Definition}
\newtheorem{example}[theorem]{Example}

\author{Alexander Kupers}
\email{kupers@stanford.edu}
\address{Department of Mathematics, Stanford University, 450 Serra Mall, 94305, Stanford, CA, USA}
\thanks{Alexander Kupers is supported by a William R. Hewlett Stanford Graduate Fellowship, Department of Mathematics, Stanford University, and was partially supported by NSF grant DMS-1105058.}

\author{Jeremy Miller}
\email{jeremykmiller@purdue.edu}
\address{Department of Mathematics, Purdue University, 150 N. University Street, 47907-2067, West Lafayette, IN, USA}

\title{Homological stability for topological chiral homology of completions}

\date{\today}

\begin{document}

\begin{abstract}By proving that several new complexes of embedded disks are highly connected, we obtain several new homological stability results. Our main result is homological stability for topological chiral homology on an open manifold with coefficients in certain partial framed $E_n$-algebras. Using this, we prove a special case of a conjecture of Vakil and Wood on homological stability for complements of closures of particular strata in the symmetric powers of an open manifold and we prove that the bounded symmetric powers of closed manifolds satisfy homological stability rationally.\end{abstract}

\maketitle

\tableofcontents

\section{Introduction} In this paper we prove a generalization of homological stability for configuration spaces. Let $M$ be a manifold and let $C_k(M)$ denote the configuration space of $k$ distinct unordered particles in $M$. If $M$ is open, then there is a map $t: C_k(M) \to C_{k+1}(M)$ adding a particle near infinity; homological stability for configuration spaces says that this map is an isomorphism in homology in a range tending to infinity with $k$.

Our generalization involves certain configuration spaces with summable labels. For example, if $A$ is a commutative monoid, then we consider spaces of particles in $M$ with labels in $A$ topologized such that if the particles collide we add their labels. However, such a construction makes sense in a more general setting; the labels only need to have the structure of a so-called \emph{framed $E_n$-algebra} \cite{M,Ge}. The analogous construction of the labeled configuration space is then called \emph{topological chiral homology} \cite{Lu} and is denoted $\int_MA$. It is also known as \emph{factorization homology} \cite{Fr2} or \emph{configuration spaces of particles with summable labels} \cite{Sa}. 

We will consider framed $E_n$-algebras $A$ with $\pi_0 (A) = \N$ that are ``generated by finitely many components,'' a notion made precise using \emph{completions} of framed $E_n$-algebras. If $\pi_0(A)=\N$ the connected components of $\int_M A$ are in bijection with $\N$ for connected $M$ and we denote the $k$th component by $\smash{\int^k_M} A$. When $M$ is open, there is again a stabilization map 
\[t: \int^k_M A \to \int^{k+1}_MA\]
The main result of this paper is that for $A$ generated by finitely many components, this map induces an isomorphism in homology in a range tending to infinity with $k$.

\subsection{Topological chiral homology} We start by introducing the spaces we are interested in, which are defined using topological chiral homology.

Topological chiral homology is a homology theory for $n$-dimensional manifolds. When discussing topological chiral homology, one needs to fix a background symmetric monoidal $(\infty,1)$-category like chain complexes or spectra. In this paper, this category will always be taken to be $\cat{Top}$, the category of topological spaces with Cartesian product. 

Like singular homology of a space depends on a choice of abelian group, the topological chiral homology of a manifold depends on choice of framed $E_n$-algebra. A framed $E_n$-algebra\footnote{The term framed $E_n$-algebra is not completely standard as other authors use it to mean an algebra over an operad homotopy equivalent to the semi-direct product of the little $n$-disks operad with the group $SL_n(\R)$. We want to use $GL_n(\R)$ to be able to prove results for non-orientable manifolds.} is by definition an algebra over an operad homotopy equivalent to the semi-direct product of the little $n$-disks operad with the group $GL_n(\R)$ \cite{M,GJ,Sa}. Given a framed $E_n$-algebra $A$, topological chiral homology gives a functor from the category of smooth $n$-dimensional manifolds (with embeddings as morphisms) to $\cat{Top}$:
\[M \mapsto \int_M A.\] 
We call $\int_M A$ the \textit{topological chiral homology of $M$ with coefficients in $A$}. 

\begin{figure}[t]
\begin{center}
\includegraphics[width=11.5cm]{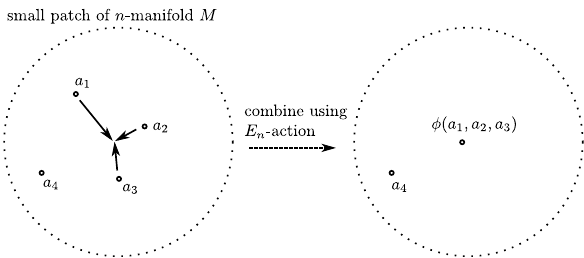}
\end{center}
\caption{An intuitive picture of $\int_M A$ is as a space of particles in $M$ with labels in $A$, topologized such that if the particles collide we get the new labels using the framed $E_n$-action on the labels.}
\label{figtchintuitive}
\end{figure}

Intuitively the topological chiral homology of a manifold $M$ with coefficients in a framed $E_n$-algebra $A$ is the space of particles in $M$ labeled by elements of $A$, topologized in such a way that when several particles collide their labels combine using the framed $E_n$ action (see Figure \ref{figtchintuitive}). The different ways to make this intuition precise lead to different models of topological chiral homology. In the model that is most convenient for our proofs, for example, one keeps track during the collisions of the relative directions and speeds at which the particles collide. In $n$ dimensions such data can be made into a framed $E_n$-operad \cite{Sa}, known as the \emph{framed Fulton-MacPherson operad}. 

One justification for thinking of topological chiral homology as a homology theory are the generalized Eilenberg-Steenrod axioms described in \cite{Fr2}. Another is given by the fact that one can recover ordinary homology from topological chiral homology as follows. Note that abelian groups are examples of framed $E_n$-algebras. For $A$ a discrete abelian group, a variant of the Dold-Thom theorem proves that $\pi_*(\int_M A)=H_*(M;A)$ \cite{DT,K3,Fr2}. From this point of view, topological chiral homology generalizes homology with coefficients in abelian groups to homology with coefficients in framed $E_n$-algebras.

\subsection{Completions of partial algebras}

Topological chiral homology can also be defined for partial framed $E_n$-algebras $P$. A \emph{partial algebra} over an operad is like an actual algebra over that operad except only some of the operadic compositions are defined. In this case one can think of topological chiral homology of $M$ with coefficients in a partial framed $E_n$-algebra $P$ as particles in $M$ labeled by elements of $P$ that are only allowed to collide if that multiplication is defined. The resulting space is denoted $\int_M P$.

Any partial framed $E_n$-algebra can be completed to form an actual framed $E_n$-algebra. One can do this by either taking the \emph{completion} $\bar{P}$ to be the adjoint to the inclusion of actual algebras into partial algebras, or by setting $\bar{P} = \int_{D^n} P$ (see Proposition 5.16 of \cite{Sa}\footnote{The theorem numbering in \cite{Sa} does not agree with that of the arXiv version. However, in all theorems and definitions that we cite, the numbering in the published version is one greater than the numbering in the arXiv version.}). These constructions are homotopy equivalent and furthermore the topological chiral homology of $\bar{P}$ is homotopy equivalent to the topological chiral homology of $P$. It is these spaces that will be the main object of study in this paper.

There are several important examples of framed $E_n$-algebras obtained as such completions. Any strictly commutative monoid is a framed $E_n$-algebra for all $n$. Let $\N$ denote the commutative monoid of non-negative integers with addition. A completion of $\{1\} \subset \N$ as a framed $E_n$-algebra is the space of points in a disk $D^n$ labeled by $1$ and these are not allowed to collide since no multiplication is defined in the partial monoid $\{1\}$. The connected components of this space are known as the \emph{configuration spaces of distinct unordered particles in the disk}, each of them given by 
\[C_k(D^n)=((D^n)^k \backslash \Delta_{f})/\mathfrak S_k\] for some $k \geq 1$. Here $\mathfrak S_k$ is the symmetric group on $k$ letters acting by permuting the terms and $\Delta_{f}$ is the fat diagonal $\{(m_1,\ldots,m_k)\,|\,m_i = m_j \text{ for some $i \neq j$}\}$. 

For $c$ a natural number, let $\{1,\ldots,c\}$ denote the partial abelian monoid induced by viewing $\{1,\ldots,c\}$ as a subset of $\N$, viewed as monoid with addition. That is, one is only allowed to add numbers if their sum is less than or equal to $c$. In that case, the partial monoid operation agrees with addition of natural numbers. Completing $\{1,\ldots,c\}$ gives a space of points labeled by a number between 1 and $c$, which we call the \textit{charge}, and these labeled points are allowed to collide if their total charge is less than or equal to $c$, in which case the charges simply add. This is exactly a disjoint union of bounded symmetric powers $\mr{Sym}^{\leq c}_k(D^n)$, with $\mr{Sym}^{\leq c}_k(D^n)$ denoting the subspace of $\mr{Sym}_k(D^n)=(D^n)^k/\mathfrak S_k$ where no more than $c$ points coincide.  More generally taking as a partial algebra $\sqcup_{1 \leq i \leq c} X^i/\mathfrak S_i$ for $X$ a space gives labeled configuration spaces in the case $c=1$ or labeled bounded symmetric powers in the case of general $c$.

\subsection{Homological stability}\label{subsechomstobintroduction} Homological stability is the phenomenon that for many naturally defined sequences of spaces $X_k$, there are maps $t : X_k \to X_{k+1}$ which induce isomorphisms in homology in a range tending to infinity as $k$ tends to infinity. Classical examples of spaces exhibiting homological stability include symmetric powers $\mr{Sym}_k(M)=M^k/\mathfrak S_k$ \cite{Srod} and configuration spaces $C_k(M)$, as long as $M$ is a connected manifold which is the interior of a manifold with non-empty boundary (we call this condition \emph{admitting a boundary}). Interpolating between these two examples are the bounded symmetric powers $\mr{Sym}^{\leq c}_k(M) \subset \mr{Sym}_k(M)$ which consist of all configurations where at most $c$ points can coincide. More recently, Yamaguchi in \cite{Y} proved homological stability for the spaces $\mr{Sym}^{\leq c}_k(M)$ when $M$ is a punctured orientable surface. In these three examples the maps which induce the isomorphisms are so-called stabilization maps, ``bringing a particle in from infinity'' (see Definition \ref{defstabilizationmap}).

As mentioned in the previous subsection, all three of these examples -- the symmetric powers $\mr{Sym}_k(M)$, configuration spaces $C_k(M)$ and bounded symmetric powers $\mr{Sym}_k^{\leq c}(M)$ are examples of topological chiral homology. Furthermore, the latter two are naturally the topological chiral homology of completions of partial framed $E_n$-algebras. The first two of these framed $E_n$-algebras were previously known to have the property that their topological chiral homology has homological stability, at least for the class of connected manifolds admitting boundary. 

We generalize this situation as follows. Let $P$ be a partial framed $E_n$-algebra with $\pi_0(P)= \{1,\ldots, c\}$ (see Remark \ref{partialpi0}). For $M$ connected, $\pi_0(\int_M P)$ will then be in bijection with the non-negative integers and we denote the $k$th component by $\smash{\int_M^k P}$. The goal of this paper is to prove that the spaces $\smash{\int_M^k P}$ have homological stability whenever $M$ admits boundary (i.e. is the interior of a manifold with non-empty boundary) and is not one-dimensional. This uses a stabilization map $t: \int^k_M P \to \int^{k+1}_M P$ which depends on a choice of manifold with boundary $\bar{M}$ with interior $M$ and an embedding $D^{n-1} \hookrightarrow \partial \bar{M}$. Note that up to homotopy, the stabilization map only depends on a choice of end of the manifold.

\begin{theorem}\label{thmmain} Let $M$ be a connected manifold of dimension $\geq 2$, oriented if it is of dimension 2, admitting a boundary. The stabilization map 
\[t: \int^k_M P \to \int^{k+1}_M P\]
induces an isomorphism  on homology in degrees $* < \lfloor \frac{k}{2c} \rfloor$ and a surjection in degree $* = \lfloor \frac{k}{2c} \rfloor$.\end{theorem}

This is proven for $\dim M \geq 3$ in Theorem \ref{thmmaind3} and for $\dim M = 2$ in Theorem \ref{thmmaind2}. There can be no such result in dimension 1: take $P$ to be the partial framed $E_n$-algebra given by a circle in charge 1 and no operations defined except the identity. Then $\int_{(0,1)}^k P \simeq (S^1)^k$ and this does not satisfy homological stability. It is true for non-orientable connected surfaces admitting a boundary in a slightly worse range, see Corollary \ref{corunorientedstability}. We also prove a generalization to topological chiral homology of a partial algebra with a set of \emph{exceptional particles}, see Theorem \ref{thmmainexcep}.

The slope $\frac{1}{2c}$ may not be optimal. It is optimal for $c=1$: for configuration spaces of particles in $\R^2$, the explicit description of its mod $2$ homology in Theorem III.A.1 of \cite{CLM} gives a slope of $\frac{1}{2}$. If $c>1$, the slope we prove differs from the optimal slope by a factor of at most $2$. To see this, consider the completion of the partial framed $E_n$-algebra with $\pi_0 = \{1,\ldots,c\}$ given by a point in charge $<c$ and a circle in charge $c$. This cannot have a stability range with a higher slope than $\frac{1}{c}$. 

By the examples in Subsection \ref{subsechomstobintroduction}, applications include homological stability theorems for configuration spaces, bounded symmetric powers and labeled variations of these.

\subsection{Complements of closures in symmetric powers} One of the motivations for this paper was Conjecture F of \cite{VW}, a particular instance of which is implied by our main theorem. The conjecture, motivated by computations using motivic $\zeta$-functions, concerns rational homological stability for certain subspaces of the symmetric powers of a manifold.

To state this conjecture, we fix some notation. Any point $x=\{x_1, \ldots x_k\} \in \mr{Sym}_k(M)$ determines a partition of the number $k$ by recording the multiplicity of each point $x_i$ appearing in $x$. For example, $(1,2,2,5-i) \in \mr{Sym}_4(\C)$ corresponds to the partition $1+1+2$. 

\begin{definition}Let $\lambda$ be a partition of $k$. Let $\mr{Sym}_\lambda(M)$ denote the subspace of $\mr{Sym}_k(M)$ whose corresponding partition is $\lambda$. We define $W_\lambda(M)$ to be the complement in $\mr{Sym}_k(M)$ of the closure of $\mr{Sym}_\lambda(M)$.\end{definition}

Given a partition $\lambda=(a_1 + \ldots + a_n)$ of $k$, let $1^j \lambda$ be the partition $(1 + 1 + \ldots + 1 + a_1 + \ldots + a_n)$ of $j+k$ obtained by adding $j$ ones to $\lambda$. After observing that a similar stability result holds in the Grothendieck ring of varieties (Theorem 1.30a of \cite{VW}), Vakil and Wood made the following conjecture (Conjecture F of \cite{VW}):

\begin{conjecture}[Vakil-Wood] \label{conjf}
For all partitions $\lambda$ and irreducible smooth complex varieties $X$, the groups $H_i(W_{1^j \lambda}(X);\Q)$ are independent of $j$ for $j$ sufficiently large compared to $i$.  
\end{conjecture}

Topologists might generalize this to a conjecture regarding the rational homology of $W_{1^j\,\lambda}(M)$ for $M$ a manifold and it is this generalization that we will consider in this paper. Note that $W_{1^j2}(M)=C_{j+2}(M)$ and more generally $W_{1^jc+1}(M)=\mr{Sym}_{j+c+1}^{\leq c}(M)$. Using Theorem \ref{thmmainexcep} we prove Conjecture \ref{conjf} for a larger class of partitions. Let $m^i$ denote the partition of the number $mi$ into $i$ sets of size $m$.

\begin{corollary}\label{corravimelanie}Let $M$ be a connected manifold of dimension $\geq 2$, oriented if it is of dimension 2, admitting a boundary and let $\lambda = (c+1)^{m+1}$. Then the stabilization map
\[t: W_{1^k \lambda}(M) \to W_{1^{k+1} \lambda}(M)\]
induces an isomorphism on homology in degrees $*\leq \lfloor \frac{k-(m+1)(c+1)+1}{2c} \rfloor$.\end{corollary}

Note that this result is true with integral coefficients, not just rational coefficients. For rational coefficients and $\lambda=c+1$, we will also prove Conjecture \ref{conjf} when $M$ is closed.

\subsection{Stable homology and closed manifolds} Every time one proves a homological stability result, the obvious follow-up problem is to determine the stable homology. In the case of topological chiral homology, a description of the stable homology precedes the stability result.

The technique used to determine the stable homology is known as \emph{scanning}. It was previously used in the context of configuration spaces by Segal \cite{Se3} and McDuff \cite{Mc1} and in the context of bounded symmetric powers by Kallel \cite{K2} and Yamaguchi \cite{Y}. Any framed $E_n$-algebra $A$ has an $n$-fold classifying space denoted $B^n A$ \cite{M}, which is unique up to weak homotopy equivalence. When the monoid $\pi_0(A)$ is actually a group, this $n$-fold classifying space can be viewed as an $n$-fold delooping. That is, $B^n A$ is a $(n-1)$-connected space whose $n$-fold based loop space is homotopy equivalent to $A$. If we apply this $n$-fold classifying space  procedure locally on an $n$ dimensional manifold $M$, we get a locally trivial fiber bundle $B^{TM} A$ with fiber $B^n A$ \cite{Sa}. The \emph{scanning map} is a map: 
\[s:\int_M A \to \Gamma_M(B^{TM} A)\]
where $\Gamma_M(B^{TM} A)$ is the space of compactly supported sections of the bundle $B^{TM} A$ over $M$, i.e. sections that send the complement of a compact subset of $M$ to the base point of $B^n A$. See Definition \ref{defscan} for the definition of the scanning map. 



The scanning map is a weak homotopy equivalence when $\pi_0(A)$ is a group \cite{Sa} \cite{Lu}, a result known as \emph{non-abelian Poincar\'e duality}. However, the scanning map is also interesting when $\pi_0(A)$ is not a group. In \cite{Mi2}, it was shown that when $M$ admits boundary the mapping telescope of $\int_M A$ under the stabilization maps is homology equivalent to $\Gamma_M(B^{TM} A)$, see Theorem \ref{thmscanning}. By substituting $A = \bar{P}$, the completion of the partial framed $E_n$-algebra $P$ and using the results of this paper, we conclude that the scanning map $s:\int^k_M P \to \Gamma^k_M(B^{TM} P)$ is a homology equivalence in a range tending to infinity with $k$. Here the superscript $k$ denotes the subspace of degree $k$ sections. We thus deduce the following corollary.

\begin{corollary}\label{stabrange} Suppose $M$ admits boundary and $\pi_0(P)=\{1,\ldots,c\}$. Then the scanning map induces an isomorphism between $H_*(\int^k_M P)$ and $H_*(\Gamma_M^k(B^{TM} P))$ for $* < \lfloor \frac{k}{2c} \rfloor$ and a surjection for $* \leq \lfloor \frac{k}{2c} \rfloor$.\end{corollary}

Using the theory of homology fibrations, one can extend the above result to the case that $M$ is closed at the cost of making the range slightly worse (see Theorem \ref{thmscanisoclosed}). Since the homology groups of $\Gamma_M(B^{TM} P)$ are often computable, this gives a method of calculating the homology of $\int_M P$ in a range. 

This theorem is also the way we approach rational homological stability for bounded symmetric powers of closed manifolds. For $M$ closed, one cannot define stabilization maps as there is no direction from which to bring in a point. To prove that the homology of $\int^k_M P$ agrees with that of $\int^{k+1}_M P$ in a range, it suffices to prove that the corresponding components of $\Gamma_M(B^{TM} P)$ are homology equivalent. In the case $P=\{1,\ldots c\}$, we can this if we work with rational coefficients, using bundle automorphisms which only exist after rational localization. In particular, we prove the following theorem, which establishes Conjecture \ref{conjf} for $\lambda=c+1$ and \emph{all} manifolds $M$, not just those admitting boundary.

\begin{theorem}\label{thmtruncstab}
Let $i < \min( \lfloor \frac{k-c}{2c} \rfloor,  \lfloor \frac{j-c}{2c} \rfloor )$ and assume either $n$ is odd or $j,k \neq \chi(M)/2$. Then we have that 
\[H_i(\mr{Sym}_j^{\leq c}(M);\Q) \cong H_i(\mr{Sym}_{k}^{\leq c}(M);\Q)\]
\end{theorem}

\subsection{Structure of the paper} 

The organization of the paper is as follows. In Section \ref{sectch} we review Salvatore's definition of topological chiral homology based on the Fulton-MacPherson operad and recall scanning theorems from \cite{Sa} and \cite{Mi2}. We also give a variation of this definition for partial algebras. In Section \ref{seccomplexes} we describe several simplicial complexes and prove they are highly connected. This will be used in Section \ref{sechomstab} to prove homological stability for topological chiral homology of manifolds admitting boundary with coefficients in completions of partial framed $E_n$-algebras. We then generalize this to other subspaces of the topological chiral homology, generalizing those appearing in Conjecture F of \cite{VW}, here Conjecture \ref{conjf}. In Section \ref{secscanningequiv} we discuss the scanning map for topological chiral homology of closed manifolds, which we use in Section \ref{secrathomstabtrunc} to prove rational homological stability for bounded symmetric powers of closed manifolds.

\subsection{Related subsequent work} After the first version of this manuscript was made publicly available, Conjecture \ref{conjf} was completely proven by the authors and  TriThang Tran in \cite{KMT}. The techniques of that paper are different than those employed here. In the open manifold case, it uses an argument involving compactly supported cohomology following the ideas of Arnol'd and Segal. In the closed manifold case, instead of using scanning and localization, it uses a resolution by punctures argument from \cite{RW}. The homological stability range established in \cite{KMT} is better than the one obtained here, but it only works rationally.

Another related paper written subsequently is \cite{KMCell}. Among many other results, that paper gives an improved integral homological stability range for bounded symmetric powers of open orientable manifolds. The techniques of \cite{KMCell} are different from those employed here, resolving $E_n$-algebras by so-called $E_n$-cells.

\subsection{Acknowledgments} We would like to thank Martin Bendersky, S\o ren Galatius, Martin Palmer, Ravi Vakil, Melanie Matchett Wood and the referee for many useful conversations and comments.


\section{Definition of topological chiral homology of partial algebras} \label{sectch}

In this section we define topological chiral homology. Usually one considers smooth manifolds -- as will be the case in this paper -- but the theory also applies to PL or topological manifolds. Topological chiral homology is also known as factorization homology. When one works with framed $E_n$-algebras in spaces, another model of topological chiral is given by Salvatore's configuration spaces of particles with summable labels introduced in \cite{Sa}. Topological chiral homology has applications to higher dimensional string topology \cite{ginottradlerzeinalian2}, quantum field theory \cite{CG} and the geometric Langlands program \cite{Gat}. References for topological chiral homology include \cite{An, Fr1, Fr2, ginottradlerzeinalian, Lu, Sa}. 

In order to work with general $n$-dimensional manifolds, we need to use framed $E_n$-algebras as coefficients. However, if one restricts attention to parallelized $n$-dimensional manifolds, the coefficients require less structure and one can define topological chiral homology using $E_n$-algebras instead. More generally, one can work with manifolds with any tangential structure if one uses algebras over the appropriate modification of $E_n$. The main theorem is true for these variations, as well as for PL and topological manifolds. We also believe that these results hold for the locally constant cosheaves of framed $E_n$-algebras considered in \cite{Lu}.


In this section we describe Salvatore's model of topological chiral homology \cite{Sa}. It is convenient for our purposes since it is easy to relate to classical particle spaces such as bounded symmetric powers. It also has the advantage that no higher categorical constructions are needed, which limits the amount of background necessary, but of course limits its generality. We will also give a variation of Salvatore's model for topological chiral homology of partial algebras $P$ with $\pi_0 (P) = \{1,\ldots,c\}$, which requires a detailed discussion in Subsection \ref{subsecalttch} of the geometry of the Fulton-MacPherson compactification.

\subsection{Operads} Before we define topological chiral homology, we briefly review operads and their partial algebras. An \emph{operad} is an object that encodes a particular type of algebraic structure, essentially those consisting of $n$-ary operations with relations that do not involve repetitions of terms. Operads and their algebras were first introduced by May in \cite{M} to study iterated loop spaces. In this paper the term operad always means symmetric operad in the category $\cat{Top}$ of topological spaces with a unit element in $\mathcal O(1)$ denoted by $1$.

This means that our operads will consist of a sequence of spaces $\{\mathcal O(k)\}_{k \geq 1}$ together with the following data: (i) a unit element $1 \in \mathcal O(1)$, (ii) an action of the symmetric group $\mathfrak S_k$ on $\mathcal O(k)$, (iii) a collection of operad composition maps $\mathcal O(k) \times \prod_{i=1}^k \mathcal O(j_i) \to \mathcal O(\sum_{i=1}^k j_i)$ that are associative, compatible with the unit element and compatible with the symmetric group actions. Now we recall the definition of the \emph{monad} in spaces associated to an operad, which is really the object we will be working with.

\begin{definition}
For $\mathcal O$ an operad, let $\mathbf O : \cat{Top} \to \cat{Top}$ be the functor which sends a space $X$ to 
\[\mathbf O X =\bigsqcup_{k \geq 1} \mathcal O(k) \times_{\mathfrak S_k} X^k.\] Here $\mathcal O(k)$ is the $k$th space of the operad and $\mathfrak S_k$ is the symmetric group on $k$ letters. \end{definition}


This functor sends a space to the free $\mathcal O$-algebra on the space. The operad composition and unit of the operad  $\mathcal O$ respectively give natural transformations $\mathbf O \mathbf O \to \mathbf O$ and $\mr{id} \to \mathbf O$. These natural transformations make $\mathbf O$ into a monad (a monoid in the category of functors).

We now recall the definition of a partial algebra over an operad. To motivate this definition we first review algebras over an operad. One definition of an algebra $A$ over an operad $\mathcal O$ is a space $A$ together with a map $\mathcal O(k) \times A^k \to A$ compatible with the operad composition, symmetric group actions and the unit. Equivalently one can define it using the monad $\mathbf O$, which is the definition we will use in this paper.

\begin{definition}Let $\mathcal O$ be an operad with associated monad $\mathbf O$. An \emph{$\mathcal O$-algebra structure} on a space $A$ is a map $m: \mathbf O A \to A$ making the following diagrams commute:
\[\xymatrix{A \ar[r]^(.4){\{1\} \times \mr{id}} \ar[rd]_{\mr{id}} & \mathbf O A \ar[d]^m & & \ar[d] \mathbf O \mathbf O A \ar[r] & \mathbf O A \ar[d]^m \\
 & A & & \mathbf O A \ar[r]_m & A.}\] The two maps $\mathbf O \mathbf O A \to \mathbf O A$ in the diagram above are induced by the monad multiplication natural transformation $\mathbf O \mathbf O \to \mathbf O$ and the algebra composition map $m: \mathbf O A \to A$ respectively.\end{definition}

The definition of a partial algebra over $\mathcal O$ is now straightforward. It is essentially the same as an algebra, except that the map $m$ only needs to be defined on a subspace of $\mathbf O A$.

\begin{definition}
Let $\mathcal O$ be an operad with associated monad $\mathbf O$. A \emph{partial $\mathcal O$-algebra structure} on a space $P$ is the data of a subspace $Comp \subset \mathbf O P$ containing $ \{1\} \times P$, and a map $m: Comp \to P$ making the following diagrams commute: \[\xymatrix{P \ar[r]^(.4){\{1\} \times \mr{id}} \ar[rd]_{\mr{id}} & Comp \ar[d]^m & & \ar[d] Comp_2 \ar[r] & Comp \ar[d]^m \\
 & P & & Comp \ar[r]_m & P}\]  Here $Comp_2 \subset \mathbf O^2 P$ is the inverse image of $Comp$ under the monad multiplication map $\mathbf O^2 P \to \mathbf O P$. Moreover, we require that the subspace $Comp_2$ is also equal to the inverse image of $Comp$ under the map $\mr{id} \times m: \mathbf O Comp \m \mathbf O P$. The two maps $Comp_2 \to Comp$ in the diagram above are induced by the monad multiplication natural transformation $\mathbf O \mathbf O \to \mathbf O$ and the partial algebra compositions $m: Comp \to P$ respectively.
\end{definition}

One can complete partial algebras into an actual algebra using the following construction.

\begin{definition}
Let $\mathcal O$ be an operad and $P$ be a partial $\mathcal O$-algebra. Let the \emph{completion} $\bar P$ be the coequalizer of the two natural maps: \[ \mathbf O Comp \rightrightarrows \mathbf O P.\] Here the first map is induced by the map $m:Comp \m P$ and the second is induced by the natural transformation $\mathbf O \mathbf O \m \mathbf O$.
\end{definition}

Note that $\bar P$ is an $\mathcal O$-algebra with a natural map $P \to \bar P$. In fact, $\bar P$ is the initial object the category of $\mathcal O$-algebras with a map from $P$ respecting the partial $\mathcal O$-algebra structure. See the end of Subsection \ref{subsecalttch} for examples of completions of partial algebras. 

\begin{remark}
This completion should be viewed as a strict completion, in contrast to the homotopy completion used in \cite{Mi3} and \cite{KMCell}. We believe that these two completion procedures yield homotopic algebras under suitable cofibrancy assumptions.  Using Proposition 5.16 of \cite{Sa} and the arguments of Section 3.3 of \cite{Mi3}, one can show that the two completions are homotopic in the case that $\mathcal O$ is the unbased Fulton-MacPherson operad.  This uses that the  unbased Fulton-MacPherson operad is cofibrant, Corollary 3.8 of \cite{Sa}.
\end{remark}

\begin{example}
Note that $\{1,\ldots,c\}$ is a partial commutative monoid and hence a partial algebra over both the associative operad and the commutative operad. Its completion as an algebra over either operad is isomorphic to $\N$. See Example \ref{examplecompletionEn} for a description of its completion as a partial algebra over other operads.

\end{example}

\begin{remark}\label{partialpi0}
For $A$ an $\mathcal O$-algebra, we have that $\pi_0(A)$ is naturally a $\pi_0(\mathcal O)$-algebra. However, for $P$ a partial $\mathcal O$-algebra, there is no natural partial $\pi_0(\mathcal O)$-algebra action on $\pi_0(P)$. This can be remedied by adding in the following assumption. We say that the partial algebra structure on $P$ respects path components if for all $o,o' \in \mathcal O(k)$ in the same path component and all $p_i,p_i' \in P$ in the same path component, $(o;p_1,\ldots,p_k) \in  Comp$ implies that $(o';p_1',\ldots,p_k') \in  Comp$. It is clear that if the partial algebra structure on $P$ respects path components, then $\pi_0(P)$ carries a natural $\pi_0(\mathcal O)$-algebra structure. In the case of $E_n$ or framed $E_n$-operads, this makes $\pi_0(P)$ a partial monoid. If $n>1$, then this partial monoid is commutative. Whenever we say $\pi_0(P)=W$ for some partial $\pi_0(\mathcal O)$-algebra $W$, we implicitly assume that the partial algebra structure on $P$ respects path components.
\end{remark}

Next we recall the definition of a right functor over an operad.

\begin{definition}Let $\mathcal O$ be an operad. A \emph{right $\mathcal O$-algebra structure} on a functor $\mathbf R: \cat{Top} \to \cat{Top}$ is a pair of natural transformations $\iota: \mathbf R \to  \mathbf R \mathbf O$ and  $\mu: \mathbf R \mathbf O \to \mathbf R$ making the following diagrams commute: \[\xymatrix{\mathbf R \ar[r]^(.4){\iota} \ar[rd]_{\mr{id}} & \mathbf R \mathbf O \ar[d]^{\mu} & & \ar[d] \mathbf R \mathbf O \mathbf O \ar[r] & \mathbf R \mathbf O \ar[d]^{\mu} \\
 & \mathbf R & &  \mathbf R \mathbf O  \ar[r]_{\mu} & \mathbf R.}\] The two natural transformations $\mathbf R \mathbf R \mathbf O  \to \mathbf R \mathbf O $ in the diagram above are induced by monad multiplication natural transformation $\mathbf O \mathbf O \to \mathbf O$ and $\mu: \mathbf R \mathbf O \to \mathbf R$ respectively.\end{definition}

\subsection{Topological chiral homology of framed $E_n$-algebras} \label{subsectchdef}
We will now describe the Fulton-MacPherson operad and Salvatore's model of topological chiral homology. This involves certain compactifications of ordered configuration spaces. See \cite{Sa}, \cite{Mi2} and \cite{sinha} for a more detailed treatment of the topics in this subsection. From now on, $M$ will denote an $n$-dimensional manifold.

\begin{definition}The \emph{configuration space of $k$ distinct ordered particles in $M$}  is defined to be 
\[\tilde{C}_k(M) = M^k \backslash \Delta_f\]
where $\Delta_f=\{(m_1,\ldots ,m_k) \in M^k | m_i=m_j \text{ for some }i \neq j \}$ is the fat diagonal.\end{definition}

Pick a proper embedding $M \subset \R^m$. Then there is a map
\[J: \tilde{C}_k(M) \to M^k \times (S^{m-1})^{\binom{k}{2}} \times [0,\infty]^{\binom{k}{3}}\]
given on a configuration $(m_1,\ldots,m_k)$ by a product of the following maps: \begin{itemize}
\item the inclusion into the product $M^k$, 
\item for all $i<j$ the maps $(m_i,m_j) \mapsto u_{ij} = \frac{m_i-m_j}{|m_i-m_j|}$,
\item for all $i<j<l$ the maps $(m_i,m_j,m_l) \mapsto d_{ijl} = \frac{|m_i-m_j|}{|m_i-m_l|}$.
\end{itemize} 
Using this we can define the Fulton-MacPherson configuration space. This is the definition given in \cite{sinha}.

\begin{definition}The \emph{Fulton-MacPherson configuration space $C_{FM}(M)(k)$ of $k$ distinct ordered particles in $M$} is the closure of the image of $J$.
\end{definition}

This space is independent up to homeomorphism of the choice of embedding. It can alternatively be constructed as the closure of the map $\tilde{C}_k(M) \to \prod_{S \subset [k]} \mr{Bl}_{\Delta_s} M^S$ where $\mr{Bl}_{\Delta_s}$ is the real blow-up along the skinny diagonal. This is the construction used in \cite{Sa}.

In Subsection \ref{subsecalttch} we will give more details about the geometry of $C_{FM}(M)(k)$, but we will give a summary of the required results now. The map $J$ induces an inclusion $J:\tilde{C}_k(M) \to C_{FM}(M)(k)$. When $M$ is compact, so is $C_{FM}(M)(k)$. In general one can think of the Fulton-MacPherson configuration space as a partial compactification of $\tilde{C}_k(M)$. In fact, $C_{FM}(M)(k)$ is a manifold with corners whose interior is diffeomorphic to $\tilde{C}_k(M)$ and thus homotopy equivalent to $\tilde{C}_k(M)$.  There is also a natural map 
$L:C_{FM}(M)(k) \to M^k$ induced by the inclusion $\tilde{C}_k(M) \to M^k$. The map $L$ forgets all of the extra data from the blow ups and records only the ``macroscopic location'' of the points in $C_{FM}(M)(k)$. The action of $\mathfrak S_k$ on $\tilde{C}_k(M)$ by permuting the points extends to $C_{FM}(M)(k)$. 

One can think of the Fulton-MacPherson configuration space as a space of macroscopic configurations in $M$, in each point of which can live a microscopic configuration up to translation and dilation, etc. See Figure \ref{figfultonmacpherson} for an example.

\begin{figure}[t]
\begin{center}
\includegraphics[width=11.5cm]{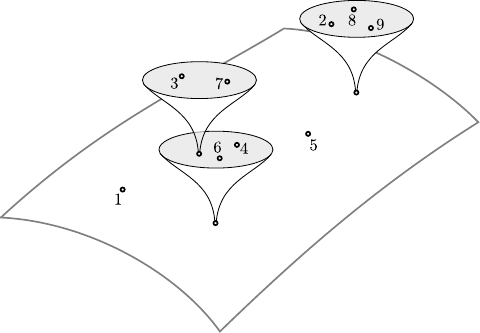}
\end{center}
\caption{A local picture of the Fulton-MacPherson compactification.}
\label{figfultonmacpherson}
\end{figure}

\begin{definition}
Let $F_n(k)$ denote the subspace of $C_{FM}(\R^n)(k)$ of points macroscopically located at the origin. 
\end{definition}

These spaces assemble to form an operad called the Fulton-MacPherson operad, originally introduced by Getzler and Jones in \cite{GJ}. See Proposition 3.5 of \cite{Sa} for the definition of the operad structure on $F_n$, given by grafting trees of microscopic configurations, and Proposition 4.9 of \cite{Sa} for proof that $F_n$ is an $E_n$-operad, by comparing it to the little $n$-disks operad $D_n$ using the Boardman-Vogt resolution $WD_n$ \cite{BV}.

Since we are interested in topological chiral homology of manifolds that are not necessarily parallelizable, we will need to consider the framed version of this operad. First we recall the definition of the semi-direct product of a group $G$ and an operad $\mathcal{O}$ in $G$-spaces (an operad in spaces where each space has a $G$-action and all structure maps are $G$-equivariant).

\begin{definition}\label{defsemidirectoperad}
Let $G$ be a group and let $\mathcal{O}$ be an operad in $G$-spaces. Define $\mathcal O \rtimes G$ to be the topological operad with $(\mathcal O \rtimes G)(k)=\mathcal O(k) \times G^k$ and composition map $\tilde m$ given by: $$\tilde m((o,g_1, \ldots g_k);(o_1,g^1_1, \ldots g^{m_1}_1), \ldots,(o_k,g^1_k, \ldots g^{m_k}_k)) =$$ $$( m(o;g_1 o_1, \ldots, g_k o_k),g_1g_1^1,\ldots, g_kg_k^{m_k})$$ with $m$ the composition map of the operad $\mathcal O$. 
\end{definition}

Now we can define the framed Fulton-MacPherson operad and the framed Fulton-MacPherson configuration space.

\begin{definition}
Define the \emph{framed Fulton-MacPherson operad} $fF_n$ as $F_n \rtimes GL_n(\R)$ with the action of $GL_n(\R)$ on $F_n(k)$ induced by the action of $GL_n(\R)$ on $\R^n$. 
\end{definition}

See Definition 5.3 of \cite{Sa} for more details on this action.

\begin{definition}
Let $P \to M$ be the frame bundle of the tangent bundle, i.e. principal $GL_n(\R)$-bundle associated to the tangent bundle of an $n$-dimensional manifold $M$ by pulling back the universal one from $BO(n)$. Define the \emph{framed Fulton-MacPherson configuration space of $k$ ordered points in $M$}, denoted $fC_{FM}(M)(k)$, to be the pullback of the diagram \[ \xymatrix{fC_{FM}(M)(k) \ar[r] \ar[d] & P^k \ar[d]\\
C_{FM}(M)(k) \ar[r]_L & M^k}\]
 with $L$ the macroscopic location map. Note that the symmetric group action on $M^k$ induces an action on $fC_{FM}(M)(k)$.
\end{definition}

\begin{definition}
Let $\mathbf{ fC(M)}: \cat{Top} \to \cat{Top}$ denote the functor 
\[X \mapsto \mathbf{ fC(M)} X =\bigsqcup_{k \geq 1}  f C_{FM}(M)(k) \times_{\mathfrak S_k} X^k.\] 
\end{definition}

See Proposition 4.5 of \cite{Sa} for a description of a right $fF_n$-functor structure on $\mathbf{fC(M)}$. We can now define topological chiral homology.

\begin{definition}\label{deftch} For $M$ a smooth $n$-dimensional manifold and $P$ a partial $fF_n$-algebra, let $\int_M P$ denote the coequalizer of the two natural maps: \[ \mathbf{ fC(M)} Comp \rightrightarrows \mathbf{ fC(M)} P.\] Here the first map is induced by the partial algebra structure and the second is induced by the right module structure. 
\end{definition}

The space $\int_M P$ is called the topological chiral homology of $M$ with coefficients in $P$. There are many other models of topological chiral homology. If one uses a framed $E_n$-operad which is not cofibrant, then one should use a homotopy invariant construction to define topological chiral homology, such as the simplicial spaces considered in \cite{Mi2}. The following follows directly from the definitions:

\begin{proposition}We have a homeomorphism:
\[\int_M \bar{P} \cong \int_M P.\]
\end{proposition}

When $M$ is connected, there is an isomorphism $\pi: \pi_0\left(\int_M P\right) \cong \pi_0(\bar P)$ given by composing the labels of all of the particles.

\begin{definition}Given $\alpha \in \pi_0(\bar P)$, define $\int_M^\alpha P$ to be the path component that is mapped to $\alpha$ by $\pi$.\end{definition}

For $N \subset M$, one can also define a right $fF_n$-functor $\mathbf{ f C(M,N)}$. 

\begin{definition}

Let $\mathbf{f C(M,N)}: \cat{Top} \to \cat{Top}$ be the functor which sends a space $X$ to: 
\[ \mathbf{fC(M,N) }X =  \mathbf{fC(M)} X/{\sim}.\] Here $\sim$ is the relation which identifies configurations which are equal after removing all points macroscopically located in $N$.

\end{definition}

See Definition 6.1 of \cite{Sa} for more detail. Intuitively, this is the configuration space of points in $M$ which vanish if they enter $N$. Using this, we can also define $\int_{(M,N)} P$ for a manifold relative to a subspace as the coequalizer of the two natural maps: \[ \mathbf{ fC(M,N)} Comp \rightrightarrows \mathbf{ fC(M,N)} P.\] 

\begin{remark}This construction was generalized in \cite{AFT}. From their perspective, the above definition uses the fact that all framed $E_n$-algebras in the category of spaces are canonically augmented. The space $\int_{(M,N)} P$ should be viewed as $\int_{(M,N)} (P,\mr{pt})$ with $\mr{pt}$ the trivial algebra. With extra conditions on $N$, one can define spaces $\int_{(M,N)} (P,Q)$ for more general algebraic objects $Q$. For example, if $N$ is the boundary of $M$, one can take the pair $(P,Q)$ to be an algebra over a variant of the swiss cheese operad of \cite{V}. \end{remark}

\subsection{Alternative models of topological chiral homology for partial algebras with $\pi_0 = \{1,\ldots,c\}$} \label{subsecalttch} In this section we will give a deformation retract $\int_{M,\leq c} P$ of $\int_M P$ for $P$ a partial $fF_n$-algebra satisfying $\pi_0 (P) = \{1,\ldots,c\}$. This is the subspace where each macroscopic location has charge $\leq c$. To define the charge of a macroscopic location, we note that there is a map $\int_M P \to \int_M \{1,\ldots,c\}$ and then for each macroscopic location of $x \in \int_M P$ we take the sum of the labels in $\{1,\ldots,c\}$ of the microscopic configurations associated to that macroscopic location.

To give the deformation retraction, we need to discuss the geometry of the Fulton-MacPherson compactification. See \cite{sinha} for an exposition of these topics, from which we borrow the notation.

To each point in $C_{FM}(M)(k)$ we can associate a rooted tree with leaves $\{1,\ldots,k\}$. Remember that a point in $C_{FM}(M)(k)$ is described by a collection of $m_i \in M$, $u_{ij} \in S^{m-1}$ and $d_{ijk} \in [0,\infty]$.  A rooted tree with leaves $\{1,\ldots,k\}$ can be given by telling you which leaves lie above which others: we say that the two leaves $i,j$ lie above $k$ if $d_{ijk} = 0$. This makes sense, because $d_{ijk}$ is zero if $i$ and $j$ are infinitely much closer than $i$ to $k$ (and hence $j$ to $k$). In particular, if the points $m_i$ are all disjoint, then we have $d_{ijk} = \frac{|m_i-m_j|}{|m_i-m_k|}$. There is a unique minimal tree with leaves labeled by the set $\{1,\ldots,k\}$ consistent with this data. We define the tree of $x$ to be this minimal tree, except when all $m_i$ are equal, in which case we add a new root vertex. See Figure \ref{figfmtree} for the tree associated to the element depicted in Figure \ref{figfultonmacpherson}.

\begin{figure}[t]
\begin{center}
\includegraphics[width=3.5cm]{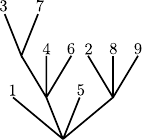}
\end{center}
\caption{The tree associated to the element of the Fulton-MacPherson compactification given in Figure \ref{figfultonmacpherson}, if one assumes that there are no particles outside of the depicted part.}
\label{figfmtree}
\end{figure}

It turns out that $C_{FM}(M)(k)$ is a manifold with corners, whose boundary strata $C_T$ are indexed by these trees, with the stratum corresponding to $T'$ in the closure of $T$ if $T$ can be obtained from $T'$ by contracting edges.

We can give a chart containing a tubular neighborhood of each stratum $C_T(M)$. To do this we define the space $D_T(M)$, which will turn out to be diffeomorphic to $C_T(M)$. If $e$ is an edge going to the root of $T$, we let $IC_e(M)$ be space given by a point $m \in M$ and the product over the vertices $v$ lying over $e$ of the quotients $\tilde{C}_{val(v)-1}(T_m M)/{\sim}$ of labeled configuration spaces by translation and dilation. We then define $D_T(M)$ as the subspace of $\prod_e IC_e(M)$ lying over $\tilde{C}_{\# e}(M)$. An element $y$ of $D_T(M)$ is given by a collection of $y^v_e$, where $v$ runs over the set $V$ of all vertices and $e$ over the set $E(v)$ of edges going into $v$. For the edges $e'$ going into the root vertex $r$ we have $y^r_{e'} \in M$, all of which have to be disjoint. For all edges $e'$ going into another vertex $v$ we have $y^v_{e'}$ in some tangent space to $M$, all of which are disjoint for fixed vertex. This makes precise our description of the Fulton-MacPherson compactification as a macroscopic configuration, labeled with microscopic configurations, etc; for $e'$ edges going into the root the $y^r_{e'}$ form the macroscopic configuration, and for each other edge $e'$ the $y^v_{e'}$ form a microscopic configuration.

As an example, a point in $D_T(M)$ for $M$ as in Figure \ref{figfmtree} consists of  (i) a labeled configuration of four points $m_1$, $m_{v_1}$, $m_{v_2}$, $m_5$ in $M$, (ii) a labeled configuration of three points $n_{v_3}$, $n_4$, $n_6$ in $T_{m_{v_1}}M$ up to translation and dilation, (iii) a labeled configuration of three points $n'_2$, $n'_8$, $n'_9$ in $T_{m_{v_2}}M$ up to translation and dilation, and (iv) a labeled configuration of two points $n''_3$ and $n''_7$ in $T_{n_{v_3}}T_{m_{v_1}}M$ up to translation and dilation.

Let us restrict our attention to the case $M = \R^n$ now, for ease of notation. Then we set $N_T$ to be subspace of $D_T(\R^n) \times [0,1)^{V_\mr{int}}$ (where $V_\mr{int}$ is the set of the internal vertices, i.e. vertices that are not equal to the root or leaves) of those $(y,(t_v))$ satisfying $t_v < r(y)$, with $r(y)$ is the function defined by \[\frac{r(y)}{1-r(y)} =  \frac{1}{3} \min_{\stackrel{v \in V }{e,e' \in e(V)}}\{d(m^v_e,m^v_{e'})\}.\] To each element $y$ and vertex $w$ we set the real number $s_w$ to be the product of all $t_v$ for $v$ in the path down to the root (the empty product is $1$).

We will now give a map $\nu_T: N_T(\R^n) \to C_{FM}(\R^n)(k)$, consisting of components $\eta_T: N_T(\R^n) \to (\R^n)^k$, $f_{ij}: N_T(\R^n) \to S^{n-1}$ and $g_{ijk}:  N_T(\R^n) \to [0,\infty]$. These are given by
\begin{enumerate}[(i)]
\item The map $\eta_T$ is most important to us. We first inductively define points $y_v(y) \in \R^n$ for each of the vertices by setting $y_{r}(y) = 0$ for the root $r$ and for all other vertices $v$ defining $y_v(y) = s_w m^w_e + y_w(y)$, where $e$ is the edge going down from $v$ and ending at some $w$. Then $\eta_T(y)$ is given by the $y_i(x)$'s for $i \in \{1,\ldots,k\}$ the leaves. Intuitively this is a ``nested solar system''  construction.
\item For each pair $i,j$, let $T_w$ be the tree obtained as the join of $i$ and $j$ in $T$ (this is the smallest tree containing $i$ and $j$) which has some root $w$. We then set $f_{ij} = \pi_{ij} \circ \eta_{T_w} \circ \rho_w$, where $\rho_w: N_T(\R^n) \to N_{T_w}(\R^n)$ is the map that projects onto the vertices in $T_w$, sets $t_w = 1$ and includes $\tilde{C}_{\mr{val}(w)-1}(\R^n)/\sim$ into $\tilde{C}_{\mr{val}(w)-1}(\R^n)$ as the configurations with average location $0$ and radius $1$.
\item For each triple $i,j,k$ we set $g_{ijk}$ to be $\pi_{ijk} \circ \eta_{T_w} \circ \rho_w$, where $T_w$ is now the join of the vertices labeled by $i,j,k$.
\end{enumerate} 

\begin{theorem}[Sinha] The maps $\nu_T: N_T(\R^n) \to C_{FM}(\R^n)(k)$ are charts covering the manifold with corners $C_{FM}(\R^n)(k)$. These charts are compatible with the $\mathfrak S_k$-action and $D_T(\R^n) \times \{0\}$ gets mapped diffeomorphically onto the stratum $C_T(\R^n)$.\end{theorem}

We are not interested in $C_{FM}(\R^n)(k)$ by itself, but the space $C_{FM}(\R^n)(k) \times_{\mathfrak S_k} P^k$ obtained by labeling each leaf with a point of $P$ and its eventual quotient $\int_M^k P$. Note the construction of the tree with leaves $\{1,\ldots,k\}$ associated to an element of $C_{FM}(\R^n)(k)$ can be adapted to this setting. It will associate to each element of $C_{FM}(\R^n)(k) \times_{\mathfrak S_k} P^k$ an equivalence class of trees with leaves labeled by elements $\pi_0 (P) = \{1,\ldots,c\}$ up to the diagonal action by $\mathfrak S_k$.

We are interested in the subspace of $C_{FM}(\R^n)(k) \times_{\mathfrak S_k} P^k$ such that no macroscopic location has charge $>c$, i.e. the sum of the labels of $\pi_0 (P)$ attached to the leaves over each root edge is at most $c$. Let $\left(C_{FM}(\R^n)(k) \times_{\mathfrak S_k} P^k\right)_{\leq c}$ denote this subspace. 

\begin{proposition}\label{prophomotopyfmcleqc} There is a homotopy $H_t$ of maps 
\[C_{FM}(\R^n)(k) \times_{\mathfrak S_k} P^k \to C_{FM}(\R^n)(k) \times_{\mathfrak S_k} P^k\] 
with the following properties
\begin{itemize}
\item $H_0 = \mr{id}$, 
\item $H_1$  has image in $\left(C_{FM}(\R^n)(k) \times_{\mathfrak S_k} P^k\right)_{\leq c}$,
\item $H_t(\left(C_{FM}(\R^n)(k) \times_{\mathfrak S_k} P^k\right)_{\leq c}) \subset \left(C_{FM}(\R^n)(k) \times_{\mathfrak S_k} P^k\right)_{\leq c}$ for all $t \in [0,1]$,
\item $H_t$ keeps unchanged the microscopic configurations of charge $\leq c$.
\end{itemize} \end{proposition}

\begin{proof}We will consider the space $C_{FM}(\R^n)(k) \times P^k$ instead of its $\mathfrak S_k$-quotient $C_{FM}(\R^n)(k) \times_{\mathfrak S_k} P^k$, and work equivariantly. 

We denote the component of $P$ corresponding to $i \in \{1,\ldots,c\}$ by $P_i$ and note that $C_{FM}(\R^n)(k) \times P^k$ is a disjoint union over all maps $\kappa: \{1,\ldots,k\} \to \{1,\ldots,c\}$ of spaces $C_{FM}(\R^n)(k) \times P_{\kappa(1)} \times \ldots \times P_{\kappa(k)}$. We will construct for each $\kappa$ a homotopy $H^\kappa_t$ of self maps of $C_{FM}(\R^n)(k) \times P_{\kappa(1)} \times \ldots \times P_{\kappa(k)}$ having the desired properties. These homotopies will be $\mathfrak S_k$-equivariant, in the sense that for $\sigma \in \mathfrak S_k$ we have that $H^{\kappa \circ \sigma}_t = H^\kappa_t \circ \hat{\sigma}$, where $\hat{\sigma}$ is the diagonal action of $\mathfrak S_k$ on $C_{FM}(\R^n)(k) \times P^k$. Combining these homotopies for all $\kappa$ gives a $\mathfrak S_k$-equivariant deformation retraction, which induces $H_t$ upon taking the quotient.

As for any manifold with corners, for each $T$ there is a neighborhood of $U_T$ of $D_T(M)$ in $N_T(M)$ such that $U_T$ is disjoint from $U_{T'}$ if $T'$ corresponds to another stratum of the same codimension. By taking the intersection of the $\mathfrak S_k$-translates, we can assume that the neighborhoods $U_T$ are $\mathfrak S_k$-equivariant for the action of $\mathfrak S_k$ on $C_{FM}(\R^n)(k)$, in the sense that $\sigma(U_T) = U_{\sigma(T)}$, where $\sigma(T)$ is equal to $T$ with $\sigma$ applied to its leaves.

For each codimension $K \geq 0$, it suffices to construct a homotopy $H^{\kappa,K}_t$ that makes the ``required changes'' on the subspace of $C_{FM}(\R^n)(k) \times P_{\kappa(1)} \times \ldots \times P_{\kappa(k)}$ corresponding to codimension $K$ strata of $C_{FM}(\R^n)(k)$. To do this we construct $H^{\kappa,T}_t$ for each chart $\nu_T: N_T(\R^n) \to C_{FM}(\R^n)(k)$ of codimension $K$ and each $\kappa$. They will have the following properties: \begin{enumerate}[(i)]
\item each of these is supported in $U_T \times P_{\kappa(1)} \times \ldots \times P_{\kappa(k)}$,
\item removes macroscopic locations in the stratum with charge $>c$,
\item  keeps unchanged microscopic configuration of charge $\leq c$, 
\item is $\mathfrak S_k$-equivariant
\end{enumerate} Then $H^{\kappa}_t$ will be the composition of $H^{\kappa,K}_t$ for all $K$, starting with $K=0$ (though in that case the homotopy will be constant). 

For each of the charts, our homotopy $H^{\kappa,T}_t$ on  $U_T \times P_{\kappa(1)} \times \ldots \times P_{\kappa(k)}$ will have as essential property that the $t_v$ of all non-root vertices having charge at least $c+1$ above them becomes non-zero, while keeping the other $t_{v'}$ fixed. If by induction the composition of the homotopies $H^{\kappa',T'}_t$ for $T'$ of codimension $K'<K$ strata give us a homotopy that makes (ii) and (iii) hold on the complement of the codimension $\geq K$ strata, then this will guarantee that additionally composing with $H^{\kappa,T}_t$ will give a homotopy satisfying properties (i) and (ii) in the union of the complement of the codimension $\geq K$ strata with the neighborhood $U_T$ of the codimension $K$ stratum corresponding to $T$. 

To construct this homotopy $H^{\kappa,T}$ we first pick continuous functions $\upsilon_T: N_T(\R^n) \to [0,1]$ which (i') are non-zero on the subspace of $N_T(\R^n)$ with at least one $t_v$ equal to zero, (ii') go to zero as one approaches $\partial U_T$ and (iii') are equivariant for the $\mathfrak S_k$-action on $C_{FM}(\R^n)(k)$. This can be done by picking a continuous function for each $T$ having the properties (i) and (ii) and averaging under the $\mathfrak S_k$-action. Our map $H^{\kappa,T}_t$ will be constant on the $y \in D_T(\R^n)$ and $(p_1,\ldots,p_k) \in P_{\kappa(1)} \times \cdots \times P_{\kappa(k)}$, but will change some of the $t_v$. In particular, it does not change $t_v$ unless $v$ is a non-root vertex with charge strictly bigger than $c$ above it. In that case $H^{\kappa,T}_t$ will send $t_v$ to $t_v + t \upsilon_T(y,(t_v))$. It is easy to check that this has the desired properties.\end{proof}

Recall that $\int_{M,\leq c}P \subset \int_M P$ is the subspace consisting of those elements where each macroscopic location has charge $\leq c$.

\begin{corollary}\label{corleqcweq} If $\pi_0 (P) = \{1,\ldots,c\}$, the following inclusion is a homotopy equivalence 
\[\int_{M,\leq c}P \hookrightarrow \int_M P\]
\label{twoModelsOfTCH}
\end{corollary}

\begin{proof}Note that the homotopy of Proposition \ref{prophomotopyfmcleqc} can be extended to $fC_{FM}(\R^n)(k)$ by moving the framings along the homotopy. Using  Proposition \ref{prophomotopyfmcleqc} we will construct a homotopy from the identity map of $fC_{FM}(M)(k) \times_{\mathfrak S_k} P^k$  to a map landing in $\left(fC_{FM}(M)(k) \times_{\mathfrak S_k} P^k\right)_{\leq c}$ which is homotopic to the identity map of $\left(fC_{FM}(M)(k) \times_{\mathfrak S_k} P^k\right)_{\leq c} $ and leaves the points with charge $\leq c$ unchanged
for arbitrary $M$. To do this, pick an open cover $U_i$ by charts and compactly supported functions $\eta_i: U_i \to [0,1]$ such that at least one of these functions is equal to $1$ for each $m \in M$. Now we can take the homotopy to be the composition of $H^{(i)}_{t \eta_i}$ for all $i$.

To prove the Corollary, we must check that this homotopy is compatible with the equivalence relation imposed by taking the coequalizer along the two maps $\mathbf{fC(M)}Comp \rightrightarrows \mathbf{fC(M)}P$. However, since $\pi_0 (P) = \{1,\ldots,c\}$, this follows because identifications happen along microscopic configurations of charge $\leq c$ and these are unchanged.\end{proof}

This means that we can use $\int_{M,\leq c} P$ as our model for topological chiral homology of partial algebras with $\pi_0 (P) = \{1,\ldots,c\}$ and we will always want to do so.

\begin{example} \label{examplecompletionEn} Note that when $A$ is a strictly commutative abelian monoid, $\int_M A$ is homeomorphic to the quotient of $\mr{Sym}(M \times A)$ by the relation that whenever several points are at the same location, we can compose their labels. In other words, the relation is generated by: \[ (m_1,m_2,\ldots;a_1,a_2,\ldots) \sim (m_2,\ldots;a_1+a_2,\ldots) \] whenever $m_1=m_2$. Thus for example we have that $\mr{Sym}(M) \cong \int_M  \mathbb N$.

If $P$ is a partial abelian monoid with $\pi_0(P)=\{1,\ldots,c \}$, then $\int_{M,\leq c} P$ is homeomorphic to a quotient of the subspace of $\mr{Sym}(M \times P)$ consisting of configurations of points such that whenever several points are at the same location, their labels are composable. The equivalence relation is as before. For the purposes of this paper, the most important example of a partial framed $E_n$-algebra is $\{1,\ldots,c\}$. Using Corollary \ref{twoModelsOfTCH}, we have that $\int_M \{1,\ldots,c\} \simeq \mr{Sym}_k^{\leq c}(M)$. Using Proposition 5.16 of \cite{Sa}, we see that $\overline{ \{1,\ldots,c\}} \simeq \mr{Sym}_k^{\leq c}(\R^n)$. Note that $ \{1,\ldots,c\}$ is a partial framed $E_n$-algebra for all $n$ and the homotopy type of the completion depends on which $n$ we choose. 
\end{example}

See  Corollary 5.21 of \cite{Sa} for other examples of completions of partial $E_n$-algebras and applications to spaces of holomorphic maps.

\begin{remark}
Corollary \ref{twoModelsOfTCH} is not the most general statement on could prove. Let $\int_{M,Comp} P$ denote the subspace of $\int_{M} P$ such that each macroscopic location is occupied by something composable. We believe that if $Comp \m fF_n P$ is a cofibration and $P$ is well based, then $\int_{M,Comp} P \m \int_M P$ is a homotopy equivalence. 
 \end{remark}

\subsection{The scanning map and approximation theorems} \label{scanningInLim}

In this subsection we define the scanning map and recall approximation theorems from \cite{Sa} and \cite{Mi2}. Let $M$ be a smooth $n$-dimensional manifold and $P$ a partial $fF_n$-algebra. Let $D^n$ denote the standard unit disk in $\R^n$.

\begin{definition}
Let $B^n P = \int_{(D^n,S^{n-1})} P$.
\end{definition}

The following is Theorem 7.3 of \cite{Sa}.

\begin{theorem}[Salvatore] \label{thmsalvgroupcompl} $B^n P$ is an $n$-fold delooping of $P$, in the sense that $\Omega^n B^n P$ is the group completion of $\bar P$.\end{theorem}

Thus if $P$ is a $fF_n$-algebra with $\pi_0(P)$ a group, then $P \simeq \Omega^n B^n P$. We call this condition ``group-like.'' We note that there are many other models of the $n$-fold delooping of a framed $E_n$-algebra. See \cite{M} for a definition using the monadic bar construction and \cite{AF} for a definition using factorization cohomology. 

\begin{definition}
\label{BTM}
For $M$ a smooth $n$-dimensional manifold and $P$ a partial $fF_n$-algebra, $B^{TM} P$ is a fiber bundle over $M$ with fibers homeomorphic to $\int_{(D^n,S^{n-1})} P$. These fibers are glued together by viewing the fiber over $m \in M$ as $\int_{(D_m,S_m)} P$. Here $D_m$ and $S_m$ are respectively the fibers of the unit disk and sphere bundle with respect to some choice of Riemannian metric on $M$.
\end{definition}

See Page 393 of \cite{Sa} (Page 20 of the arXiv version) for more details. One can think of this as the parametrized topological chiral homology over $M$ of the relative manifold bundle $(D(M),S(M))$ over $M$ given by the disk and sphere bundles. Before we define the scanning map, we fix some notation for spaces of sections of a bundle.

\begin{definition}
Let $N \subset M$ and $E \to M$ be a bundle with fiberwise base point. Let $\Gamma_{(M,N)}(E)$ denote the space of sections $\sigma$ of the restriction of $E$ to $M\backslash N$ such that the closure of the support of $\sigma$ in $M$ is compact. Topologize this space with subspace topology as a subspace of the space $\Gamma_{M \backslash N}(E)$ of all sections of $E$ restricted to $M \backslash N$ in the compact open topology. When $N$ is empty we denote this simply by $\Gamma_M(E)$.
\end{definition}

Observe that $\Gamma_M(E)$ is the space of compactly supported sections of $E$. For $N \subset M$, there is a natural map $\pi: \int_M P \to \int_{(M,N)} P$. Using the exponential function with respect to a Riemannian metric on $M\backslash N$ that at each point has injectivity radius at least 1, we can continuously pick diffeomorphisms $\phi_m:B_1(m) \to D_m$. We can now define the scanning map. 

\begin{definition}
Let $s:\int_{(M,N)} P \to \Gamma_{(M,N)}(B^{TM} P)$ be the map which sends $\xi \in \int_{(M,N)} P$ to the section with value $\phi_{m*}(\pi(\xi))$ over $m \in M-N$. Here \[ \pi: \int_{(M,N)} P \to \int_{(M,M-B_1(m))} P=\int_{(B_1(m),\partial B_1(m))} P \] is the natural projection and \[ \phi_{m*}:\int_{(B_1(m),\partial B_1(m))} P \to \int_{(D_m,S_m)} P \] is the map induced by the diffeomorphism $\phi_m$.
\label{defscan}
\end{definition}

The following theorem was proven in Salvatore in \cite{Sa} and is known as non-abelian Poincar\'e duality. 

\begin{theorem}[Salvatore]
If $P$ is a partial $fF_n$-algebra with $\bar P$ group-like, then 
\[s:\int_M P \to \Gamma_{M}(B^{TM} P)\] is a homotopy equivalence.
\end{theorem}

A version for pairs $(M,N)$ also appears in Theorem 7.6 of \cite{Sa}. We state a slight generalization of this, which is Theorem 3.14 of \cite{Mi2}.

\begin{theorem}[Salvatore, Miller] \label{pair}
Let $P$ be a partial $fF_n$-algebra and assume $\pi_0(N) \to \pi_0(M)$ is onto, then \[s:\int_{(M,N)} P \to \Gamma_{(M,N)}(B^{TM} P)\] is a homotopy equivalence. 
\end{theorem}

When $\bar P$ is not group-like and $N$ is empty, $s$ is not a homotopy equivalence. However if $M$ admits boundary, there is a procedure for completing $\int_M P$ into a space homology equivalent to $\Gamma_{M}(B^{TM} P)$. Recall that a manifold admits boundary if it is the interior of a manifold with non-empty boundary. Note that we do not require the manifold with boundary to be compact. Let $\bar M$ be a manifold with non-empty boundary $\partial \bar{M}$ that has interior $M$. Let $\chi: (0,1)^{n-1} \to \partial \bar{M}$ be an embedding and let \[\tilde{\chi}: \mr{int}(\bar M \cup_{\chi} (0,1)^{n-1} \times [0,1)) \to M \] be a diffeomorphism whose inverse is isotopic to the standard inclusion. Fix a point $m \in (0,1)^{n-1} \times (0,1)$.

\begin{definition}\label{defstabilizationmap}
For $M$, $\chi$, $\tilde{\chi}$, $m$ as above and $p \in P$, let the stabilization map \[t_p : \int_M P \to \int_M P\]
be the composition of the following two maps. Adding a point at $m$ labeled by $p$ gives a map $ \int_M P \to \int_{\mr{int}(\bar M \cup_{\partial M} \partial M \times [0,1))} P$ and the map $\tilde{\chi}$ induces a map $  \int_{\mr{int}(\bar M \cup_{\partial M} \partial M \times [0,1))} P \to \int_M P$.
\end{definition}

We call this map the stabilization map for topological chiral homology. Up to homotopy, it only depends on a choice of end of $M$ and a choice of element of $\pi_0(\bar P)$.

We can define a similar map for the space of sections. Let $\sigma_p \in \Gamma_{M}(B^{TM} P)$ be the image of empty configuration under the composition of the scanning map and the stabilization map $t_p$. Note that $\sigma_p$ is supported on $\tilde{\chi}((0,1)^{n-1} \times [0,1))$. 

\begin{definition}\label{defstabilizationmapsections}
For $M$, $\chi$, $m$ as above and $p \in P$, let the stabilization map 
\[T_p : \Gamma_{M}(B^{TM} P) \to \Gamma_{M}(B^{TM} P)\]
be the map given by the following formula: \[T_p(\sigma) = \begin{cases}
\sigma(\tilde{\chi}^{-1}(m)) & \text{ if $m \in \tilde{\chi}(M)$} \\
\sigma_p(m) & \text{ if } m \in \tilde{\chi}( (0,1)^{n-1} \times [0,1)). 
\end{cases} \] 
\end{definition}

Let $p_i$ be a sequence of elements of $P$ such that every element of $\pi_0(P)$ appears infinitely many times. Let $\mr{hocolim}_{t_p} \int_M P$ denote the mapping telescope associated to the diagram of spaces with $i$th space given by $ \int_M P$ and the map from the $i$th space to the $(i+1)$st space given by $t_{p_i}$. This can be thought of as $ \int_M P$ with the stabilization map inverted. Similarly define $\mr{hocolim}_{T_p}  \Gamma_{M}(B^{TM} P)$ by replacing $ \int_M P$ by $\Gamma_{M}(B^{TM} P)$ and $t_{p_i}$ by $T_{p_i}$. In \cite{Mi2}, the second author proved the following theorem. 

\begin{theorem}[Miller] \label{thmscanning}
If $M$ admits boundary, the scanning map induces a homology equivalence between 
\[\mr{hocolim}_{t_p} \left(\int_M P\right) \cong_{H_*} \mr{hocolim}_{T_p} \left(\Gamma_{M}(B^{TM} P)\right)\]
\end{theorem}

Since the maps $T_p$ are homotopy equivalences, there is a homotopy equivalence between $\Gamma_{M}(B^{TM} P)$ and the homotopy colimit $\mr{hocolim}_{T_p}  \Gamma_{M}(B^{TM} P)$.

\section{Highly connected bounded charge complexes}\label{seccomplexes} In this section we prove several technical results on the connectivity of various simplicial complexes, semisimplicial sets and semisimplicial spaces. The definitions of these complexes and the statements of the results on their connectivity are given in Subsection \ref{subsecstatres}.

\subsection{Statement of results}\label{subsecstatres} In this subsection we freely use the language of semisimplicial spaces and simplicial complexes, whose definitions we recall in Subsection \ref{subsecsimplicialmethods}. To resolve the topological chiral homology of completions, we will need to construct complexes that allow us to drag away a non-zero but not too large amount of charge to the boundary. This is done by picking collections of disjoint embedded disks connected to the boundary containing a non-zero but bounded amount of charge. 

The techniques and results are slightly different in dimension $\geq 3$ and $2$, in the following ways:
\begin{itemize}
	\item In dimension $\geq 3$, we start from a complex of injective words, while in dimension $2$ we start with an arc complex.
	\item In dimension $\geq 3$, we only need transversality, since disks connected to the boundary behave like arcs connecting their center to the boundary, while the argument in dimension $2$ is more involved.
	\item In dimension $\geq 3$, it turns out to be convenient not to topologize the embedded disks, so that one can remove one step from the argument. On other hand, in dimension $2$, it is easiest to topologize the embedded disks, so that the complex has a more direct relationship to the isotopy classes of arcs in the arc complex.
	\item In dimension $\geq 3$, all the embedded disks will be disjoint, while in dimension $2$ they are also allowed to be nested. This leads to the complex in dimension $2$ having higher connectivity, which unfortunately does not translate into a better homological stability range.
\end{itemize}

\subsubsection{Dimension $\geq 3$} We start by defining the required complex for manifolds of dimension $\geq 3$. We fix a connected manifold $M$ of dimension $n \geq 3$ that is the interior of a manifold $\bar{M}$ with boundary $\partial \bar{M}$. We also fix an embedding $\psi: (0,1)^{n-1} \to \partial \bar{M}$ disjoint from but isotopic to the  embedding $\chi$ used in Definition \ref{defstabilizationmap} of the stabilization map. We also pick a collar $\tilde{\psi}: (0,1)^{n-1} \times [0,1) \to \bar{M}$ extending $\psi$. The image of $\psi$ will be used to connect our bounded charge disks to the boundary and $\tilde{\psi}$ will be used to prescribe a standard form for embeddings. We cannot use the entire boundary as that would leave no portion of the boundary to use to define the stabilization map.

\begin{definition}Let $D^n_+$ be the upper half disk $\{(x_1,\ldots,x_n) \in \R^n\,|\,||x|| \leq 1, x_1 \geq 0\}$ and $D^{n-1} \subset D^n_+$ be the subset $\{(x_1,\ldots,x_n) \in \R^n\,|\,||x|| \leq 1, x_1 = 0\}$. A \emph{standard embedding} $\varphi: (D^n_+,D^{n-1}) \to (\bar{M},\im(\psi))$ is an embedding of manifolds with corners such that near the $D^{n-1}$ the map $\tilde{\psi}^{-1} \circ \varphi$ is equal the standard embedding of an $\epsilon$-ball around $(t,\frac{1}{2},\ldots,\frac{1}{2})$ for some $t \in (0,1)$ and $0<\epsilon<\min(\frac{1}{2},t,1-t)$.\end{definition}

Note that the real numbers $t$ and $\epsilon$ are uniquely determined by $\varphi$.  We fix $x \in \mr{Sym}_k^{\leq c}(M)$, which can be thought of as a collection of points in $M$ labeled by elements of $\{1,\ldots,c\}$ such that the sum of these labels is $k$. The charge contained in a subset of $M$ is the number of elements of $x$ contained in it, counted with multiplicity.

\begin{definition}A standard embedding $\varphi: (D^n_+,D^{n-1}) \to (\bar{M},\im(\psi))$ is said to be \emph{of bounded charge} if $\varphi(\partial D^n_+) \cap x = \emptyset$ and the charge contained in $\varphi(\mr{int}(D^n_+))$ is between $1$ and $c$. \end{definition}

\begin{definition} We define $K_\bullet^{\delta,c}(x)$ to be the semisimplicial set with $0$-simplices bounded charge embeddings $\varphi: (D^n_+,D^{n-1}) \to (\bar{M},\im(\psi))$. The $p$-simplices are $(p+1)$-tuples $(\varphi_0,\ldots,\varphi_p)$ of disjoint bounded charge embeddings satisfying $0< t_0<t_1<\ldots<t_p<1$, where the $t_i$'s are determined by the image of $D^{n-1}$ under $\varphi_i$ (see the left part of Figure \ref{figsimplices}). The face map $d_i$ forgets the $i$th bounded charge embedding $\varphi_i$.\end{definition}

\begin{theorem}\label{thmconndiskdim3} The semisimplicial set $K_\bullet^{\delta,c}(x)$ is $(\lfloor \frac{k}{c} \rfloor-2)$-connected.\end{theorem}

\begin{remark}\label{remdimthreetop} One can also define a topologized version of the semisimplicial set $K_\bullet^{\delta,c}(x)$, by taking the $C^\infty$-topology on the spaces of embeddings involved. The result is a semisimplicial space $K_\bullet^{c}(x)$. The techniques used to prove Theorem 4.6 of \cite{sorenoscarstability} (see also Appendix A of \cite{kuperscircles}) prove that $K_\bullet^c(x)$ has the same connectivity as $K_\bullet^{c,\delta}(x)$. This will not be used in this paper.\end{remark}

\subsubsection{Dimension 2} Next we define the required semisimplicial space for surfaces. We fix a surface $\Sigma$ with boundary $\partial \Sigma$, choose an embedding of $\psi$ of $(0,1)$ into the boundary and a finite set of points $V$ in $(0,1)$.  We also pick a collar $\tilde{\psi}: (0,1) \times [0,1) \to \Sigma$ extending $\psi$.

Let $D^2_+$ and $D^{1}$ be as before, and denote $\partial D^1$ by $D^0_R \sqcup D^0_L$, which also equals the two corner points of $D^2_+$. We also define $D^2_\vee$ to be the manifold with corners which topologically is a disk but has one corner which we denote by $D^0$. More precisely it is given by the subset of $[0,1]^2$ given by union of the disk of radius $\frac{1}{2}$ around $(\frac{1}{2},\frac{1}{2})$ and the triangle with vertices $(\frac{1}{2},0)$, $(\frac{1}{2}+\frac{\sqrt{2}}{4},\frac{1}{2}-\frac{\sqrt{2}}{4})$ and $(\frac{1}{2}-\frac{\sqrt{2}}{4},\frac{1}{2}-\frac{\sqrt{2}}{4})$. 

\begin{definition} A \emph{standard embedding $\varphi: (D^2_+,D^{1}) \to (\Sigma,\im(\psi))$ relative to $V$} is an embedding of manifolds with corners such that
\begin{enumerate}[(i)]
\item $\varphi^{-1}(\im(\psi)) = D^1$, 
\item $\varphi^{-1}(V) = D^0_R \sqcup D^0_L$,
\item near $D^1$ the map $\tilde{\psi} \circ \varphi$ is a composition of rescaling and horizontal translation. 
\end{enumerate} 

A \emph{standard embedding $\varphi: (D^2_\vee,D^0) \to (\Sigma,\im(\psi))$ relative to $V$} is an embedding of manifolds with corners such that
\begin{enumerate}[(i)]
\item $\varphi^{-1}(\im(\psi)) = \varphi^{-1}(V) = D^0$,
\item near $D^1$ the map $\tilde{\psi} \circ \varphi$ is a composition of rescaling and horizontal translation. 
\end{enumerate} \end{definition}

In dimension $2$, we say that a collection of standard embeddings is disjoint if the intersections of their boundaries lie in $\im(\psi)$ and the boundaries have distinct germs at corners. Thus they are allowed to be \emph{nested}. Any disjoint $(p+1)$-tuple of such bounded charge embedding is canonically ordered: they are lexicographically ordered by the point of $V$ that $D^0_L$ or $D^0$ hits and then by the ordering of the germs at each point of $V$.

\begin{definition}\label{defsspacedim2} The semisimplicial space $K^c_\bullet(x)$ has as space of $p$-simplices the space of $(p+1)$-tuples of disjoint bounded charge embeddings $(D^2_+,D^1) \to (\Sigma,\im(\psi))$ relative to $V$ or $(D^2_\vee,D^0) \to (\Sigma,\im(\psi))$ relative to $V$, ordered according to the prescription given above (see the right part of Figure \ref{figsimplices}). The face map $d_i$ forgets the $i$th bounded charge embedding.\end{definition}

We will prove the following theorem using the connectivity of arc complexes, the fact that the space of representatives of an isotopy class is contractible, and Smale's theorem from \cite{smaledisk}.

\begin{figure}[t]
\begin{center}
\includegraphics[width=11.5cm]{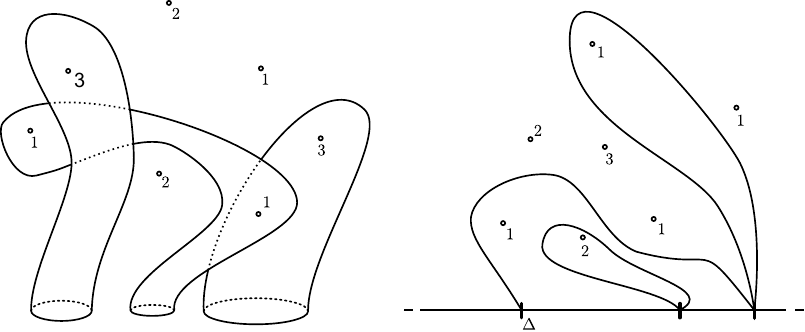}
\end{center}
\caption{On the left is an example of a 2-simplex of $K^{\delta,c}(x)$ for some 3-manifold and on the right is an example of a 2-simplex for $K^{c}(x)$ for some surface with boundary.}
\label{figsimplices}
\end{figure}

\begin{theorem}\label{thmconndiskdim2} The semisimplicial space $K_\bullet^c(x)$ is $(k-2)$-connected.\end{theorem}

\begin{remark} Note that the definition of $K_\bullet^c(x)$ in dimension 2 does not coincide with that of $K_\bullet^c(x)$ given in Remark \ref{remdimthreetop} for dimension $\geq 3$. The differences are that in dimension $2$ disks can intersect the boundary in a single point and are allowed to be nested.\end{remark}

\subsection{Background on connectivity and simplicial methods} \label{subsecsimplicialmethods} For the convenience of the reader we review simplicial methods. We start by giving the definitions of the objects involved:

\begin{enumerate}[(i)]
\item A \emph{semisimplicial set} is a functor $X_\bullet: \Delta^\mr{op}_\mr{inj} \to \cat{Set}$, where $\Delta_\mr{inj}$ is the category with objects $[p] = \{0,1,\ldots,p\}$ for $p \geq 0$ and morphisms the injective order-preserving maps. More concretely, it is a collection of sets $X_p$ of $p$-simplices for $p \geq 0$ together with face maps $d_i: X_p \to X_{p-1}$ for $0 \leq i \leq p$ satisfying the standard identities.
\item A \emph{semisimplicial space} is a functor $X'_\bullet: \Delta^\mr{op}_\mr{inj} \to \cat{Top}$, which can similarly be described as a sequence of spaces with face maps between them.
\item A \emph{simplicial complex} $Y_\circ$ is a set $Y_0$ of 0-simplices (or \emph{vertices}) together with collections $Y_p$ for $(p+1)$-element subsets of $Y_0$ called $p$-simplices, such that all $p$-elements subsets of an $p$-simplex are $(p-1)$-simplices.\end{enumerate} 

We will denote semisimplicial sets or spaces by $(-)_\bullet$ and simplicial complexes by $(-)_\circ$. The main difference between semisimplicial sets and simplicial complexes is that in the former the faces are ordered, while in the latter they are not. 

\subsubsection{Links, stars and joins} The most important of advantage of simplicial complexes is that they admit various useful constructions, in particular links and stars.

\begin{definition}Let $\sigma = \{y_0,\ldots,y_k\}$ be a $k$-simplex of $Y_\circ$.
\begin{enumerate}[(i)]
\item The \emph{link} $\mr{Link}(\sigma) \subset Y_\circ$ is the simplicial subcomplex with $m$-simplices consisting of all sets $\tau = \{y_{k+1},\ldots,y_{k+m+2}\}$ such that $\{y_0,\ldots,y_{k+m+2}\}$ is a $(k+m+1)$-simplex of $Y_\circ$.
\item The \emph{star} $\mr{Star}(\sigma) \subset Y_\circ$ is the simplicial subcomplex with $m$-simplices consisting of all sets $\tau = \{y'_0,\ldots,y'_{m}\}$ such there is a $\tilde{\tau}$ with $\tau$ and $\sigma$ both faces of $\tilde{\tau}$.
\end{enumerate}
\end{definition}

A simplicial complex is a triangulation of a PL-manifold with boundary if and only if the links and stars are spheres and disks respectively. We will use the following consequence of this.

\begin{lemma}Let $Y_\circ$ be a PL-triangulation of an $n$-dimensional manifold $M$ with boundary $\partial M$. If $\sigma$ is a $p$-simplex in the interior of $M$, then $\mr{Link}(\sigma) \cong S^{n-p-1}$ for some PL-triangulation of $S^{n-p-1}$ and similarly $\mr{Star}(\sigma) \cong D^n$.\end{lemma}

In our arguments on the connectivity of complexes we will use the join construction. The join $X * Y$ of two spaces is defined by $(X \times [0,1] \times Y)/\sim$ where $\sim$ is the equivalence relation identifying $(x,0,y) \sim (x,0,y')$ for all $y,y' \in Y$ and $(x,1,y) \sim (x',1,y)$ for all $x,x' \in X$. We will however need it for simplicial complexes.

\begin{definition}Let $Y_\circ$ and $Y'_\circ$ be simplicial complexes. The join $(Y * Y')_\circ$ has as $0$-simplices the set $Y_0 \sqcup Y'_0$. A $(k+1)$-element subset of $Y_0 \sqcup Y'_0$ forms a $p$-simplex if  it is either (i) a $p$-simplex of $Y_\circ$, (ii) a $p$-simplex of $Y'_\circ$, or (iii) a disjoint union of a $p_1$-simplex of $Y_\circ$ and a $p_2$-simplex of $Y'_\circ$ for $p_1+p_2 = p-1$. \end{definition}

The join relates the link and the star of a simplex. The next lemma follows directly form the definitions:

\begin{lemma}We have that $\mr{Star}(\sigma) \cong \mr{Link}(\sigma) * \sigma$.\end{lemma}

\subsubsection{Connectivity} We are interested in the way these various definitions interact with the notion of connectivity of a space. We start by giving several equivalent definitions of the connectivity of spaces and maps.

\begin{definition}\begin{enumerate}[(i)] \item A space is always \emph{$(-2)$-connected} and \emph{$(-1)$-connected} if and only if it is non-empty.
\item For $k \geq 0$, a space $X$ is \emph{$k$-connected} if it is non-empty, connected and $\pi_i(X,x_0) = 0$ for $1 \leq i \leq k$ and all base points $x_0 \in X$. This is equivalent to the existence of a dotted lift for $-1 \leq i \leq k$ (where $S^{-1} = \emptyset$) in diagrams
\[\xymatrix{S^i \ar[r] \ar[d] & X \\
D^{i+1} \ar@{.>}[ru] & }\]
making the triangle commute.
\item A map $f: X \to Y$ is \emph{$k$-connected} if it is a bijection on $\pi_0$ and for each base point $x_0 \in X$ the homotopy fiber $\mr{hofib}(f,x_0)$ is $(k-1)$-connected. This is equivalent to $\pi_i(f)$ being an isomorphism for $0 \leq i \leq k-1$ and a surjection for $i = k$. A third characterization of $k$-connectivity is the existence of dotted lifts  for $-1 \leq i \leq k-1$ in diagrams
\[\xymatrix{S^i \ar[r] \ar[d] & X \ar[d] \\
D^{i+1} \ar@{.>}[ru] \ar[r] & Y. }\]
making the top triangle commute and the bottom triangle commute up to homotopy rel $S^i$.
\end{enumerate}
\end{definition}

To talk about connectivity of semisimplicial sets, semisimplicial spaces or simplicial complexes, we need to convert them into spaces by \emph{geometric realization}. Let $X_\bullet$, $X'_\bullet$ and $Y_\circ$ a semisimplicial set, semisimplicial space or simplicial complex respectively, then their geometric realizations are defined as
\[|X_\bullet| := \left(\bigsqcup_p X_p \times \Delta^p\right)/{\sim} \qquad 
||X'_\bullet|| := \left(\bigsqcup_p X'_p \times \Delta^p \right)/{\sim} \qquad
|Y_\circ| := ||Y^{<}_\bullet||\]
where $\sim$ is the equivalence relation generated by $(x,d^i t) \sim (d_i x,t)$ and $Y^{<}_\bullet$ is the semisimplicial set obtained from $Y_\circ$ by picking a total order $<$ on the vertices, setting the $p$-simplices to be the $p$-simplices of $Y_\circ$ and letting $d_i$ delete the $i$th vertex in the order $<$. The space $|Y_\circ|$ is up to homeomorphism independent of the choice of total order.

We next discuss methods to prove that a simplicial object is highly-connected. Firstly, joins in general increase connectivity. The following well-known proposition quantifies this (e.g. Proposition 6.1 of \cite{wahlmcg}).

\begin{proposition}If $X_i$ is $k_i$-connected, then $X_1 * \ldots * X_k$ is $(\sum_i (k_i+2) - 2)$-connected.\end{proposition}

Let us next remark on the relationship between semisimplicial sets and simplicial complexes. Associated to a simplicial complex $Y_\circ$ there is a semisimplicial set $Y^\mr{or}_\bullet$ which has as $p$-simplices the $p$-simplices of $Y_\circ$ together with an ordering of their $p+1$ elements. The connectivity of a simplicial complex and its associated ordered semisimplicial set are closely related.

\begin{lemma}\label{lemcompsimps} We have that $|Y_\circ|$ is as highly connected as $||Y^\mr{or}_\bullet||$.
\end{lemma}

\begin{proof}Forgetting the ordering induces a continuous map $||Y^\mr{or}_\bullet|| \to |Y_\circ|$. Picking total ordering on the vertices of $Y_\circ$ induces a continuous map $|Y_\circ| \to ||Y^\mr{or}_\bullet||$. These factor the identity map of $|Y_\circ|$ as $|Y_\circ| \to ||Y^\mr{or}_\bullet|| \to |Y_\circ|$. If the identity map of a space factors over an $n$-connected space, it is itself $n$-connected.\end{proof}

Finally, let us sketch the two standard techniques to prove that simplicial complexes are highly connected. We hope this makes the arguments that appear later this section easier to understand, because the general outline is the same in each implementation.

The first is called a \emph{lifting argument} and involves a surjective map $Y_\circ \to Z_\circ$ with $Z_\circ$ highly-connected. Usually the vertices of $Y_\circ$ are vertices of $Z_\circ$ endowed with additional data. One considers a diagram
\[\xymatrix{S^i \ar[r] \ar[d] & |Y_\circ| \ar[d] \\
D^{i+1} \ar[r] \ar@{.>}[ur] & |Z_\circ|}\]
whose maps by simplicial approximation can be assumed to be simplicial with respect a PL-triangulation of $S^i$ and $D^{i+1}$. One wants to find a dotted lift map making the diagram commute, which involves enumerating the interior vertices in $D^{i+1}$ and inductively picking the required additional data. This is made easier if $Y_\circ$ is weakly Cohen-Macauley (see Definition \ref{defwcm}), since then the bottom horizontal map can be assumed to be simplex wise injective by a result of Galatius and Randal-Williams \cite{sorenoscarstability}, here Lemma \ref{lemwchinj}.

The second technique is called a \emph{badness argument} and involves an injective map $Y_\circ \to Z_\circ$ with $Z_\circ$ highly-connected. Usually the simplices of $Y_\circ$ are those of $Z_\circ$ satisfying some additional constraint. One considers a diagram
\[\xymatrix{S^i \ar[r] \ar[d] & |Y_\circ| \ar[d] \\
D^{i+1} \ar[r] \ar@{.>}[ur] & |Z_\circ|}\]
whose maps by simplicial approximation can be assumed to be simplicial with respect a PL-triangulation of $S^i$ and $D^{i+1}$. The goal is to replace the bottom horizontal map with a new map landing in $Y_\circ$, while remaining the same on $S^i$. To do this, one defines a notion of badness for simplices of $D^{i+1}$ with the property that if a simplex is not bad, it is mapped to $Y_\circ \subset Z_\circ$. One then tries to modify the map on bad simplices, starting with those of maximal dimension or highest badness, working one's way down until no bad simplices remain.

\subsection{Bounded charge injective collections} To prove Theorem \ref{thmconndiskdim3} we will first consider a complex of collections of charged points and then lift such collections to disks containing them. In this subsection we prove that the complex of these bounded charge injective collections is highly connected, and furthermore is weakly Cohen-Macaulay, meaning that links of simplices are also highly connected. This condition is a weakening of the condition for a simplicial complex to be a PL-manifold. The reason we care about this property is Lemma \ref{lemwchinj}.

\begin{definition}\label{defwcm} A simplicial complex $X_\circ$ is said to be \emph{weakly Cohen-Macaulay of dimension $\geq n$} if $|X_\circ|$ is $(n-1)$-connected and the link of each $p$-simplex is $(n-p-2)$-connected.\end{definition}

\begin{remark}If $Y_\circ$ is weakly Cohen-Macaulay of dimension $n$, as in Definition \ref{defwcm}, one can prove a converse of Lemma \ref{lemcompsimps}: in that case $||Y^\mr{or}_\bullet||$ is also $(n-1)$-connected. We will not use this in this paper.\end{remark}

An important example of a weakly Cohen-Macaulay complex is the complex of injective words, $\mr{Inj}_\circ(S)$. Here $S$ is a set of cardinality $\#S$ and $\mr{Inj}_\circ(S)$ is the simplicial complex with set of $p$-simplices equal to the set of subsets of $S$ of cardinality $p+1$. Since $\mr{Inj}_\circ(S)$ is isomorphic to an $(\#S-1)$-simplex, this simplicial complex is weakly Cohen-Macaulay of dimension $\#S-1$. We are interested in the following generalization of the complex of injective words:

\begin{definition}Let $S$ be a finite set with a map $f: S \to \{1,\ldots,c\}$. We define the simplicial complex $\mr{Inj}^c_{\circ}(S)$ by setting $0$-simplices to be non-empty subsets $I \subset S$ with charge $\sum_{i \in I} f(i) \leq c$. A $(p+1)$-tuple of subsets $\{I_0,\ldots,I_p\}$ forms $p$-simplex if each pair $I_j$, $I_k$ satisfies $I_j \cap I_k = \emptyset$.\end{definition}

The argument that this complex is weakly Cohen-Macaulay is essentially the same as that used by Randal-Williams in \cite{RW} to prove $\mr{Inj}^\mr{or}_\bullet(S)$, the ordered semisimplicial set associated to the complex $\mr{Inj}_\circ(S)$, is highly connected. Let $\#S$ denote the cardinality of the finite set $S$.

\begin{proposition}\label{propboundedchargewcm} The simplicial complex $\mr{Inj}^c_\circ(S)$ is weakly Cohen-Macaulay of dimension $\geq \lfloor \#S/c \rfloor-1$.\end{proposition}

\begin{proof}
We will prove that $|\mr{Inj}^c_\circ(S)|$ is $(\lfloor \# S/c \rfloor - 2)$-connected by proving that the associated ordered semisimplicial set $\mr{Inj}^{c,\mr{or}}_\bullet(S)$ has $(\lfloor \# S/c \rfloor - 2)$-connected realization, which suffices by Lemma \ref{lemcompsimps}. We will perform a proof by induction over $\#S$.

The initial case is $\# S = 1$, in which case the statement is trivially true. Now suppose $\# S > 1$ and pick a $s \in S$ and consider $\mr{Inj}^{c,\mr{or}}_\bullet(S\backslash \{s\}) \to \mr{Inj}^{c,\mr{or}}_\bullet(S)$. This map is null homotopic since it extends over the cone with $s$, implying that
\[\mr{hocofib}\left(||\mr{Inj}^{c,\mr{or}}_\bullet(S\backslash \{s\})|| \to ||\mr{Inj}^{c,\mr{or}}_\bullet(S)||\right) \simeq ||\mr{Inj}^{c,\mr{or}}_\bullet(S)|| \vee \Sigma ||\mr{Inj}^{c,\mr{or}}_\bullet(S\backslash \{s\})||.\]

We claim that level wise homotopy cofiber of $\mr{Inj}^{c,\mr{or}}_\bullet(S\backslash \{s\}) \to \mr{Inj}^{c,\mr{or}}_\bullet(S)$ is equal to
\[\bigvee_{\stackrel{J \in \mr{Inj}^c_0(S)}{\text{such that }s \in J}} \left(\bigvee_{j=0}^\bullet \mr{Inj}^{c,\mr{or}}_{\bullet-1}(S\backslash J)_+\right)\] 
where the semisimplicial set $\bigvee_{j=0}^\bullet (X_{\bullet -1})_+$ was defined in Section 3 of \cite{RW}. It has $p$-simplices equal to $(X_{p-1} \times \{0,\ldots,p\})_+$ and face maps $\bar{d}_i$ given by
\[\bar{d}_i(x,j) = \begin{cases} (d_i(x),j-1) & \text{if $i < j$} \\
* & \text{if $i = j$} \\
(d_{i-1}(x),j) & \text{if $i > j$} \end{cases}\]
where the $d_i$ are the original face maps of semisimplicial set $X_\bullet$. 

This identification of the level wise homotopy cofiber goes as follows. The $p$-simplices of $\mr{Inj}^{c,\mr{or}}_\bullet(S)$ can be divided into disjoint groups; the first group consists of those not containing $s$ in any of its vertices, and the other groups are those containing $s$ in one of its vertices, and this vertex is equal to $J$ for $J$ some set of charge $\leq c$. The first group is exactly the image of $\mr{Inj}^{c,\mr{or}}_p(S\backslash \{s\})$ and gets identified to the unique basepoint. For every $J$, each element of the set of simplices containing $J$ as a vertex can be uniquely specified by its remaining $p$ vertices and the position in $\{0,\ldots,p\}$ of $J$. This gives the identification of each of these sets of $p$-simplices of the level wise homotopy cofiber. Under this identification the face maps are given by the formula above. 

We will now prove that there is a highly connected map from a highly connected space to each $\bigvee_{j=0}^\bullet \mr{Inj}^{c,\mr{or}}_{\bullet-1}(S\backslash J)_+$ for a subset $J \subset S$ of bounded charge and containing $s$. To do this consider the augmented semisimplicial space $\mr{Inj}^{c,\mr{or}}_{\bullet+i}(S\backslash J) \to \mr{Inj}^{c,\mr{or}}_{i-1}(S\backslash J)$, which has $p$-simplices given by the $(p+i)$-simplices of $\mr{Inj}^{c,\mr{or}}_\bullet(S \backslash J)$ and face maps acting on the first $p+1$ subsets. This has the property that the map $d_i: ||\mr{Inj}^{c,\mr{or}}_{\bullet+i}(S\backslash J)|| \to \mr{Inj}^{c,\mr{or}}_{i-1}(S\backslash J)$ has fiber $||\mr{Inj}^{c,\mr{or}}_\bullet(S\backslash (J \cup \bigcup_{j=0}^{i-1} I_j))||$ over $(I_0,\ldots,I_{i-1})$ in $\mr{Inj}^{c,\mr{or}}_{i-1}(S\backslash J)$. Since $\#S-1-i \geq \#(S\backslash (J \cup \bigcup_{j=0}^{i-1} I_j)) \geq \#S-(1+i)c$, this is $(\lfloor \#S/c\rfloor - i-3)$-connected by the inductive hypothesis. By the realization lemma for semisimplicial spaces (e.g. Proposition 2.6 of \cite{sorenoscarstability}) the map $d_i$ is $(\lfloor \#S/c\rfloor - i-2)$-connected. Using Proposition 3.1 of \cite{RW}, this implies that the map 
\[\Sigma ||\mr{Inj}^{c,\mr{or}}_{\bullet-1}(S\backslash J)|| \to \left|\left|\bigvee_{j=0}^\bullet \mr{Inj}^{c,\mr{or}}_{\bullet-1}(S\backslash J)_+\right|\right|\]
is $(\lfloor \#S/c\rfloor - 2)$-connected.

From this we conclude that the map 
\[\bigvee_{\stackrel{J \in \mr{Inj}^c_0(S)}{\text{such that }s \in J}} \Sigma \left|\left|\bigvee_{j=0}^\bullet \mr{Inj}^{c,\mr{or}}_{\bullet-1}(S\backslash J)_+\right| \right| \to \mr{hocofib}\left(||\mr{Inj}^{c,\mr{or}}_\bullet(S\backslash \{s\})|| \to ||\mr{Inj}^{c,\mr{or}}_\bullet(S)||\right)\]
is $(\lfloor \#S/c\rfloor -2)$-connected. Since the domain is a wedge sum of spaces that are by induction at least $(\lfloor \#S/c \rfloor-2)$-connected, we have that the homotopy cofiber $\mr{hocofib}(||\mr{Inj}^{c,\mr{or}}_\bullet(S\backslash \{s\})|| \to ||\mr{Inj}^{c,\mr{or}}_\bullet(S)||)$ is $(\lfloor \#S/c\rfloor -2)$-connected and since the domain $||\mr{Inj}^{c,\mr{or}}_\bullet(S\backslash \{s\})||$ in the homotopy cofiber is at least $(\lfloor \#S/c\rfloor -3)$-connected by induction we get that $||\mr{Inj}^{c,\mr{or}}_\bullet(S)||$ is $(\lfloor \#S/c\rfloor -2)$-connected.

Thus, by induction and Lemma \ref{lemcompsimps}, we see that  $|\mr{Inj}^c_\circ(S)|$ is $(\lfloor \# S/c \rfloor - 2)$-connected.  It follows that $\mr{Inj}^c_\circ(S)$ is weakly Cohen-Macaulay of dimension $\geq \lfloor \#S/c \rfloor-1$ since the link of a $k$-simplex $(I_0,\ldots,I_k)$ is isomorphic to $\mr{Inj}^c_\circ(S \backslash \cup_{j=0}^k I_j)$ and \[\#S-(k+1) \geq \#(S \backslash \cup_{j=0}^k I_j) \geq \#S-(k+1)c\]
and hence the link of $k$-simplex is $(\lfloor \#S /c \rfloor - k-3)$-connected.

\end{proof}

\subsection{Discrete bounded charge disks in dimension at least three} In this section we will prove Theorem \ref{thmconndiskdim3} about the connectivity of the complex of bounded charge disks. This will use a combination of a transversality argument and the connectivity result for the bounded charge injective words complex established in the previous subsection.

We will give the definition of a simplicial complex corresponding to the semisimplicial sets $K^{\delta,c}_\bullet(x)$ and some small variations of it useful for the induction. For $j \geq 0$ let 
\[\dot{D}^{n-1}_{(j)} := \left(\bigcup_{i=0}^j \left(\frac{i}{j+1},\frac{i+1}{j+1}\right)\right) \times (0,1)^{n-2}\]
which is equal to a disjoint union of $j+1$ open disks, but we call an \emph{$(n-1)$-dimensional disk with $j$ gaps}.

\begin{definition}\label{defnnestedemb} Let $M$ be a connected manifold of dimension $n \geq 3$ that is the interior of a manifold $\bar{M}$ with boundary $\partial \bar{M}$ and fix a $j \geq 0$ and an embedding $\psi: \dot{D}^{n-1}_{(j)}  \to \partial \bar{M}$.  Let $x \in \mr{Sym}_k^{\leq c}(M)$, which is a configuration of points in $M$ labeled by $\{1,\ldots,c\}$ such that the total sum of these labels is $k$. \begin{enumerate}[(i)]
\item  We define $K_\circ^{\delta,c}(x)$ to be the simplicial complex with $0$-simplices bounded charge embeddings $\varphi: (D^n_+,D^{n-1}) \to (\bar{M},\im(\psi))$. A $p$-simplex is a $(p+1)$-tuple of disjoint bounded charge embeddings.
\item We define $NK_\circ^{\delta,c}(x)$ to be the simplicial complex with $0$-simplices bounded charge embeddings $\varphi: (D^n_+,D^{n-1}) \to (\bar{M},\im(\psi))$. A $p$-simplex is a $(p+1)$-tuple of bounded charge embeddings $\{\varphi_0,\ldots,\varphi_p\}$ such that for at most one pair $i \neq j$ we have $\im(\varphi_i) \subset \im(\varphi_j)$ or $\im(\varphi_i) \subset \im(\varphi_j)$ and for all other pairs $i' \neq j'$ we have $\im(\varphi_{i'}) \cap \im(\varphi_{j'}) = \emptyset$.\end{enumerate}\end{definition}

Note that $NK^{\delta,c}_\circ(x)$ contains $K^{\delta,c}_\circ(x)$ as a subcomplex, and $NK^{\delta,c}_c(x)$ differs from $K^{\delta,c}_\circ(x)$ in that a single occurrence of nesting is allowed. We will use this flexibility to replace embedded disks with smaller ones. The following lemma will imply that this flexibility does not come at a loss of connectivity. More precisely, it will be used to compare these two complexes in the induction step in the proof of Theorem \ref{thmconndiskdim3}. 

\begin{lemma}\label{lemkkbarconn} Suppose we are in the situation of Definition \ref{defnnestedemb}, except that $x \in \mr{Sym}_{k+1}^{\leq c}(M)$ instead of $\mr{Sym}_k^{\leq c}(M)$. Further suppose that for all $M'$ that are interiors of connected $n$-dimensional manifolds with boundary, $j' > j$, embeddings of a disk with $j'$ gaps into the boundary and $x' \in \mr{Sym}_{k'}^{\leq c}(M')$ with $k' < k+1$, we have that $|K^{\delta,c}_\circ(x')|$ is $(\lfloor k'/c \rfloor-2)$-connected. 

Then the map $|K^{\delta,c}_\circ(x)| \to |NK^{\delta,c}_\circ(x)|$ is injective on $\pi_i$ for $i \leq \lfloor (k+1)/c \rfloor -2$ with any base point.
\end{lemma}

\begin{proof}This is a badness argument. Let $i \leq \lfloor(k+1)/c \rfloor - 2$ and consider a diagram as follows
\[\xymatrix{S^i \ar[r]^(.4){f} \ar[d] & |K_\circ^{\delta,c}(x)| \ar[d] \\
D^{i+1} \ar[r]_(.4){F} & |NK_\circ^{\delta,c}(x)| }\]
where we can assume that the maps are simplicial with respect to some PL triangulations of $S^i$ and $D^{i+1}$ by simplicial approximation. To prove the lemma it suffices to modify the map $F : D^{i+1} \to |NK_\circ^{\delta,c}(x)| $ to a map $D^{i+1} \to |K_\circ^{\delta,c}(x)| $ making the top triangle commute. This is done inductively, and we keep denoting the intermediate steps by $F$.

We say a simplex $\sigma$ of $D^{i+1}$ is bad if there is a pair $i \neq j$ such that $\im(\varphi_i) \cap \im(\varphi_j) \neq \emptyset$. Note that a simplex of $D^{i+1}$ maps to $K^{\delta,c}_\circ(x) \subset NK^{\delta,c}_\circ(x)$ if it is not bad. Note that if any bad simplices occur, necessarily they must be bad 1-simplices. Suppose that $\sigma$ is a bad 1-simplex. Name them $\varphi_M$ and $\varphi_m$ so that $\im(\varphi_m) \subset \im(\varphi_M)$. By definition of the embeddings we consider, we get a $t_M$ and $\epsilon_M$ such that the image of $\varphi_M$ in the boundary corresponds to an $\epsilon_M$-ball around $(t_M,\frac{1}{2},\ldots,\frac{1}{2})$ under $\psi$.

By restriction to the link we claim we obtain from this a simplicial map 
\[h: \mr{Link}(\sigma) \to K^{\delta,c}_\circ(x \backslash \im(\varphi_M))\]
where the latter denotes the complex for the manifold $M' = M \backslash \im(\varphi_M)$, $j' = j+1$, the disk with gaps $\dot{D}^{n-1}_{(j+1)} \cong (\dot{D}^{n-1}_{(j)}) \backslash (t_M-\epsilon_M,t_M+\epsilon_M) \times (0,1)^{n-2}$ in the boundary and $x' = x \backslash (\im(\varphi_M) \cap x)$. Firstly, the link maps to $K^{\delta,c}_\circ(x \backslash \im(\varphi_M))$ instead of $NK^{\delta,c}_\circ(x \backslash \im(\varphi_M))$ because if a $\tau \in \mr{Link}(\sigma)$ has a 1-simplex $\{\varphi,\varphi'\}$ as a face where the embeddings $\varphi$ and $\varphi'$ intersect, then $\{\varphi,\varphi'\} * \sigma$ has two pairs of vertices intersecting, which is not allowed. Secondly, we can remove the image of $\varphi_M$; suppose that $\tau \in \mr{Link}(\sigma)$ has a $0$-simplex $\varphi$ whose image intersects a $\varphi_M$, then since $\{\varphi,\varphi_M\}$ must be a 1-simplex we have that either $\im(\varphi) \subset \im(\varphi_M)$ or $\im(\varphi_M) \subset \im(\varphi)$. In either of the two cases $\{\varphi\} * \sigma$ would have three vertices intersecting, which is not allowed.

The map $h: \mr{Link}(\sigma) \to K^{\delta,c}_\circ(x \backslash \im(\varphi_M))$ is a map from a sphere $\mr{Link})(\sigma) \cong S^{i-1}$ with $i-1 \leq \lfloor (k+1)/c \rfloor - 3$ to a space that is at least $(\lfloor (k+1)/c \rfloor - 3)$-connected.  This means that we can extend $h$  to a map $H: K \cong D^{i} \to K^{\delta,c}_\circ(x \backslash \im(\varphi_M))$. Now we can modify the original map $F$ by replacing the restriction \[F|_{\mr{Star}(\sigma)}: \mr{Star}(\sigma) = \sigma * \mr{Link}(\sigma) \to |NK_\circ^{\delta,c}(x)|\] with the map
\[F|_{\partial \sigma} * H: \partial \sigma * K \to |NK_\circ^{\delta,c}(x)|\]
Note that both $\mr{Star}(\sigma)$ and $\partial \sigma * K$ have the same boundary, and $F|_{\mr{Star}(\sigma)}$ and $F|_{\partial \sigma} * H$ coincide there. Furthermore, we did not modify the map $F$ on $S^i$, since no bad simplices occur there, so that $S^i \cap (\sigma * \mr{Link}(\sigma)) \subset \partial \sigma * \mr{Link}(\sigma)$.

We claim that this new map has one less bad 1-simplex. Since we removed $\sigma$, it suffices to show we did not introduce any new bad 1-simplices. To see this, note that all changed 1-simplices are of the form $\{\alpha,\beta\}$ where $\alpha$ is a proper face of $\sigma$, hence equal to or contained in $\varphi_M$, and $\beta$ is a 0-simplex in $K^{\delta,c}_\circ(x \backslash \im(\varphi_M))$.
\end{proof}


The next step of the proof of Theorem \ref{thmconndiskdim3} we use that the bounded charge injective word complex is weakly Cohen-Macaulay. We will define this notion next, followed by an important consequence.

\begin{definition}A simplicial map $f: X_\circ \to Y_\circ$ between simplicial complexes is said to be simplex wise injective if for any simplex $\sigma = (x_0,\ldots,x_p)$ in $X_\circ$ with $x_i \neq x_j$ for $i \neq j$ we have that $f(x_i) \neq f(x_j)$ for $i \neq j$.\end{definition}

The following is part of Theorem 2.4 of \cite{sorenoscarstability}. Its proof is a generalization of the Lemma 3.1 of \cite{hatcherwahl}.

\begin{lemma}[Galatius-Randal-Williams] \label{lemwchinj} Let $X_\circ$ be a simplicial complex and $f : S^i \to |X_\circ|$ be a map which is simplicial with respect to some PL triangulation of $S^i$. Then, if $X_\circ$ is weakly Cohen-Macaulay of dimension $n$ and $i 
\leq n-1$, $f$ extends to a simplicial map $g: D^{i+1} \to |X_\circ|$ which is simplex wise injective on the interior of $D^{i+1}$.\end{lemma}

Let $M$ be a manifold of dimension $n \geq 3$ that is the interior of a manifold $\bar{M}$ with boundary $\partial M$ and fix a $j \geq 0$ and an embedding $\psi: \dot{D}^{n-1}_{(j)} \to \partial M$. We will see that  Theorem \ref{thmconndiskdim3} is equivalent to the statement that $|K_\circ^{\delta,c}(x)|$ is at least $(\lfloor k/c \rfloor-2)$-connected in the case $j=0$. We will now give the proof of this in the slightly more general case of arbitrary $j \geq 0$.

\begin{proposition}\label{propconndiskdim3gen} Suppose we are in the situation of Definition \ref{defnnestedemb}, then we have that $|K_\circ^{\delta,c}(x)|$ is at least $(\lfloor k/c \rfloor-2)$-connected.\end{proposition}

\begin{proof}This is a lifting argument. We proceed by induction over $k$, where the inductive hypothesis is that the conclusion of Theorem \ref{thmconndiskdim3} holds for all $M'$, $j' \geq 0$, $\psi'$, and $x'$ of charge $k'$ in $M'$ for $k'<k+1$. In particular note that this inductive hypothesis puts us in a position to apply Lemma \ref{lemkkbarconn}. 

In the case $k=1$ we have that $(\lfloor \frac{k}{c} \rfloor-2) = -2$ or $-1$. In the latter case it is easy to check that $|K_\circ^{\delta,c}(x)|$ is non-empty. Now suppose that we have proven the statement for all $M'$, $j' \geq 0$, $\psi'$ and $x' \in \mr{Sym}^{\leq c}_{k'}(M)$ with $k' <k+1$. We will prove the claim for $k+1$.

We can also think of $x$ as a configuration with each element labeled by a charge in $\{1,\ldots,c\}$. Sending a bounded charge disk to the set it contains gives us a map of simplicial complexes 
\[K_\circ^{\delta,c}(x) \to \mr{Inj}^c_\circ(x)\]

For $i \leq (\lfloor (k+1)/c \rfloor - 2)$ let $f: S^i \to |K_\circ^{\delta,c}(x)|$ be a map that is simplicial with respect to a PL triangulation of $S^i$. The complex $\mr{Inj}^c_\circ(x)$ is weakly Cohen-Macaulay of dimension at least $(\lfloor (k+1)/c \rfloor-1)$ by Proposition \ref{propboundedchargewcm}. Lemma \ref{lemwchinj} implies that there exist a simplicial extension $F: D^{i+1} \to |\mr{Inj}^c_\circ(x)|$ that is simplex wise injective on the interior. Thus consider the following diagram:
\[\xymatrix{S^i \ar[r]^(.42)f \ar[dd] & |K_\circ^{\delta,c}(x)| \ar[dd] \ar[rd] & \\
& & |NK_\circ^{\delta,c}(x)| \\
D^{i+1} \ar[r]_(.45)F \ar@{.>}[uur]^{G} \ar@{.>}[urr]^(.3){G'} & |\mr{Inj}^c_\circ(x)| &}\]
By Lemma \ref{lemkkbarconn} and our inductive hypothesis the map $|K_\circ^{\delta,c}(x)| \to |NK_\circ^{\delta,c}(x)|$ is injective on $\pi_i$ for $i \leq \lfloor (k+1)/c \rfloor -2$, so to produce $G$ it suffices to produce a map $G'$ such that the following square commutes
\[\xymatrix{S^i \ar[r]^(.45)f \ar[d] & |K^{\delta,c}_\circ(x)| \ar[d] \\
D^{i+1} \ar@{.>}[r]_(.4){G'} & |NK^{\delta,c}_\circ(x)|}\]

In producing $G'$ we need to pick the additional data of embedded disks, but since we can not a priori say how small the disks have to be, we need use the trick of at first picking an infinite sequence of increasingly small embedded disks, only to specify to particular ones at the end of the argument. We next give the required definitions to implement this strategy.

For each vertex $v \in S^{i}$ corresponding to an embedding $\varphi^v = h(v)$, pick an embedded rooted tree $\gamma_{v}$ in $\im(\varphi^v)$ having the points of $x$ in $\im(\varphi^v)$ as vertices and a point in the image of $\{(t,\frac{1}{2},\ldots,\frac{1}{2})|\,t \in (0,1)\} \cap D^{n-1}_{(j)}$ under $\im(\psi)$. Next pick a collection of embeddings $\varphi^v_m: (D^n_+,D^{n-1}) \to (\bar{M},\partial \bar{M})$ for $m \in \N$ having the correct behavior near $D^{n-1}$, such that
\begin{enumerate}[(i)]
\item $\varphi^v_0 = \phi_v$,
\item $\im(\varphi^v_{m+1}) \subset \im(\varphi^v_{m}) \subset \im(h(v))$,
\item the $\varphi^v_m$ are cofinal in the images of regular neighborhoods of $\gamma_{v}$.
\end{enumerate}
We call such a collection a \textit{cofinal refinement of an embedding}. Two such cofinal refinements $\{\varphi^v_m\}$ and $\{\psi^v_m\}$ are said to be disjoint if they are eventually disjoint, i.e. there exists an $N$ such that $\im(\varphi^v_m) \cap \im(\psi^{v'}_m) = \emptyset$ for $m \geq N$. This is true if and only if the trees $\gamma_{v}$ and $\gamma_{v'}$ are disjoint , and one can recover the tree from the collection as the intersection of all the images.

By putting a total order $v_0 \prec v_1 \prec \ldots \prec v_N$ on the vertices in the interior of $D^{i+1}$ it suffices to inductively pick lifts for vertices in the interior of $D^{i+1}$. We will denote this lift by $\tilde{G}$ and instead of lifting vertices to embedded disks, we pick cofinal refinements of such embeddings.

Suppose we have already lifted $v_0$ up to $v_{q-1}$, and let $X \subset D^{i+1}$ be the subcomplex spanned by $S^i$ and the $v_{q'}$ for $0 \leq q' \leq q-1$. We want to lift $v_q$ to a cofinal refinement of embeddings that contains the points of $x$ picked out by $v_{q}$ and is disjoint from the cofinal refinement of $\tilde{G}(\mr{Link}(v_q) \cap X)$. This always exists by picking a generic embedded rooted tree $\gamma_{v_q}$ and taking a collection of embedded disks cofinal in the regular neighborhoods. This works because firstly, none of cofinal refinements $\tilde{G}(\mr{Link}(v_q) \cap X)$ contains the points picked by $v_q$ (this is where simplex wise injectivity is used) and secondly, the union of the embedded trees $\gamma$ is a union of codimension $1$ submanifolds, so for a generic embedded tree for the points picked out by $v_q$ any sufficiently small regular neighborhood will avoid all sufficiently small regular neighborhoods for the $\gamma$.

After lifting all vertices in the interior we end with a choice of cofinal refinement for all $v_q$ in $\mr{int}(D^{i+1})$ and $v$ in $S^i = \partial D^{i+1}$. Since there are finitely many of these, we can find a $L \in \N$ such that if we take the $L$th embeddings of the cofinal refinements, we get a map $\tilde{G}: D^{i+1} \to |K^{\delta,c}_\circ(x)|$. However, the restriction to the boundary no longer agrees with $f$. To fix this, we replace $D^{i+1}$ with $S^i \times [0,1] \cup_{S^i} D^{i+1}$, with the given PL triangulations on $S^{i}$ and $D^{i+1}$ and the minimal one on $[0,1]$. This admits a map to $|NK^{\delta,c}_\circ(x)|$ by putting $f$ on the first term and $\tilde{G}$ on the second. Compose with a PL homeomorphism $D^{i+1} \cong S^i \times [0,1] \cup_{S^i} D^{i+1}$ to get the map $G'$.
\end{proof}

We can now finish the proof of Theorem \ref{thmconndiskdim3}.

\begin{proof}[Proof of Theorem \ref{thmconndiskdim3}] This follows from Proposition \ref{propconndiskdim3gen} by taking $j = 0$ and noting that $||K^{\delta,c}_\bullet(x)|| \cong |K^{\delta,c}_\circ(x)|$. For the latter, note that every simplex of $K^{\delta,c}_\circ(x)$ has a canonical ordering coming from the $t$'s determined by image of the embedding in $\im(\psi)$, so that there is a bijection between the $k$-simplices of $K^{\delta,c}_\bullet(x)$ and $K^{\delta,c}_\circ(x)$, compatible with the face relations. It follows that their geometric realizations are homeomorphic.\end{proof}

\subsection{Bounded charge isotopy classes of arcs in dimension two}Our next goal is to prove that a certain arc complex is highly connected. In the next section this will be used to prove Theorem \ref{thmconndiskdim2}. As most arguments on the connectivity of arc complexes, our argument starts with the contractibility (or, in a finite number of exceptional cases, high connectivity) of full arc complexes. After that it is a number of reductions using badness arguments, which can be summarized as follows:
\[\xymatrix{\text{arcs } \mathcal A_\circ(\Sigma,V) \ar[d] \\
 \text{disk-like arcs } \mathcal A_\circ^D(\Sigma,V,\lambda) \ar[d] \\
\text{non-empty arcs } \mathcal A_\circ^{D,\emptyset}(\Sigma,V,\lambda) \ar[d] \\
\text{bounded charge arcs } \mathcal B_\circ(\Sigma,V,c) }\]

Examples of these types of arcs are given in Figure \ref{figtypesofarcs}. The techniques we use are similar to those used to prove connectivity of arc complexes described in \cite{wahlmcg}.

\begin{figure}[t]
\begin{center}
\includegraphics[width=13.5cm]{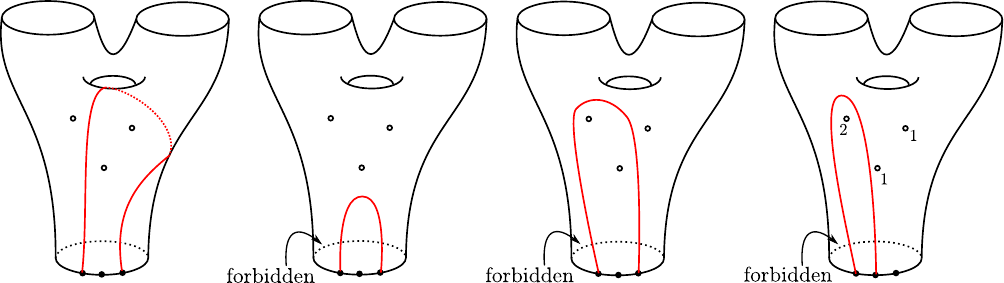}
\end{center}
\caption{All of these arcs are non-trivial, the rightmost three are disk-like (which required picking at least one forbidden interval in the lower boundary for its definition), the rightmost two are non-empty and the rightmost one is of bounded charge $\leq 2$ (which required endowing each puncture with a non-zero charge $\leq 2$).}
\label{figtypesofarcs}
\end{figure}

\subsubsection{Arcs} We start by discussing the arc complex and its connectivity. Let $\Sigma$ be a connected oriented surface with boundary $\partial \Sigma$ and let $V$ be a finite non-empty subset of $\partial \Sigma$. The surface $\Sigma$ has genus $g$, $k$ punctures and $r$ boundary components. An arc $\gamma$ in $\Sigma$ starting and ending at possibly different elements of $V$ is said to be \emph{trivial} if one of the components of $\Sigma \backslash \gamma$ is a disk such that the intersection of its boundary with $V$ is equal to the endpoint(s) of $\gamma$. An isotopy class of arcs is non-trivial if its representatives are non-trivial. A $(p+1)$-tuple of distinct isotopy classes of arcs $\{[\gamma_0],\ldots,[\gamma_p]\}$ is disjoint if there are representatives $\gamma_0,\ldots,\gamma_p$ that are disjoint on the interior.

\begin{definition} The \emph{arc complex} $\mathcal A_\circ(\Sigma,V)$ is the simplicial complex with $p$-simplices collections of $p+1$ disjoint non-trivial isotopy classes of arcs in $\Sigma$ starting and ending at possibly but not necessarily different elements of $V$.
\end{definition}

The following is a well-known theorem, see e.g. \cite{hatcherarc} or the exposition in \cite{wahlmcg}:

\begin{theorem}[Hatcher] \label{thmarccomplex} The complex $\mathcal A_\circ(\Sigma,V)$ is contractible, unless $\Sigma$ is either a disk or an annulus with $V$ contained in a single boundary component. In those cases, it is $(|V|+2r-7)$-connected, where $r$ is the number of boundary components of $\Sigma$.\end{theorem}

\subsubsection{Disk-like arcs} We have that $\partial \Sigma \backslash V$ is a disjoint union of open intervals and circles. We will decorate the intervals by a labeling $\lambda$ of the following form: for each boundary component $\partial_i \Sigma$ containing at least one element of $V$, we pick a non-empty subset of the intervals of $\partial_i \Sigma \backslash V$, which we call \emph{forbidden}. The goal of these forbidden intervals is to encode the complement of $\im(\psi)$.

\begin{definition}A vertex $[\gamma]$ in $\mathcal A_\circ(\Sigma,V)$ is \emph{disk-like with respect to $\lambda$} if for each representative $\gamma$ we have that $\Sigma \backslash \gamma$ consists of two components, at least one of which is a disk with a possibly non-empty set of punctures whose boundary contains no forbidden boundary segments.\end{definition}

\begin{definition} The \emph{disk-like arc complex} $\mathcal A_\circ^D(\Sigma,V,\lambda)$ is the subcomplex of $\mathcal A_\circ(\Sigma,V)$ spanned by the disk-like arcs.\end{definition}

\begin{proposition}The complex of disk-like arcs $\mathcal A_\circ^D(\Sigma,V,\lambda)$ is at least $(k-2)$-connected, where $k$ is the number of punctures, except when $g=0$, $r=1$, $k=1$ and $|V|=1$, in which case it is empty.\end{proposition}

\begin{proof}This is a badness argument. We do an induction over genus $g$, number of boundary components $r$, the number of punctures $k$ and $|V|$, in lexicographic order. The induction starts with the cases $g=0$, $r=1$, $k=1$ and $|V| = 1$ or $2$. To see that the complex is non-empty in the case $g=0$, $r=1$, $k=1$ and $|V| = 2$, pick a point $\delta$ in $V$ and take the arc starting and ending at $\delta$ and circling around the unique puncture, then both components of its complement are non-trivial, see Figure \ref{fignonemptydisklike}.

\begin{figure}[t]
\begin{center}
\includegraphics[width=3.7cm]{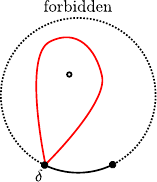}
\end{center}
\caption{A disk-like arc, showing that $\mathcal A_\circ^D(\Sigma,V,\lambda)$ is non-empty in the case $g=0$, $r=1$, $k=1$ and $|V| =2$.}
\label{fignonemptydisklike}
\end{figure}


To prove $|\mathcal A^D_\circ(\Sigma,V,\lambda)|$ is $(k-2)$-connected, we let $-1 \leq i \leq k-2$, consider a map $S^i \to |\mathcal A^D_\circ(\Sigma,V,\lambda)|$ and show that it can be extended to a map from the disk $D^{i+1}$. By simplicial approximation, we can assume that it is simplicial with respect to some PL triangulation of $S^i$. Since $|\mathcal A_\circ(\Sigma,V)|$ is contractible, we can find a PL triangulation of $D^{i+1}$ and a simplicial map $D^{i+1} \to \mathcal A_\circ(\Sigma,V)$ making the following diagram commute:
\[\xymatrix{S^i \ar[r]^(.4)f \ar[d] & |\mathcal A^D_\circ(\Sigma,V,\lambda)| \ar[d] \\
D^{i+1} \ar[r]_(.4)F & |\mathcal A_\circ(\Sigma,V)|.}\] We say that a simplex $\sigma$ of $D^{i+1}$ is bad if all vertices of $F(\sigma)$ are not disk-like. Thus if $F$ has no bad simplices it lifts to $\mathcal A^D_\circ(\Sigma,V,\lambda)$. We will inductively modify $F$ by removing bad simplices of maximal dimension, denoting the intermediate steps by $F$ again. Let $\sigma$ be a bad simplex of maximal dimension $p'$ and suppose $f(\sigma)$ has dimension $p \leq p'$, represented by a $(p+1)$-tuple $\{\gamma_0,\ldots,\gamma_p\}$ of disjoint arcs. 

The surface $\Sigma \backslash \{\gamma_0,\ldots,\gamma_p\}$ consists of $s+1$ components $\Sigma_0,\ldots,\Sigma_s$ and we claim that all of the components $\Sigma_i$ have $(g_i,r_i,k_i,|V_i|) \prec (g,r,k,|V|)$ with respect to lexicographical order. To see this, first note that by induction over the number of arcs it suffices to prove this for $p = 0$, i.e. when there is one arc. Consider the various cases for $\gamma_0$: \begin{enumerate}[(i)]
\item If $\gamma_0$ is non-separating we get a single component whose genus or number of boundary components is smaller, depending on whether $\gamma_0$ has endpoints on different or the same boundary component.
\item If $\gamma_0$ is separating we get two components, each of which has smaller genus, a smaller number of boundary components, a smaller number of punctures or a smaller number of elements of $V$.  This is true because these numbers are essentially additive under glueing along arcs in the boundaries of two different components: if $g_1,r_1,k_1,|V_1|$ and $g_2,r_2,k_2,|V_2|$ denote the relevant numbers for the two components, then for the original surface we must have had $g = g_1+g_2$, $r = r_1 + r_2 -1$, $k = k_1+k_2$ and $|V| = |V_1|+|V_2|-2$. The only exceptional case could be $g_1=0$, $r_1=1$, $k_1=0$ and $|V_1|=1,2$. If this occurs then $\gamma_0$ would have been trivial, which proves the claim. 
\end{enumerate}

We associate to each of the components $\Sigma_j$ a new labeling $\lambda_j$ of the boundary intervals as follows. The intervals in $\partial \Sigma_j$ are a disjoint union of a subset of the intervals of $\partial \Sigma$ and at most two copies of the arcs $\gamma_i$. In the new labeling, all the intervals of $\partial \Sigma$ retain their original decoration, while all intervals coming from the arcs $\gamma_i$ are labeled forbidden. This choice of labeling will guarantee that being disk-like in $\Sigma_i$ implies being disk-like in $\Sigma$.

We then claim that restricting to the link of $\sigma$ gives a map: \[S^{i-p'} \cong \mr{Link}(\sigma) \to \mathcal A_\circ^D(\Sigma_0,V_0,\lambda_0) * \ldots * \mathcal A_\circ^D(\Sigma_s,V_s,\lambda_s).\] If this is not the case, then $\sigma$ was not of maximal dimension. By virtue of being in the link any arc $\gamma$ in the image of $\sigma$ must lie in $\mathcal A_\circ(\Sigma_i,V_i)$ for some $i$. If $\gamma$ is not disk-like with respect to the set of forbidden intervals $\lambda_i$ in $\Sigma_i$, then it must satisfy one of the following conditions:
\begin{enumerate}[(i')]
\item It is non-separating in $\Sigma_i$ and thus must have been non-separating in $\Sigma$ and hence not disk-like in $\Sigma$.
\item Its complement has two components that both are not punctured disks and hence $\gamma$ cannot have been disk-like in $\Sigma$.
\item Its complement has two components, one or both of which are punctured disks with a forbidden interval from $\lambda_i$. In this case, one must distinguish the reasons that the interval is forbidden by $\lambda_i$. If the forbidden interval came from $\lambda$ then that component of the complement of $\gamma$ in $\Sigma$ had a forbidden interval in its boundary. If on the other hand the forbidden interval came from one of the arcs $\gamma_i$ for $0 \leq i \leq p$, that means that $\gamma$ in $\Sigma$ must either have been non-separating, have the wrong complement or have a forbidden interval in its boundary, depending on the reason $\gamma_i$ was not disk-like. 
\end{enumerate}

Note that the exceptional case $g_i = 0$, $r_i=1$, $k_i=1$, $|V_i|=1$ can not occur, as then one of the arcs would have been disk-like. Thus the codomain has connectivity $\left(\sum_{j=0}^s k_j\right)-2 = k-2$ using the inductive hypothesis and since $i-p' \leq i-p \leq i \leq k-2$, we can extend the restriction of $F$ to a map $H$ from the disk $D^{i-p'+1} \cong K\to \mathcal A_\circ^D(\Sigma_0,V_0,\lambda_0) * \ldots * \mathcal A_\circ^D(\Sigma_s,V_s,\lambda_s)$, where $K$ is some PL triangulation of the disk $D^{i-p'+1}$ that extends the original PL triangulation of $S^{i-p'} \cong \mr{Link}(\sigma)$.

Now we can modify the original map $f$ by replacing
\[F|_{\mr{Star}(\sigma)}: \mr{Star}(\sigma) = \sigma * \mr{Link}(\sigma) \to |\mathcal A_\circ(\Sigma,V,\lambda)|\]
with the map
\[F|_{\partial \sigma} * H: \partial \sigma * K \to |\mathcal A_\circ(\Sigma,V,\lambda)|\]
Note that both $\mr{Star}(\sigma)$ and $\partial \sigma * K$ have the same boundary, and $F|_{\mr{Star}(\sigma)}$ and $F|_{\partial \sigma} * H$ coincide there. Furthermore, we did not modify the map on $S^i$, since no bad simplices occur there, so that $S^i \cap (\sigma * \mr{Link}(\sigma)) \subset \partial \sigma * \mr{Link}(\sigma)$.

We claim this modification reduces the number of bad simplices of maximal dimension by one. For this we only need to remark that the labelings $\lambda_i$ guarantee that an isotopy class of arc in $\mathcal A_\circ^D(\Sigma_i,V_i,\lambda_i)$ is also disk-like for $\Sigma$ with labeling $\lambda$. This uses the same reasoning used to conclude that the link of a bad simplex of maximal dimension maps into a join of complexes of disk-like arcs; analyze the possible ways how an arc in $\mathcal A_\circ^D(\Sigma_i,V_i,\lambda_i)$ could no longer be disk-like when considered as an arc in $\Sigma$ and conclude this never occurs.\end{proof}

\subsubsection{Non-empty arcs} Next we define the non-empty arc complex. It is the subcomplex of the disk-like arc complex consisting of those arcs that in a precise sense encircle some punctures.

\begin{figure}[t]
\begin{center}
\includegraphics[width=4.5cm]{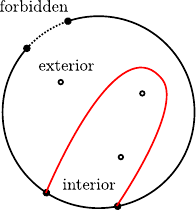}
\end{center}
\caption{An example of a non-empty disk-like arc with its interior and exterior.}
\label{figinteriorexterior}
\end{figure}

\begin{definition}A disk-like arc $\gamma$ cuts a connected surface $\Sigma$ into two components, one of which is a punctured disk that does not have any forbidden intervals in its boundary apart from $\gamma$ itself. This component is called the \textit{interior} of $\gamma$, the other component the \textit{exterior}, see Figure \ref{figinteriorexterior}. 

A disk-like arc $\gamma$ is said to be \emph{non-empty} if the interior contains at least one puncture. A vertex $[\gamma]$ in $\mathcal A_\circ^D(\Sigma,V,\lambda)$ is said to be non-empty if its representatives are. \end{definition}

\begin{definition} The \emph{non-empty arc complex} $\mathcal A_\circ^{D,\emptyset}(\Sigma,V,\lambda)$ is the subcomplex of $\mathcal A_\circ^D(\Sigma,V,\lambda)$ spanned by the non-empty arcs. 
\end{definition}

\begin{proposition}The complex $\mathcal A_\circ^{D,\emptyset}(\Sigma,V,\lambda)$ is $(k-2)$-connected, except when $g=0$, $r=1$, $k=1$ and $|V| = 1$, in which case it is empty.\end{proposition}

\begin{proof}This is neither a lifting argument nor a badness argument, but a parametrized surgery argument particular to arc complexes as in Section 4.2 of \cite{wahlmcg}. If all the intervals are forbidden, $\mathcal A_\circ^{D,\emptyset}(\Sigma,V,\lambda)$ is actually equal to $\mathcal A_\circ^D(\Sigma,V,\lambda)$, because in this case to be non-trivial a disk-like arc must contain a puncture on its interior and thus be non-empty. We conclude from the previous proposition that in this case $\mathcal A_\circ^{D,\emptyset}(\Sigma,V,\lambda)$ is $(k-2)$-connected. We will prove that $\mathcal A_\circ^{D,\emptyset}(\Sigma,V,\lambda)$ has the same connectivity as $\mathcal A_\circ^{D,\emptyset}(\Sigma,V,\lambda')$ for $\lambda'$ containing all intervals, so that $\mathcal A_\circ^{D,\emptyset}(\Sigma,V,\lambda)$ is $(k-2)$-connected as well. 

We will now describe the procedure for checking this connectivity statement. There are two cases: 
\begin{enumerate}[(i)]
\item $\Sigma$ is a punctured disk with $|V| = 2$,
\item $\Sigma$ is not a punctured disk with $|V|=2$.
\end{enumerate}

We do the second case first, as it is easier. In that case $\mathcal A_\circ^{D,\emptyset}(\Sigma,V,\lambda)$ deformation retracts onto $\mathcal A_\circ^{D,\emptyset}(\Sigma,V,\lambda')$. It suffices to show how to increase by one the number of forbidden intervals of one of the boundary components. Pick an orientation on the boundary component we are considering and let $\delta_0$ be a point bordering the forbidden interval, and $\delta_1$ a point in $V$ adjacent to it and not separated from it by the forbidden interval. We will give a deformation retraction of $\mathcal A_\circ^{D,\emptyset}(\Sigma,V,\lambda)$ onto the subcomplex of isotopy classes of arcs not having the interval from $\delta_0$ to $\delta_1$ in their interior, so we can label it forbidden. 

The idea is that given a $p$-simplex in $\mathcal A_\circ^{D,\emptyset}(\Sigma,V,\lambda)$ we look at the germs of arcs going in $\delta_0$ and move these one at a time to $\delta_1$. See Figure \ref{figarcgermflow} for a figure; a formula in simplicial coordinates can be found in the proof of Lemma 4.2 of \cite{wahlmcg}.

\begin{figure}[h]
\begin{center}
\includegraphics[width=13.5cm]{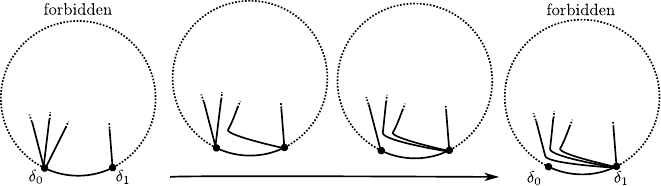}
\end{center}
\caption{The deformation of arcs away from $\delta_0$ to $\delta_1$. Note that the modifications are made locally near the interval connecting $\delta_0$ and $\delta_1$, and hence we only drew the germs.}
\label{figarcgermflow}
\end{figure}

We need to check that this procedure does not create trivial arcs. This is true since by construction they will be non-empty and disk-like. This uses that $\Sigma$ is not a punctured disk with $|V| = 2$. Indeed, only in the latter case is there a non-empty arc that leads to trivial arcs during the procedure. It is given by the arc with both endpoints at $\delta_0$ containing all punctures in its interior. 

To see this, first note that the interior of a non-empty arc must always contain some punctures and hence the same is true for the interiors of the arcs involved in the procedure. Secondly, note that the exterior $\Sigma_\mr{ext}$ of the complement of a non-empty arc always has genus or additional boundary components if $\Sigma$ had $g \geq 1$ or $r \geq 2$, so in that case can't be trivial, and the same is true for all arcs that arise from it during the procedure. The same is true if $\Sigma_\mr{ext}$ has at least one puncture. This means that the only case to consider are arcs in a surface satisfying $g=0$ and $r=1$, that contain all punctures in their interior. This means that if a trivial arc arises during the procedure, it must be not the interior but the exterior which becomes trivial. If $|V| \geq 3$, then any arc modified by the procedure will have at least three points of $V$ in $\partial \Sigma_\mr{ext}$, something which is unchanged by the procedure. So the only case to consider is $|V| = 2$. In that case arcs containing all the punctures in their interior must have both endpoints at the some point of $V$ to be non-trivial; there are exactly two of these. The only one of these arcs that is modified during the procedure is the one starting and ending at $\delta_0$, and one can easily check this leads to a trivial arc.

We will now consider the first case, i.e. $\Sigma$ is a punctured disk with $|V|=2$. As mentioned before, the previous argument fails because we might encounter trivial arcs. This happens when $\gamma$ is the arc starting and ending at $\delta_0$ containing all punctures in its interior. Instead, note that \[\mathcal A_\circ^{D,\emptyset}(\Sigma,V,\lambda) \cong (\mathcal A_\circ^{D,\emptyset}(\Sigma,V,\lambda) \backslash [\gamma]) \cup_{\mr{Link}(\gamma)} \mr{Star}(\gamma)\]
and it suffices to prove that $\mathcal A_\circ^{D,\emptyset}(\Sigma,V,\lambda) \backslash [\gamma]$ and $\mr{Link}(\gamma)$ are highly connected, because $\mr{Star}(\gamma)$ is contractible. Firstly note that $\mr{Link}(\gamma)$ is isomorphic to the subcomplex of non-empty arcs starting and ending at $\delta_0$ that are not equal to $\gamma$. This is isomorphic to the non-empty arc complex for the $k$-fold punctured disk with $|V'|=1$ and thus is $(k-2)$-connected. Secondly note that the argument of the previous case does apply to $\mathcal A_\circ^{D,\emptyset}(\Sigma,V,\lambda) \backslash [\gamma]$ since $[\gamma]$ is the only isotopy class made trivial by the procedure outlined there. Thus we can prove that $\mathcal A_\circ^{D,\emptyset}(\Sigma,V,\lambda) \backslash [\gamma]$ deformation retracts onto the non-empty arc complex for the $k$-fold punctured disk with $|V'|=1$ and thus also is $(k-2)$-connected.
\end{proof}

\subsubsection{Bounded charge arcs} In the following argument we restrict to the case that $V$ is contained in a single boundary component and there is a single bad interval in $\lambda$. Next we suppose that the $k$ punctures $\{p_1,\ldots,p_k\}$ are labeled by a map $c_p: \{p_1,\ldots,p_k\} \to \{1,\ldots,c\}$, called the charge. We also pick a second (possibly empty) subset $\lambda_c$ of the intervals in the boundary disjoint from the forbidden intervals $\lambda$, which are called the overcharged intervals. The set of overcharged intervals can be empty, and in the end will be, since they are just a tool in the induction. We denote this triple $(\lambda,c_p,\lambda_c)$ of labelings, i.e. the collection of forbidden boundary segments $\lambda$, the charge map $c_p$ and the collection of overcharged boundary segments $\lambda_c$, by $c$.

\begin{definition}A non-empty arc is said to be \emph{of bounded charge} if the total sum of the charges of the punctures in its interior is $\leq c$ and none of the intervals $\{I_1,\ldots,I_m\}$ in its interior are overcharged.\end{definition}

\begin{definition}The \emph{bounded charge arc complex} $\mathcal B_\circ(\Sigma,V,c)$ is the subcomplex of $\mathcal A^{D,\emptyset}_\circ(\Sigma,V,\lambda)$ spanned by the non-empty arcs of bounded charge.\end{definition}

\begin{proposition}\label{propboundedchargeconn} Suppose that $V$ is contained in a single boundary component and $\lambda$ consists of a single forbidden interval in that boundary component. Then the bounded charge arc complex $\mathcal B_\circ(\Sigma,V,c)$ is $(k-2)$-connected, except when $g = 0$, $r=1$, $k=1$ and $|V| = 1$ in which case it is empty.\end{proposition}

\begin{proof}This is a badness argument. We first do some initial cases. If $g=0$, $r=1$ and $k=1$ it is empty if $|V|=1$ and non-empty if $|V| \geq 2$, again using the arc obtained by picking a point $\delta$ in $V$ and taking the arc starting and ending at $\delta$ and circling around the unique puncture. If $k=1$, $|V|\geq 1$ and we have $g >0$ or $r>1$, then it is non-empty by a similar construction of an arc starting and ending at some $\delta$ in $V$ and circling around the unique puncture.

We then do an induction over $g$, $r$, $k$ and $|V|$ in lexicographic ordering. The previous preliminary work implies it suffices to only consider the cases $k \geq 2$.

Let $i \leq k-2$ and consider a simplicial map $f: S^i \to \mathcal B_\circ(\Sigma,V,c)$, then since $\mathcal A^{D,\emptyset}_\circ(\Sigma,V,\lambda)$ is $(k-2)$-connected, using simplicial approximation we can find a PL triangulation of $D^{i+1}$ and a simplicial map $F: D^{i+1} \to \mathcal A^{D,\emptyset}_\circ(\Sigma,V,\lambda)$ making the following diagram commute:
\[\xymatrix{S^i \ar[r]^(.4){f} \ar[d] & \mathcal B_\circ(\Sigma,V,c) \ar[d] \\
D^{i+1} \ar[r]_(.4)F & \mathcal A^{D,\emptyset}_\circ(\Sigma,V,\lambda).}\]

A simplex $\sigma$ of $D^{i+1}$ is bad if all of the vertices of $F(\sigma)$ are not of bounded charge. Thus if there are no bad simplices then $F$ lifts to $\mathcal B_\circ(\Sigma,V,c)$. We inductively modify $F$ by removing bad simplices of maximal dimension, denoting the intermediate steps by $F$ again. Let $\sigma$ be a bad simplex of maximal dimension $p'$ and suppose that $F(\sigma)$ is of dimension $p \leq p'$, represented by a collection $\gamma_0,\ldots,\gamma_p$.

Consider the surface with boundary obtained by taking $\Sigma \backslash \{\gamma_0,\ldots,\gamma_p\}$. This has $p+2$ components, which we denote 
$\Sigma_0,\ldots,\Sigma_{p+1}$. Each component $\Sigma_i$ has $k_i \geq 0$ punctures (where the $k_i$ satisfy $\sum k_i = k$) each of which is labeled by their original charge. Each component $\Sigma_i$ has a set $V_i$ of marked points on its boundary, with $|V_i| \geq 1$. The interior segments of the boundary of $\Sigma_i$ are a union of a subset of the interior segments of the boundary of $\partial \Sigma \backslash V$ and a subset of the arcs $\gamma_j$. We label the former by their original labeling and for the latter we do the following: if $\Sigma_i$ is in the interior of $\gamma_j$ we label it forbidden, if $\Sigma_i$ is not we label it overcharged. Note that this again leads to surfaces with $V_i$ contained in a single boundary component and a single forbidden interval. Combined with the labeling of the punctures we get labelings $c_i$ for $\Sigma_i$. 

Next we claim that every $\Sigma_i$ precedes $\Sigma$ in the lexicographic ordering and the exceptional case does not occur. By induction over the number of arcs it suffices to prove this for $p=0$, i.e. when there is one arc. If $\Sigma$ had $g>0$ or $r>1$ this is clear: the interior has no genus and one boundary component, and the other component has at least one less puncture. So it suffices to restrict to the case of $\Sigma$ being a disk. In that case both components have less punctures, unless the arc contains all punctures in its interior. But in that case if the boundary of the interior contained all points of $V$ it would be trivial. We also claim that the components with a single puncture always satisfy $|V| \geq 2$. Since $\sigma$ is bad, all of the $\gamma_i$'s are not of bounded charge. If a disk-like curve contains exactly one puncture, the only way it can fail to be of bounded charge is if there is an overcharged interval. However, this requires that $|V|$ is at least $2$. This rules out the exceptional case of $g=0$, $r=1$, $k=1$ and $|V| = 1$ from occurring among the $\Sigma_i$.

By restriction to the link we get a map:
\[S^{i-p'} \cong \mr{Link}(\sigma) \to \mathcal B_\circ(\Sigma_0,V_0,c_0) * \ldots * \mathcal B_\circ(\Sigma_{p+1},V_{p+1},c_{p+1}).\]
To see this, first note that a vertex of the link goes to a non-empty arc in some $\Sigma_i$, because it must be of bounded charge in $\Sigma$ and thus can't contain any $\gamma_i$ in its interior. Secondly, we note that if a vertex of the link goes to an arc that is not of bounded charge $\sigma$ would not be of maximal dimension; this is true because either that arc would contain punctures of total charge $>c$ or it has an overcharged interval in the boundary of its interior, in which case in $\Sigma$ its interior containing one of the bad arcs $\gamma_i$ and thus in $\Sigma$ contains punctures of total charge $>c$ or has an overcharged interval as its boundary. In both cases the arc was bad in $\Sigma$. By the inductive hypothesis the codomain has connectivity $\sum_{i \geq 0} k_i-2 = k-2$, with $k_i$ the number of punctures in $\Sigma_i$. We are mapping in a sphere of dimension at most $k-2$ and thus this means we can extend our map to a map $H: D^{i-p'+1} \cong K \to \mathcal B_\circ(\Sigma_0,V_0,c_0) * \ldots * \mathcal B_\circ(\Sigma_{p+1},V_{p+1},c_{p+1})$. 

Now we can modify the original map $f$ by replacing
\[F|_{\mr{Star}(\sigma)}:  \mr{Star}(\sigma) = \sigma * \mr{Link}(\sigma) \to |\mathcal A^{D,\emptyset}_\circ(\Sigma,V,\lambda)|\]
with the map
\[F|_{\partial \sigma} * H: \partial \sigma * K \to |\mathcal A^{D,\emptyset}_\circ(\Sigma,V,\lambda)|\]
Note that both $\mr{Star}(\sigma)$ and $\partial \sigma * K$ have the same boundary, and $F|_{\mr{Star}(\sigma)}$ and $F|_{\partial \sigma} * H$ coincide there. Furthermore, we did not modify the map on $S^i$, since no bad simplices occur there, so that $S^i \cap (\sigma * \mr{Link}(\sigma)) \subset \partial \sigma * \mr{Link}(\sigma)$.

We claim this reduces the number of bad simplices of maximal dimension by one. To see this, note that all changed simplexes are of the form $\alpha * \beta$ where $\alpha$ is a proper face of $\sigma$ and $\beta$ is a simplex of $K$. By construction all vertices of $K$ are not bad, which uses similar reasoning about the labelings as before, and in $\alpha * \beta$ at least one vertex of $\sigma$ does not occur.\end{proof}


\subsection{Topologized bounded charge disks in dimension two}

In the previous section we discussed the complex $\mathcal B_\circ(\Sigma,V,c)$. If we pick an orientation on the boundary component containing $V$ there is a canonical ordering on the elements of each $p$-simplex: they are lexicographically ordered by the first point of $V$ they contain, starting at the one to the right of the forbidden interval, and then by the ordering of the germs at each point of $V$ in counterclockwise direction. This also gives each arc a canonical orientation.

\begin{definition}Let $\mathcal D_\bullet(\Sigma,V,c)$ be the semisimplicial space with space of $p$-simplices equal to the space of $(p+1)$-tuples of disjoint embeddings of $(I,\partial I) \to (\Sigma,V)$ such that \begin{enumerate}[(i)]
\item each embedding is of bounded charge and disk-like,
\item each embedding has the same orientation as the canonical one described above,
\item the elements within the $(p+1)$-tuple are ordered as described above.
\end{enumerate}\end{definition}

\begin{definition} Let $\mathcal D^0_\bullet(\Sigma,V,c)$ be the semisimplicial set obtained by taking $\pi_0$ of each of the spaces of $p$-simplices of $\mathcal D_\bullet(\Sigma,V,c)$.
\end{definition}

There is a semisimplicial map $\mathcal D_\bullet(\Sigma,V,c) \to \mathcal D^0_\bullet(\Sigma,V,c)$ sending a $(p+1)$-tuple to the $(p+1)$-tuple of their isotopy classes. The ordering discussed above gives a bijection between the $p$-simplices of $\mathcal D^0_\bullet(\Sigma,V,c)$ and those of $\mathcal B_\circ(\Sigma,V,c)$. Because this is compatible with the face relations we get the following lemma:

\begin{lemma}There is a homeomorphism $||\mathcal D^0_\bullet(\Sigma,V,c)|| \cong |\mathcal B_\circ(\Sigma,V,c)|$.\end{lemma}

Let us next compare  $\mathcal D_\bullet(\Sigma,\Delta,c)$ and $\mathcal D^0_\bullet(\Sigma,\Delta,c)$.

\begin{lemma}The map $\mathcal D_\bullet(\Sigma,\Delta,c) \to \mathcal D^0_\bullet(\Sigma,\Delta,c)$ induces a weak equivalence $||\mathcal D_\bullet(\Sigma,\Delta,c)|| \simeq ||\mathcal D^0_\bullet(\Sigma,\Delta,c)||$.\end{lemma}

\begin{proof}The realization lemma (e.g. Proposition 2.6 of \cite{sorenoscarstability}) says that a map of semisimplicial spaces that is a level wise weak equivalence, induces a weak equivalence on geometric realization. It thus suffices to prove that the map $\mathcal D_\bullet(\Sigma,\Delta,c) \to \mathcal D^0_\bullet(\Sigma,\Delta,c)$ is a level wise homotopy equivalence. This follows from the fact that the space of representatives of an isotopy class of arcs with fixed endpoints in the $C^\infty$-topology is contractible. A reference for this is part (III) of Appendix B of \cite{hatchermw}. Alternatively, one can deduce this by a doubling argument from curve-shortening \cite{grayson} (which says that the space of curves in a given isotopy class is contractible), or the techniques of \cite{gramain}. \end{proof}

Recall that the semisimplicial space $K^c_\bullet(x)$ has as space of $p$-simplices the space of $(p+1)$-tuples of disjoint bounded charge embeddings $(D^2_+,D^1) \to (\Sigma,\im(\psi))$ relative to $V$ or $(D^2_\vee,D^0) \to (\Sigma,\im(\psi))$ relative to $V$. There is a map $K_\bullet^c(x) \to \mathcal D_\bullet(\Sigma,V,c)$ by restricting the embeddings of disks to their boundary. Here $\psi$ provides the orientation on the boundary component containing $V$ that we used before. Furthermore, the part of that boundary component containing the complement of the image of $\psi$ gives the forbidden boundary segment in $\lambda$.

\begin{lemma}The map $K_\bullet^c(x) \to \mathcal D_\bullet(\Sigma,V,c)$ induce a weak equivalence $||K_\bullet^c(x)|| \simeq ||\mathcal D_\bullet(\Sigma,V,c)||$.\end{lemma}

\begin{proof}The second claim follows from the first by the realization lemma. The map of $p$-simplices from topologized disks to topologized arcs is a principal $\mr{Diff}(D^2;\partial D^2)^{p+1}$-bundle, but this fiber is contractible by Smale's theorem \cite{smaledisk}.\end{proof}

We can now finish the proof of Theorem \ref{thmconndiskdim2}.

\begin{proof}[Proof of Theorem \ref{thmconndiskdim2}] By the previous lemma's we have that
\[||K_\bullet^c(x)|| \simeq ||\mathcal D_\bullet(\Sigma,V,c)|| \simeq ||\mathcal D^0_\bullet(\Sigma,V,c)|| \cong |\mathcal B_\circ(\Sigma,V,c)|\]
and the right-most term is $(k-2)$-connected by Proposition \ref{propboundedchargeconn}.
\end{proof}

\begin{remark}\label{remarcconnectedcomponents} Note that these arguments imply that $K^c_p(x)$ has contractible connected components and that these components are in bijection with the $p$-simplices of $\mathcal B^0_\circ(\Sigma,V,c)$.\end{remark}

\section{Homological stability for completions} \label{sechomstab} We now prove homological stability for the spaces $\int^k_M P$ (Theorem \ref{thmmain}). We will prove our result for the cases of manifolds of dimension $\geq 3$ and manifolds of dimension $2$ separately. The reason for the independent proofs is that the case of surfaces requires a slightly different semisimplicial resolution of $\smash{\int_M^k P}$. 

Both arguments follow the standard outline for homological stability arguments: perform an induction over $k$ by 
\begin{enumerate}[(i)]
\item resolving the $k$th space by a semisimplicial space whose spaces of $p$-simplices are related to spaces earlier in induction,
\item simplifying the $E^1$-page of the spectral sequence for the geometric realization of this resolution using the inductive hypothesis,
\item using knowledge of the augmentation map to the original space to get control over the edge map or final differential.
\end{enumerate}

\subsection{The bounded charge resolution in dimension at least three} Recall that we fixed a connected manifold $M$ of dimension $n \geq 3$ that is the interior of a manifold $\bar{M}$ with boundary $\partial \bar{M}$. We also fixed an embedding $\psi: (0,1)^{n-1} \to \partial \bar{M}$ disjoint from but isotopic to $\chi$ and a collar $\tilde{\psi}$. Then we can define a semisimplicial space $K^{\delta,c}_\bullet$ as the semisimplicial space with $0$-simplices embeddings $\varphi: (D^n_+,D^{n-1}) \to (\bar{M},\im(\psi))$. These embeddings are \emph{not} topologized, hence the $\delta$. A $p$-simplex is a $(p+1)$-tuple of disjoint embeddings, with ordering induced by $\psi$. Recall the macroscopic location map 
\[L: \int_{M,\leq c}^k P \to \mr{Sym}^{\leq c}_k(M)\] defined in Subsection \ref{subsectchdef}: for each $x \in \int_{M,\leq c}^k P$, $L(x) \in \mr{Sym}^{\leq c}_k(M)$ is given by recording the macroscopic locations of the particles and their charge. Then the complex $K^{\delta}_\bullet(L(x))$ is a subcomplex of $K^{\delta,c}_\bullet$.

\begin{definition}Let $\mathcal T_\bullet[k]$ be the augmented semisimplicial space with space of $P$-simplices $\mathcal T_p[k]$ consisting of pairs of an $x \in \int^k_{M,\leq c} P$ and an element $(\phi_0,\ldots,\phi_p)$ of $K^{\delta}_p$ such that $(\phi_0,\ldots,\phi_p) \in K^{\delta,c}_p(L(x)) \subset K^{\delta}_p$. The augmentation is the forgetful map to $\int^k_M P$.\end{definition}

Propositions \ref{propd3resconn} and \ref{propd3respsimplices} are direct consequences of this definition. The first will use a property of microfibrations.

\begin{definition}A map $f: E \to B$ is a (Serre) microfibration if for $k \geq 0$ and each commutative diagram 
\[\xymatrix{\{0\} \times D^k \ar[rr] \ar[dd] \ar[rd] & & E \ar[dd] \\
& [0,\epsilon) \times D^k \ar[ld] \ar@{.>}[ur]  & \\
[0,1] \times D^k \ar[rr] & & B}\]
there exists an $\epsilon \in (0,1]$ and a dotted map making the diagram commute.\end{definition}

Serre fibrations are characterized by a lifting property involving lifting a $D^n$-parametrized family of maps over the unit interval, and similarly Serre microfibrations are characterized by a weaker lifting property where we only require that we are able to find a lift for some small portion of the interval. We recall the following lemma, which is Proposition 2.5 of \cite{sorenoscarstability}. This was a generalization of Lemma 2.2 of \cite{weissclassify}.

\begin{proposition}[Galatius-Randal-Williams]\label{lemmicrofibconn} If $f: E \to B$ is a microfibration such that for all $b \in B$ the fiber $f^{-1}(b)$ is $n$-connected, then the map $f$ is $(n+1)$-connected.\end{proposition}

Using this we prove the first property of $\mathcal T_\bullet[k]$.

\begin{proposition}\label{propd3resconn} The augmentation map $||\mathcal T_\bullet[k]|| \to \int_M^k P$ is $(\lfloor k/c\rfloor -1)$-connected.
\end{proposition}

\begin{proof}We first note that map $\int^k_{M,\leq c} P \to \int^k_M P$ is a homotopy equivalence by Corollary \ref{corleqcweq}, so it suffices to prove that the forgetful map to $\int^k_{M,\leq c} P$ is $(\lfloor k/c\rfloor -1)$-connected. To see this, we prove that the map $||\mathcal T_\bullet[k]|| \to \int^k_{M,\leq c} P$ is a microfibration with highly-connected fibers. To check it is a microfibration, consider a commutative diagram 
\[\xymatrix{\{0\} \times D^n \ar[r]^f \ar[d] & ||\mathcal T_\bullet[k]||  \ar[d] \\
[0,1] \times D^n \ar[r]_F & \int^k_{M,\leq c} P}\] 
then we need to find a diagonal lift for some initial segment $[0,\epsilon) \times D^n$. For this it suffices to note that if a collection of disks is bounded charge with respect to some configuration $x \in \int_{M,\leq c}^k P$, it also is bounded charge with respect to nearby points in $ \int_{M,\leq c}^k P$. Thus for each $d \in D^k$ we can find an $\epsilon_d>0$ and open neighborhood $U_d$ of $d$ such that on $[0,\epsilon_d) \times U_d$ we can define a lift as follows: let $\pi: ||\mathcal T_\bullet[k]|| \to ||K_\bullet^{\delta,c}||$ be the map that forgets the $x \in \int^k_{M,\leq c}P$, then our lift will be given by sending $(t,d') \in [0,\epsilon_d) \times U_d$ to the pair $(F(t,d'),\pi \circ f(0,d')) \in ||\mathcal T_\bullet[k]|| \subset \int_{M,\leq c}^k P \times ||K_\bullet^{\delta,c}||$. Compactness of $D^k$ now gives the existence of a finite collection $\{U_{d_i}\}$ of $U_d$'s that cover $D^n$ and taking $\epsilon = \min_i\{\epsilon_{d_i}\}$ these lifts glue together to produce a lift on $[0,\epsilon) \times D^n$. This concludes the proof that the map $||\mathcal T_\bullet[k]|| \to \int^k_{M,\leq c} P$ is a microfibration.

The fiber of the map $||\mathcal T_\bullet[k]|| \to \int^k_{M,\leq c} P$ over $x$ is homeomorphic to $K^{\delta,c}_\bullet(L(x))$ and thus is at least $(\lfloor k/c\rfloor -2)$-connected by Theorem \ref{thmconndiskdim3}. Proposition \ref{lemmicrofibconn} now gives the result.\end{proof}

The next proposition identifies the space of $p$-simplices of $\mathcal T_\bullet[k]$.

\begin{proposition}\label{propd3respsimplices} The space $\mathcal T_p[k]$ is homotopy equivalent to \[\bigsqcup_{(\varphi_i)} \bigsqcup_{(k_i)} \left[ \left(\int^{k - \sum k_i}_M P\right) \times \prod_{i=0}^p P_{k_i} \right]\]
where $P_{k_i}$ denotes the component of $P$ corresponding to $k_i \in \{1,\ldots,c\}$ and the disjoint union is over the set of all $(p+1)$-tuples $(\varphi_0,\ldots,\varphi_p)$ of disjoint embeddings $(D^n_+,D^{n-1}) \to (\bar{M},\im(\psi))$ ordered by $\psi$ and over all ordered $(p+1)$-tuples $(k_0,\ldots,k_p)$ of integers $1 \leq k_i \leq c$. 
\end{proposition}

\begin{proof}We can write $\mathcal T_p[k]$ as a disjoint union over $\{\varphi_i\}$ and $\{k_i\}$ (as in the statement) of the space of elements of $\int^k_{M,\leq c} P$ consisting of labeled configurations which do not intersect the boundary of any embedded $\varphi_i$ and such that there is charge $k_i$ in the interior of the image of $\varphi_i$. By contracting all particles in the image of $\varphi_i$ to the center, which is possible because $k_i \leq c$ and $\varphi_i$ provides us with a parametrization of $\im(\varphi_i)$ by a disk, we retract onto $\int^{k-\sum k_i}_{M',\leq c} P \times \prod_{i=0}^p P_{k_i}$ with $M' = M \backslash \left(\bigsqcup_{i=0}^p \im(\varphi_i)\right)$. Next note that the closure $\bar{M}'$ of $M'$ is diffeomorphic to $M$ such that the inclusion $\iota$ of $\bar{M}'$ into $M$ is isotopic to this diffeomorphism by an isotopy $\iota_t$. This implies that $\int^{k-\sum k_i}_{M', \leq c} P$ is homotopy equivalent to $\int^{k-\sum k_i}_{M, \leq c} P$. More precisely, the homotopy inverse is given by pulling back an element of $\int^{k-\sum k_i}_{M, \leq c} P$ along $\iota_1$. Finally $\int^{k-\sum k_i}_{M, \leq c} P$ is homotopy equivalent to $\int^{k - \sum k_i}_{M} P$ by Corollary \ref{corleqcweq}.\end{proof}

Because we picked the embedding $\psi$ to have image disjoint from the embedding $\chi$ as used in Definition \ref{defstabilizationmap} of the stabilization map, the stabilization map $t: \int_M^k P \to \int_M^{k+1} P$ extends to a semisimplicial map:
\[t_\bullet: \mathcal T_\bullet[k] \to \mathcal T_\bullet[k+1]. \]

\subsection{A spectral sequence argument for dimension at least three}

Next we construct a relative resolution $C\mathcal T_\bullet[k]$. This is a pointed augmented semisimplicial space, given by $C\mathcal T_p[k] = \mr{Cone}(t: \mathcal T_p[k] \to \mathcal T_{p}[k+1])$.

\begin{proposition}\label{propd3resrelative} \begin{enumerate}[(i)]
\item The augmentation map 
\[||C\mathcal T_\bullet[k]|| \to \mr{Cone}\left(\int_M^k P \to \int_M^{k+1} P\right)\]
is homologically $(\lfloor (k+1)/c\rfloor-1)$-connected.
\item We have that $C\mathcal T_p[k]$ is homotopy equivalent as a based space to \[\bigvee_{(\phi_i)} \bigvee_{(k_i)} \left[ \mr{Cone} \left(\int^{k - \sum k_i}_M P \to \int^{k+1 - \sum k_i}_M P\right) \wedge \bigwedge_{i=0}^p (P_{k_i})_+ \right]\]
where the wedge is over the set of all $(p+1)$-tuples $(\varphi_0,\ldots,\varphi_p)$ of disjoint embeddings $(D^n_+,D^{n-1}) \to (\bar{M},\im(\psi))$ ordered by $\psi$ and over all ordered $(p+1)$-tuples $(k_0,\ldots,k_p)$ of integers $1 \leq k_i \leq c$. 
\end{enumerate}\end{proposition}

\begin{proof}Part (i) follows from the five lemma, which gives that the map is an isomorphism for $*<\min(\lfloor k/c\rfloor,\lfloor (k+1)/c\rfloor-1) = \lfloor (k+1)/c\rfloor-1$ and a surjection if $*=\lfloor (k+1)/c\rfloor-1$. Part (ii) follows directly from Propositions \ref{propd3resconn} and \ref{propd3respsimplices} respectively.\end{proof}

Associated to each pointed augmented semisimplicial space there is a reduced geometric realization spectral sequence (for example see Section 2.3 of \cite{RW}). We will apply this construction to $C\mathcal  T_\bullet[k]$.

\begin{theorem}\label{thmmaind3} Let $M$ be a connected manifold of dimension $\geq 3$ that is the interior of a manifold with boundary. The map $t_*: H_*(\int^k_M P) \to H_*(\int^{k+1}_M P)$ is an isomorphism for $* < \lfloor \frac{k}{2c} \rfloor$ and a surjection for $* = \lfloor \frac{k}{2c} \rfloor$.\end{theorem}

\begin{proof}The statement is equivalent to $\tilde{H}_*(\mr{Cone}(\int^k_M P \to \int^{k+1}_M P)) = 0$ if $* \leq \lfloor \frac{k}{2c} \rfloor$. We will prove this by induction on $k$, the case $k < 2c$ being trivial as both $\int^k_M P$ and $\int^{k+1}_M P$ are connected. Suppose we have proven it for $k' < k$, then we will prove it for $k$.

We will use the spectral sequence for the reduced geometric realization of a pointed augmented semisimplicial space, which by part (ii) of Proposition \ref{propd3resrelative} has an $E^1$-page given by
\[E^1_{pq} = \bigoplus_{(\phi_i)} \bigoplus_{(k_i)} \tilde{H}_q\left[ \mr{Cone} \left(\int^{k - \sum k_i}_M P \to \int^{k+1 - \sum k_i}_M P\right) \wedge \bigwedge_{i=0}^p \left( P_{k_i} \right)_+ \right]\]
and converges to
\[H_{p+q+1}\left(\mr{Cone}\left(||C\mathcal T_\bullet[k]|| \to \mr{Cone}\left(\int_M^k P \to \int_M^{k+1} P\right)\right)\right)\]
and thus to $0$ in the range $p+q \leq \lfloor (k+1)/c\rfloor-2$ by part (i) of Proposition \ref{propd3resrelative}.

If $k-(p+1)c \geq 0$ the stabilization maps induces a bijection on the indexing set of the direct sum in the $p$th column. For these $p$ the inductive hypothesis tells us that the group $E^1_{pq}$ is $0$ for $p \geq 0$ and $q \leq \lfloor \frac{k-(p+1)c}{2c} \rfloor$. On the other hand, for $p=-1$ we get $\tilde{H}_*(\mr{Cone}(\int^k_M P \to \int^{k+1}_M P))$. Since the spectral sequence converges to zero in the range $p+q \leq \lfloor \frac{k+1}{c} \rfloor-2$, this implies that $\tilde{H}_*(\mr{Cone}(\int^k_M P \to \int^{k+1}_M P)) = 0$ for $* \leq  \lfloor \frac{k-c}{2c} \rfloor$.

This completes the proof, unless $\lfloor \frac{k}{2c} \rfloor = \lfloor \frac{k-c}{2c} \rfloor + 1$. Let $K=\lfloor \frac{k}{2c} \rfloor$ In that case it suffices to prove that the $d^1$-differential $E^1_{0K} \to E^1_{-1K}$ is zero, because all the higher differentials that can hit $E^1_{-1K}$ start from groups that are zero. 

The relevant differential can be viewed as a map: \[\xymatrix{\bigoplus_{\phi_0} \bigoplus_{1 \leq k_0 \leq c}\tilde{H}_K\left(\mr{Cone}\left(\int^{k-k_0}_M P \to \int^{k+1-k_0}_M P\right)\right) \ar[d]_{d^1_{0q}} \\
 \tilde{H}_K\left(\mr{Cone}\left(\int^{k}_M P \to \int^{k+1}_M P\right)\right)}\]
It is induced by undoing the removal of the points in the interior of $\phi_0$. To make this precise, recall that the inclusion $\iota$ of closure of $M \backslash \im(\phi_0)$ into $M$ is isotopic to a diffeomorphism by an isotopy $\iota_t$. Then the map is induced by sending an element of $\int^{k-k_0}_M P$ to the element of $\int^{k}_M P$ by pulling back along $\iota_1$, and adding an element of $P_{k_0}$ at a point in the interior of $\im(\phi_0)$. We will see that this is essentially an iterated stabilization map $t^{k_0}$. We remark that one does not see the homology of $P_{k_0}$ here, because by vanishing of the reduced homology of the cone only $H_0(P_{k_0}) \cong \Z$ shows up and using K\"unneth we can ignore this.

Let us restrict our attention to one summand. By construction there is an isotopy from $\phi_0$ to $\chi$. Indeed, by assumption $\psi$ is isotopic to $\chi$ and by the boundary behavior of $\phi_0$ it is isotopic to $\psi$. This isotopy gives a homotopy from the map undoing the removal of the points in the interior of $\phi_0$ to the $k_0$-fold suspension. Up to homotopy $d^1_{0q}$ is the map induced on homology between cones of the horizontal maps in the diagram
\[\xymatrix{\int^{k-k_0}_M P \ar[r]^t \ar[d]_{t^{k_0}} & \int^{k+1-k_0}_M P \ar@{.>}[ld]_{t^{k_0-1}} \ar[d]^{t^{k_0}} \\
\int^k_M P \ar[r]_t & \int^{k+1}_M P}\]
where the square commutes, but the two triangles commute up to homotopy. Puppe sequence considerations imply that the map $\tilde{H}_K(\mr{Cone}(\int^{k-k_0}_M P \to \int^{k-k_0+1}_M P)) \to \tilde{H}_K(\mr{Cone}(\int^{k}_M P \to \int^{k+1}_M P))$ factors over $\tilde{H}_K(\Sigma (\int_M^{k-k_0} P)_+) = H_{K-1}(\int_M^{k-k_0})$, with maps determined by the Puppe sequence and the homotopies for the triangles respectively. Remark that the reduced homology group $\tilde{H}_K(\mr{Cone}(\int^{k}_M P \to \int^{k+1}_M P))$ is equal to the relative homology group $H_k(\int^{k+1}_M P,\int^k_M P)$. Then if we write the domain as $H_{K-1}(\int_M^{k-k_0} P)$ the map is the bottom composite of the following commutative diagram: \[\xymatrix{H_{K-1}(\int_M^{k-k_0-1} P) \ar[r] \ar[d] & H_K(\int_M^{k} P) \ar[rd]^0 \ar[d] & \\
H_{K-1}(\int_M^{k-k_0} P) \ar[r] & H_K(\int_M^{k+1} P) \ar[r] & H_K(\int_M^{k+1} P,\int_M^{k} P).}\] 

The diagonal map is always zero by construction and by induction the left vertical map is a surjection. This means that the bottom composite is zero and hence we are done.
\end{proof}

\subsection{The bounded charge resolution and spectral argument in dimension two} We will repeat the arguments of the previous two sections for a surface $\Sigma$ with boundary $\partial \Sigma$. Apart from a few technical details the proof is essentially the same.

Fix a connected surface $\Sigma$ with boundary $\partial \Sigma$, an embedding $\psi: (0,1) \to (\Sigma,\partial \Sigma)$ and a finite subset $V$ of $(0,1)$. Using this data we can define a simplicial space $K^c_\bullet$ with space of $p$-simplices the $(p+1)$-tuples of disjoint embeddings $(D^2_+,D^1) \to (\Sigma,\im(\psi))$ relative to $V$ or $(D^2_\vee,D^0) \to (\Sigma,\im(\psi))$ relative to $V$. For each $x \in \int_{\Sigma,\leq c}^k P$, the simplicial space $K^c_\bullet(L(x))$ is a subspace of the semisimplicial space of $K^c_\bullet$.

\begin{definition}Let $\mathcal T^2_\bullet[k]$ be the augmented semisimplicial space with space of $P$-simplices $\mathcal T^2_p[k]$ consisting of pairs of an $x \in \int^k_{\Sigma,\leq c} P$ and an element $(\phi_0,\ldots,\phi_p)$ of $K^{c}_p$ such that $(\phi_0,\ldots,\phi_p) \in K^c_p(L(x)) \subset K^c_p$. The augmentation is the forgetful map to $\int^k_\Sigma P$.\end{definition}

Exactly the same proof as Proposition \ref{propd3resconn} gives us the connectivity of the map induced by the augmentation:
\begin{proposition}The augmentation map $||\mathcal T^2_\bullet[k]|| \to \int_\Sigma^k P$ is $(k -1)$-connected.\end{proposition}

The analogue of Proposition \ref{propd3respsimplices} is slightly different, because we are now dealing with topologized instead of discrete embeddings.

\begin{proposition}We have that $\mathcal T^2_p[k]$ is homotopy equivalent to \[\bigsqcup_{([\gamma_i])} \bigsqcup_{(k_i)} \left[ \left(\int^{k - \sum k_i}_\Sigma P\right) \times \prod_{i=0}^p P_{k_i} \right]\]
where $P_{k_i}$ denotes the component of $P$ corresponding to $k_i \in \{0,1,\ldots,c\}$ and the disjoint union is over the elements of $\mathcal B_p(\Sigma,V,\emptyset)$ and over all ordered $(p+1)$-tuples $(k_0,\ldots,k_p)$ of integers $0 \leq k_i \leq c$ such that each arc $\gamma_i$ contains non-zero charge.
\end{proposition}

\begin{proof}The map that forgets everything about the configuration except how much charge is in each disk gives us a fibration with total space $\mathcal T^2_p[k]$ and base homotopy equivalent to the indexing set given in the proposition, by Remark \ref{remarcconnectedcomponents}. The fiber above a point representing an element of the indexing set is homeomorphic to $\left(\int^{k - \sum k_i}_{\Sigma,\leq c} P\right) \times \prod_{i=0}^p P_{k_i}$ by the same argument as in the proof of Proposition \ref{propd3respsimplices}; leave the configurations outside of the disks untouched and contract the configurations in the disks down to a fixed point in the disk.\end{proof}

After this modification, one just follows the proof for the case of dimension $\geq 3$. The definition of the relative resolution $C\mathcal T^2_\bullet[k]$ is again $C\mathcal T^2_p[k] = \mr{Cone}(t: \mathcal T^2_p[k] \to \mathcal T_p^2[k+1])$ and the proof of its properties is the same as in Proposition \ref{propd3resrelative}:

\begin{proposition} \begin{enumerate}[(i)]
\item The augmentation map 
\[||C\mathcal T^2_\bullet[k]|| \to \mr{Cone}\left(\int_\Sigma^k P \to \int_\Sigma^{k+1} P\right)\] is $(k-1)$-connected.
\item We have that $C\mathcal T^2_p[k]$ is homotopy equivalent as a based space to \[\bigvee_{([\gamma_i])} \bigvee_{(k_i)} \left[ \mr{Cone} \left(\int^{k - \sum k_i}_\Sigma P \to \int^{k+1 - \sum k_i}_\Sigma P\right) \wedge \bigwedge_{i=0}^p (P_{k_i})_+ \right]\]
where the wedge sum is over the set of $(p+1)$-simplices $([\gamma_0,\ldots,\gamma_p]) \in \mathcal B_p(\Sigma,\Delta,\emptyset)$ and over all ordered $(p+1)$-tuples $(k_0,\ldots,k_p)$ of integers $0 \leq k_i \leq c$ such that each arc $\gamma_i$ contains non-zero charge.
\end{enumerate}\end{proposition}

The relative spectral sequence argument is similarly modified and we obtain: 

\begin{theorem}\label{thmmaind2} Let $M$ be the interior of a connected orientable surface $\Sigma$ with boundary. The map $t_*: H_*(\int^k_M P) \to H_*(\int^{k+1}_M P)$ is an isomorphism for $* < \lfloor \frac{k}{2c} \rfloor$ and a surjection for $* = \lfloor \frac{k}{2c} \rfloor$.\end{theorem}

We can deduce a similar result for non-orientable surfaces from this. It will have a slightly worse range. Alternatively, one could try to redo our analysis of arc complexes in the non-orientable case.

\begin{corollary}Let $M$ be the interior of a connected non-orientable surface $\Sigma$ with boundary. The map $t_*: H_*(\int^k_M P) \to H_*(\int^{k+1}_M P)$ is an isomorphism for $* < \lfloor \frac{k}{4c} \rfloor$ and a surjection for an isomorphism for $* = \lfloor \frac{k}{4c} \rfloor$. \label{corunorientedstability} \end{corollary}

\begin{proof}[Sketch of proof] The proof will be a variation on handle induction. By classification of surfaces, any non-orientable surface with boundary is diffeomorphic to $\R P^2 \# (\#_g S^1 \times S^1)$ or $(\R P^2 \# \R P^2) \# (\#_g S^1 \times S^1)$ with a number of disks removed. We will give the proof in the first case, the second case being essentially the same except for the use of two one-sided arcs instead of one one-sided arc. 

Fix a point $x \in \partial \Sigma$, then using the $\R P^1$ in $\R P^2$ one can find a one-sided arc $\eta$ in $\Sigma$ starting and ending at $x$, such that removing $\eta$ from $\Sigma$ gives an orientable surface. Fix such an $\eta$, and fix a tubular neighborhood $U \cong \eta \times (-1,1)$.

Consider the subspace $X_k$ of $\int^k_{\Sigma,\leq c} P$ such that the total charge of points lying on the arc $\eta$ is $\leq \lfloor\frac{k}{4c}\rfloor$. Using transversality in the manifold $\mr{Sym}_k^{\leq c}(\Sigma)$ and the fact that perturbations in the macroscopic locations of points in $\int^k_{\Sigma, \leq c} P$ lift to $\int^k_{\Sigma, \leq c} P$, one can check that the inclusion 
\[X_k \hookrightarrow \int^k_{\Sigma, \leq c} P\]
is at least $\lfloor\frac{k}{4c}\rfloor$-connected. 

Next one resolves $X_k$ by a semisimplicial space $\mathcal X_\bullet[k]$ with $p$-simplices equal to pairs of a point in $X_k$ and a $(p+1)$-tuple of real numbers $a_i \in (0,1)$ (in discrete topology) such that $0 < a_0 < \ldots < a_p < 1$ and the strip $\eta \times (a_i,-a_i)$ contains at most charge $\lfloor\frac{k}{4c}\rfloor$. This has an augmentation to $X_k$ by forgetting everything except the points in $X_k$ and the induced map $||\mathcal X_\bullet[k]|| \to X_k$ will be a weak equivalence.

If one does the same for $\int^{k+1}_{\Sigma,\leq c} P$, one will obtain a map of augmented semisimplicial spaces $\mathcal X_\bullet[k] \to \mathcal X_\bullet[k+1]$, with  $\lfloor\frac{k}{4c}\rfloor$-connected augmentations to $\int^k_\Sigma P$ and $\int^{k+1}_\Sigma P$ respectively. Running for both an augmented realization spectral sequence gives a map of spectral sequences computing $t_*: \int^k_\Sigma P \to \int^{k+1}_\Sigma P$ for $* \leq \lfloor\frac{k}{4c}\rfloor$. On the $E^1$-page one will see the homology of a disjoint union of terms of the following form: a product of $\int^{\kappa}_{\Sigma \backslash U} P$ for some number $\kappa \geq \lfloor\frac{k}{4c}\rfloor$ with $p$ copies of topological chiral homology of a pair of strips and a copy of topological chiral homology of a single strip. The induced map between $E^1$-pages will be the stabilization map on the homology of the term $\int^{\kappa}_{\Sigma \backslash U}$ and the identity on the other terms. Since $\Sigma \backslash U$ is an orientable connected surface with boundary, the map is an isomorphism on the $E^1$-page in the range $q \leq \lfloor\frac{k}{4c}\rfloor$. This proves the corollary.\end{proof}

\subsection{Generalization to bounded exceptional points}One can generalize theorems \ref{thmmaind3} and \ref{thmmaind2} to spaces of configurations that have an additional set of allowed collections of exceptional particles of bounded charge. This will imply some cases of Conjecture \ref{conjf}, and Theorem \ref{thmmain} is the case where this set consists of the empty collection. More precisely, we consider a partial framed $E_n$-algebra $P$ with $\pi_0(P) = \{1,\ldots,c,c+1,\ldots, C\}$. 

\begin{definition}
An \emph{exceptional collection} $\Lambda$ is a collection of subsets of $\mathcal P(\{c+1,\ldots,C\})$ where $\mathcal P$ denotes the power set.\end{definition} 

The space $\int^{k,\Lambda}_M P$ is the subspace of $\int_M^k P$ such that the set of particles of charge $> c$ has set of charges lying in $\Lambda$. When $M$ admits boundary, there is again a stabilization map and the induced map $t_*: H_*(\int^k_M P) \to H_*(\int^{k+1}_M P)$ will be an isomorphism in a range.

Fix an exceptional collection $\Lambda$ with charge $|\Lambda| := \max_{\lambda \in \Lambda}(\sum_{c' \in \lambda} c')$. Then one can generalize the argument of Subsection \ref{subsecalttch} and the previous two arguments to the following theorem:

\begin{theorem}\label{thmmainexcep}  Let $M$ be a connected manifold of dimension $\geq 2$ admitting a boundary, oriented if it is of dimension $2$. Suppose $P$ is a partial framed $fF_n$-algebra with $\pi_0(P) = \{1,\ldots,c,c+1,\ldots, C\}$ and $\Lambda$ an exceptional collection. Then the map 
\[t_*: H_*\left(\int^{k,\Lambda}_M P\right) \to H_*\left(\int^{k+1,\Lambda}_M P\right)\]
is an isomorphism for $* < \lfloor \frac{k-|\Lambda|}{2c} \rfloor$ and a surjection for $* = \lfloor \frac{k-|\Lambda|}{2c} \rfloor$.\end{theorem}

\begin{proof}[Sketch of proof]Using arguments similar to those of Subsection \ref{subsecalttch}, one proves that there is a model of topological chiral homology that maps to the subspace of the symmetric power of points of charge $\leq c$ and some additional higher charge points lying in the exceptional collection $\Lambda$. We want to use a similar resolution and spectral sequence argument. To do so, we need to make sure we do not try to remove particles labeled in the exceptional collection. This uses the same complexes studied above, denoted $K^{\delta,c}_\bullet(L(x))$ in dimension $\geq 3$ and $K^{c}_\bullet(L(x))$ in dimension $2$, consisting of disk containing charge $\leq c$. The only difference is that there now possibly some particles of larger charge around, but the same arguments show that the complexes are still highly connected. This gives the desired resolution and the spectral sequence argument can be run as before.\end{proof}

We will finally give the proof of Corollary \ref{corravimelanie}.

\begin{proof}[Proof of Corollary \ref{corravimelanie}]Take $P = \{1,\ldots,\ldots,(m+1)(c+1)-1\}$ with exceptional collection $\Lambda \subset \mathcal P(\{c+1,\ldots,(m+1)(c+1)-1\})$ consisting of all sets of charge $< (m+1)(c+1)$. This satisfies $|\Lambda| = (m+1)(c+1)-1$. Note that \[\int^{k+(c+1)(m+1),\Lambda}_M P =W_{1^k \lambda}(M)\]
so that the result follows by Theorem \ref{thmmainexcep}, except for injectivity in the highest degree. To prove this, note that it is possible to define a transfer map $\tau: H_*(\int_M^{k,\Lambda} P) \to H_*(\int_M^{k-1,\Lambda} P)$ by summing over all ways of decreasing the charge by one at some particle in the configuration, see Section 5 of \cite{KMT}. An argument similar to that of \cite{Do} or \cite{N}, shows that the stabilization map is injective on homology.\end{proof}

\section{The scanning map and closed manifolds} \label{secscanningequiv} The goal of this section is to prove that the scanning map $s: \int^k_M P \to \Gamma^k_{M}(B^{TM} P)$ is a homology equivalence in a range tending to infinity with $k$, both for open and closed manifolds $M$. Before we can prove this, we discuss some technical details regarding homology fibrations.

\subsection{Homology fibrations} Homology fibrations were introduced in \cite{Mc1} and are a homological version of a quasi-fibration. They are useful because they are more general than fibrations but one can still use the Serre spectral sequence to study the homology of the total space of a homology fibration. 

\begin{definition}
A map $r : E \to B$ is a \emph{homology fibration} if the inclusions of every fiber of $r$ into the corresponding homotopy fiber are homology equivalences. 

We call $r$ a \emph{homology fibration in a range $* \leq L$} if the inclusion of every fiber of $r$ into the corresponding homotopy fiber induces an isomorphism on homology groups $H_*$ for $* < L$ and surjection for $*=L$. In other words, the relative homology vanishes for $* \leq L$. \label{homfib}
\end{definition}

McDuff and Segal gave the following useful criterion for proving that a map is a homology fibration.

\begin{theorem}[McDuff, Segal] \label{fibcr}

A map $r:E \to B$ is a homology fibration if we have a filtration $B=\bigcup B_j$ with each $B_i$ closed such that:
\begin{enumerate}[(i)]
\item Each $b\in B$ has a basis of contractible neighborhoods $U$ such that any deformation retraction of $U$ onto $b$ lifts to a deformation retraction of $r^{-1}(U)$ into $r^{-1}(b)$.
\item Each $r:r^{-1}(B_j\backslash B_{j-1}) \to B_j\backslash B_{j-1}$ is a fibration.
\item For each $j$, there is an open subset $U_j$ of $B_j$ such that $B_{j-1} \subset U_j$, and there are homotopies $h_t : U_j \to U_j$ and $H_t : r^{-1}(U_j) \to r^{-1}(U_j)$ satisfying:
\begin{enumerate}[(a)]
\item $h_0= \mr{id}$, $h_t(B_{j-1}) \subset B_{j-1}$ and $h_1(U_j) \subset B_{j-1}$,
\item $H_0 = \mr{id}$ and $r \circ H_t =h_t \circ r$,
\item $H_1:r^{-1}(b) \to r^{-1}(h_1(b))$ induces an isomorphism on homology for all $b \in U_j$.
\end{enumerate} 
\end{enumerate}

\label{DoldMcDuff}

\end{theorem}

\begin{remark} \label{notneeded} The original version of this result, Proposition 5.1 of \cite{Mc1}, had the following additional conditions:
\begin{enumerate}[(i)]
\item  All spaces $B_j$, $B_j\backslash B_{j-1}$, $r^{-1}(B_j)$ and $r^{-1}(B_j\backslash B_{j-1})$ have the homotopy type of CW-complexes.
\item Each $B_j$ is uniformly locally connected.
\end{enumerate}

These are unnecessary. They were used in \cite{Mc1} to prove that homology fibrations can be glued together. This is obvious from the more local alternate definition of homology fibration given in \cite{MSe}. The proof that these two definitions are equivalent given in \cite{MSe} uses an additional paracompactness assumption, but does not use that any spaces are homotopic to CW-complexes or uniformly locally connected. Additionally, on page 56 of \cite{Se}, Segal notes that this paracompactness assumption is also unnecessary. This is also discussed in the proof of Lemma 3.1 of \cite{Mc3}. However, condition (i) is needed, as it is used to prove that maps satisfying Definition \ref{homfib} are also homology fibrations in the local sense given in  \cite{MSe}. See also \cite{jeremymartinhomfib}.
\end{remark}

\begin{remark}\label{inrange}
The proof can be modified to give a criterion for a map being a homology fibration in a range by replacing (iii)(c) with homology equivalence in a range, see page 51 of \cite{Se}.
\end{remark}

\subsection{Scanning for closed manifolds}

Let $M$ be a connected $n$-dimensional manifold. Let $P$ be a partial $fF_n$-algebra with $\pi_0(P)=\{1, \ldots, c\}$ as a partial monoid under addition. Let $B^{TM} P$ be the $B^n P$ bundle described in Definition \ref{BTM} and let $s:\int_M P \to \Gamma_{M}(B^{TM} P)$ be the scanning map. We will often restrict the scanning map to the subspace $\int_{M,\leq c} P$, and this restriction will also be denoted $s$.

Note that $\pi_0(\Gamma_{M}(B^{TM}P))=\Z$, because $\Omega^n B^n P$ is a group completion of $P$ and the Grothendieck group of the monoid completion of $\{1, \ldots, c\}$ is $\Z$. Let $\Gamma^k_{M}(B^{TM} P)$ denote the degree $k$ component. First we prove that the scanning map is a homology isomorphism in a range for manifolds admitting boundary.

\begin{proposition}\label{propscanningopen}
Let $M$ be a connected manifold of dimension $n \geq 2$ admitting boundary, oriented if it is of dimension $2$. The map $s_*: H_*(\int^k_M P) \to H_*(\Gamma^k_{M}(B^{TM} P))$ is an isomorphism for $* < \lfloor \frac{k}{2c} \rfloor$ and a surjection for $* = \lfloor \frac{k}{2c} \rfloor$.
\end{proposition}

\begin{proof}Let $T: \Gamma^k_{M}(B^{TM} P) \to \Gamma^{k+1}_{M}(B^{TM} P)$ be the corresponding stabilization map which brings in a degree one section near the boundary as in Definition \ref{defstabilizationmapsections}. Note that $T$ is a homotopy equivalence. Consider the following commuting diagram: \[\xymatrix{\int^{k}_{M} P \ar[d]_{s} \ar[r]^{t} & \int^{k+1}_{M} P \ar[d]^{s}\\
\Gamma^k_{M}(B^{TM} P) \ar[r]_(.48){T} & \Gamma^{k+1}_{M}(B^{TM} P)}\] By Theorem \ref{thmscanning}, $s$ induces a homology equivalence on homotopy colimits
\[\mr{hocolim}\, s: \mr{hocolim}_t\, \int^{j}_M P \to \mr{hocolim}_T\, \Gamma_{M}^{j}(B^{TM} P).\]  Since each map $T$ is a homotopy equivalence, we have that $\mr{hocolim}_T\, \Gamma_{M}^{j}(B^{TM} P) \simeq \Gamma_{M}^k(B^{TM} P)$ for any choice of $k$. Since the maps $t$ are homology equivalences in ranges, the inclusion $\int_M^k P \to \mr{hocolim}_t \int^{j}_M P$ induces an isomorphism on homology for $* < \lfloor \frac{k}{2c} \rfloor$ and a surjection for $* = \lfloor \frac{k}{2c} \rfloor$. Thus the claim follows. 
\end{proof}

At this point we replace $\int_M P$ with $\int_{M,\leq c} P$. Pick a Riemannian metric on the manifold $M$ such that some point $m \in M$ has injectivity radius $2$. Let $B_s(m)$ denote the ball of radius $s$ around the point $m$.  Let $d_m: \int_{M,\leq c} P \to \R_{>0}$ be the function given by:
\[d_m = \min(1,\inf \{s > 0\,|\,\text{$B_s(m)$ contains charge $c+1$}\}).\]
This function is always positive since at most charge $c$ is macroscopically located at $m$. Let $E_k \subset \int_{M,\leq c}^{k} P$ be the subspace of configurations where at most charge $c$ is contained in $B_1(m)$. Using the function $d_m$ one can construct a deformation retraction of $\int_{M,\leq c}^k P$ onto $E_k$ as follows. Let $\Upsilon_t: M\to M$ for $t \in (0,1]$ be a family of diffeomorphisms such that $\Upsilon_t$ restricts to a diffeomorphism between $B_t(m)$ and $B_1(m)$. A retraction of $R:\int_{M,\leq c} P \to E_k$ is given by the formula $R(\xi) = (\Upsilon_{d_m(\xi)})_*(\xi)$ with $(\Upsilon_{t})_* :\int_{M,\leq c} P \to \int_{M,\leq c} P$ the map induced by $\Upsilon_{t}$. The interpolation $s \mapsto  (\Upsilon_{(1-s)+sd_m(\xi)})_*(\xi)$  gives a homotopy between $R$ and the identity.

Recall that we defined $B^n P$ to be $\smash{\int_{(D^n,S^{n-1})} P}$, see Theorem \ref{thmsalvgroupcompl}. Let $B \subset B^n P$ be the image of $\int_{D^n}^c P$. One can construct a deformation retraction of $B^n P$ onto $B$ by first using arguments similar to those in Corollary \ref{corleqcweq} and a function similar to $d_m$.  Let $r: E_k \to B$ be the function that sends a configuration $x$ to the portion of $x$ to $x \cap B_1(m)$, which we can identify with a point in $B$ after fixing a diffeomorphism between $B_1(m)$ and the standard unit ball.

\begin{lemma} \label{projhomfib}
Let $M$ be a connected manifold of dimension $n \geq 2$, oriented if it is of dimension $2$. Then the map $r: E_k \to B$ is a homology fibration in the range $* \leq \lfloor \frac{k-c}{2c} \rfloor$.
\end{lemma}
\begin{proof}
We will prove that $r$ is a homology fibration in the range $* < \lfloor \frac{k-c}{2c} \rfloor$ by using Theorems \ref{thmmain} and \ref{fibcr} along with Remark \ref{inrange}. Condition (i) follows by an argument similar to that used in the proof of Lemma 4.1 of \cite{Mc1}. For $j \leq c$, let $B_j$ be the image of $\int_{D^n}^j P \to B$. For $j>c$, let $B_j=B$. Condition (ii) follows because $r:r^{-1}(B_j \backslash B_{j-1}) \to B_j \backslash B_{j-1}$ is a trivial fiber bundle. 

Let $U_j$ be the subspace of $B_j$ consisting of configurations with charge at most $j-1$ in $B_{1/2}(m)$.  Let $\phi_t:M \to M$ be an isotopy through diffeomorphisms with the following properties: \begin{enumerate}[(i')]
\item $\phi_0=id$,
\item for each $t$, $\phi_t(M \backslash B_1(m)) \subset M \backslash B_1(m)$,
\item $\phi_1$ restricts to a diffeomorphism of $B_{1/2}(m)$ to $B_1(m)$.
\end{enumerate}

Let $h_t:B_j \to B_j$ and $H_t: r^{-1}(B_j) \to r^{-1}(B_j)$ be the maps induced by applying $\phi_t$ to the macroscopic locations of each point in the configuration. By construction, $h$ and $H$ satisfy conditions (iii)(a) and (iii)(b). They satisfy condition (iii)(c) in the range $* \leq \lfloor \frac{k-c}{2c} \rfloor$ since the maps induced between fibers are homotopic to compositions of stabilization maps 
\[\int_{M \backslash B_1(m),\leq c}^l P \to \int_{M \backslash B_1(m),\leq c}^{l+1} P\]
with $l \geq k-c$. Thus, $r: E_k \to B$ is a homology fibration in the range $* \leq \lfloor \frac{k-c}{2c} \rfloor$.\end{proof}

We now prove that the scanning map is a homology equivalence in a slightly lower range than the stabilization map for manifolds admitting boundary.

\begin{theorem}\label{thmscanisoclosed} Let $M$ be a closed connected manifold of dimension $n \geq 2$, oriented if it is of dimension $2$. The map $s_*: H_*(\int_M^k P) \to H_*(\Gamma^k_{M}(B^{TM} P))$ is an isomorphism for $*  < \lfloor \frac{k-c}{2c} \rfloor$ and surjective for $* = \lfloor \frac{k-c}{2c} \rfloor$.
\end{theorem}

\begin{proof} Again we can assume we are working with $\int_{M,\leq c}^k P$ instead of $\int_M^k P$ and it suffices to prove $s : E_k \to \Gamma^k_M(B^{TM} P)$ is a homology equivalence in the above range. Consider the following commuting diagram:  \[\xymatrix{E_k \ar[d]_{r} \ar[rr]^{s} & & \Gamma^k_M(B^{TM} P) \ar[d]^{R}\\
B \ar[rr]_(.3){s} & & \Gamma_{(M,M \backslash B_1(m))}(B^{TM} P)}\]  The fiber of the rightmost map over a configuration with charge $l$ macroscopically located in $B_1(m)$ is homeomorphic to $\int^{k-k'}_{M-B_1(m),\leq c} P$. Note that $k'$ is at most $c$. The fibers of the rightmost map are homotopic to $\Gamma_{M-B_1(m)}^{k'}(M)$. The right map is a fibration and the left map is a homology fibration in the range $* \leq \lfloor \frac{k-c}{2c} \rfloor$. Furthermore, the bottom map is a homotopy equivalence by Theorem \ref{pair}. Since the inclusion of the fibers of the left map into the fibers of the right map are homology equivalences in the range $* \leq \lfloor \frac{k-c}{2c} \rfloor$ by Proposition \ref{propscanningopen}, the claim follows by comparing the Serre spectral sequence for the fibration on the right with the Serre spectral sequence for a fibrant replacement of the map on the left.\end{proof}

\section{Rational homological stability for bounded symmetric powers of closed manifolds} \label{secrathomstabtrunc}
Recall that the bounded symmetric power $\mr{Sym}^{\leq c}_k(M)$ was defined to be the subspace of $\mr{Sym}_k(M) = M^k/\mathfrak S_k$ consisting of unordered $k$-tuples $\{m_1,\ldots,m_k\}$ of points in $M$ such that no more than $c$ points coincide. As noted in the introduction and Example \ref{examplecompletionEn}, homological stability for bounded symmetric powers $\mr{Sym}_k^{\leq c}(M)$ of manifolds admitting boundary is an immediate corollary of the main theorem.  This is because the space $\mr{Sym}_k^{\leq c}(M)$ is homotopy equivalent to $\int_M P$ for $P=\{1,\ldots,c\}$ (in fact homeomorphic to $\int_{M,\leq c} P$) and thus the stabilization map \[t_*: H_*(\mr{Sym}_k^{\leq c}(M)) \to H_*(\mr{Sym}_{k+1}^{\leq c}(M))\]
induces an isomorphism on homology groups for $* < \lfloor \frac{k}{2c} \rfloor$ and a surjection for $* = \lfloor \frac{k}{2c} \rfloor $. However, a transfer argument as in \cite{Do} or \cite{N} shows that the stabilization map is always injective, so that for bounded symmetric powers the stabilization map is an isomorphism for $* \leq \lfloor \frac{k}{2c} \rfloor$. By the universal coefficients theorem, the same is true for homology with rational coefficients. In this section, we will prove a similar statement for closed manifolds; their bounded symmetric powers satisfy rational homological stability in the range $* < \lfloor \frac{k-c}{2c} \rfloor$. This answers part of Conjecture F of \cite{VW}.

When $M$ is closed, one cannot define a stabilization map as there is no direction from which to bring in new particles. However, by the results of the previous section, the scanning map $s:\mr{Sym}_k^{\leq c}(M) \to \Gamma^k_{M}(B^{TM} \mr{Sym}^{\leq c}(\R^n))$ is a homology equivalence in a range tending to infinity with $k$. Here $\mr{Sym}^{\leq c}(\R^n)$ denotes the $fF_n$-algebra $\bigsqcup_{k \geq 0} \mr{Sym}_k^{\leq c}(\R^n)$. Thus, to prove rational homological stability, it suffices to compare the rational homotopy types of the components of $\Gamma_{M}(B^{TM} \mr{Sym}^{\leq c}(\R^n))$. We begin by analyzing the rational homotopy type of $B^n \mr{Sym}^{\leq c}(\R^n)$ and of the topological monoid of rational homotopy automorphisms of $B^n \mr{Sym}^{\leq c}(\R^n)$.

\subsection{The rational cohomology and homotopy type of $B^n \mr{Sym}^c(\R^n)$}
In this section we compute a minimal model for $B^{n} \mr{Sym}^{\leq c}(\R^n)$. We let $F(X,Y)$ denote the space of unbased continuous maps from $X$ to $Y$ in the compact-open topology and $F_1(X,X)$ is the identity component, which is a topological monoid under composition. We use this minimal model to compute the rational connectivity of $F_1(B^n \mr{Sym}^{\leq c}(\R^n),B^n \mr{Sym}^{\leq c}(\R^n))$.

\begin{lemma}We have that $B^n \mr{Sym}^{\leq c}(\R^n) \simeq \mr{Sym}_c(S^n) = (S^n)^c/{\mathfrak S}_c$.\end{lemma}

\begin{proof}
The subspace $B \subset B^n \mr{Sym}^{\leq c}(\R^n)$ of Lemma \ref{projhomfib} is homeomorphic to $\mr{Sym}_c(S^n)$. This is induced by a homeomorphism between $S^n$ and $D^n/S^{n-1}$.
\end{proof}

Let $\dot TM$ denote the fiberwise one-point compactification of the tangent bundle and let $\mr{Sym}_c(\dot TM)$ denote its fiberwise symmetric $c$-fold symmetric power. By the previous lemma, the bundle $B^{TM} \mr{Sym}^{\leq c}(\R^n)$ is fiberwise homotopy equivalent to $\mr{Sym}_c(\dot TM)$. 

Since $\N$ is commutative, there is a natural stabilization map $t: \mr{Sym}_c(S^n) \to \mr{Sym}_{c+1}(S^n)$. More concretely, this is given by choosing a base point $S^n$ and adding an extra point to configurations at this base point. If there already are points at that location, simply increase the label by one. These maps are injective on homology (see Theorem 1 and Lemma 2 of \cite{Do}, or \cite{N}). Also note that $\mr{colim}_t\, \mr{Sym}_{c}(S^n) \simeq K(\Z,n)$ by the Dold-Thom Theorem \cite{DT} and $\mr{Sym}_{1}(S^n) \cong S^n$. 

Note that for any $k \geq 1$ we have that $\mr{Sym}_k(S^1) \simeq S^1$. Thus, we may assume that $n>1$ and that the spaces $B^{n} \mr{Sym}^{\leq c}(\R^n)$ are nilpotent. Hence their localizations are well behaved. We denote the rationalization of a space $X$ by $X_\Q$.

\begin{theorem}
If $n$ is odd, $B^n \mr{Sym}^{\leq c}(\R^n)_{\Q} \simeq S^n_{\Q}$. If $n$ is even, \[H^*(B^n \mr{Sym}^{\leq c}(\R^n);\Q)=\Q[\iota_n]/(\iota_n^{c+1})\]
with $\iota_n$ in homological dimension $n$.
\end{theorem}

\begin{proof}
We may assume that $n>1$, so that all relevant spaces are simply connected. Thus, to prove that a map is a homotopy equivalence, it suffices to prove that it is a homology equivalence. Let $n$ be odd. The rational localization of odd-dimensional spheres are Eilenberg-MacLane spaces. Since the  maps 
\[S^n \to \mr{Sym}_k (S^n) \to \mr{Sym}_{\infty}(S^n) \simeq K(\Z,n) \simeq_{\Q} S^n\] 
are all injective on homology, they are all rational homotopy equivalences. 

Now let $n$ be even. Note that $K(\Z,n) \simeq_{\Q} \Omega S^{n+1}$ so its rational homology is a polynomial algebra (with respect to the Pontryagin product) on a class $x_n$ in homological dimension $n$. Since $\N$ is $E_{\infty}$, $\int_{S^n} \N=\bigsqcup \mr{Sym}_k(S^n)$ is an homotopy commutative $H$-space. More concretely this $H$-space structure is induced by viewing $\bigsqcup \mr{Sym}_k(S^n)$ as the free abelian monoid on $S^n$. The stabilization map is induced by multiplication with an element of $\int_{S^n}^1 \N$. The map $\mr{Sym}_c (S^n) \to K(\Z,n)$ is injective on homology and $\int_{S^n} \N \to K(\Z,n)$ is a map of $H$-spaces. We can assume that the fundamental class of $S^n$ is mapped to $x_n$ of $K(\Z,n)$ so the image in homology of $\mr{Sym}_c (S^n) \to K(\Z,n)$ contains the first $c$ powers of the $x_n$. Since $\mr{Sym}_c (S^n)$ is $cn$-dimensional, this is the entire image in rational homology. This proves the claim additively. Note that $H^*(K(\Z,n);\Q)=\Q[\iota_n]$ with the class $\iota_n$ a degree $n$ class called the fundamental class. The cohomology class $\iota_n$ is dual to the fundamental class, which was mapped to $x_n$. Since $\mr{Sym}_c (S^n) \to K(\Z,n)$ is injective in homology, the map in cohomology is surjective. This forces $\Q[\iota_n]/(\iota_n^{c+1})$ to be the ring structure of the rational cohomology of $\mr{Sym}_c (S^n)$.\end{proof}

Recall that $F(X,Y)$ denotes the unbased mapping space and let $[X,Y]$ denote the set of homotopy classes of unbased maps. For spaces $X,Y$ with a natural identification $H_n(X;\Q)=H_n(Y;\Q)=\Q$, let $F_{\lambda}(X,Y)$ denote union of connected components of $F(X,Y)$ consisting of maps $f$ such that $H_n(f)$ is multiplication by $\lambda \in \Q$. By Example 2.5 of M{\o}ller-Rausen \cite{MR}, for $n$ odd and $\lambda \in \Q$, \[F_\lambda(S^n_{\Q},S^n_{\Q}) \simeq S^n_\Q.\] Thus, these spaces are highly connected and this connectivity increases with $n$.

We will next investigate the even case. For $n=2$ we have that $B^n \mr{Sym}^{\leq c}(\R^n) = \C P^c$ \cite{K2} and Example 3.4 of M{\o}ller-Rausen shows that $F_\lambda(\C P^c_{\Q},\C P^c_\Q)$ is always connected and simply-connected if $\lambda \neq 0$. We want to generalize this to higher dimensions. Let $n$ be even and for convenience set $X_n = B^n \mr{Sym}^{\leq c}(\R^n)_{\Q}$. Our goal is to prove that the components of $F(X_n,X_n)$ are highly connected. When $M$ is orientable, the fiberwise rational localization of $B^{TM} \mr{Sym}^{\leq c} (\R^n)$ is classified by a map to $BF_1(X_n,X_n)$, as the latter is the classifying space for oriented fibrations with fiber $X_n$. Here $BF_1(X_n,X_n)$ is the classifying space of the component $F_1(X_n,X_n)$ of the mapping space containing the identity component. Thus, proving that this space is highly connected will in particular imply that the fiberwise rational localization of $B^{TM} \mr{Sym}^{\leq c} (\R^n)$  is a trivial bundle. Let CDGA stand for commutative differential graded algebra.

\begin{lemma}A minimal model of $X_n$ is given by the CDGA $(\Q[a_n,b_{(c+1)n-1}],d)$ with $d(a_n) = 0$ and $d(b_{(c+1)n-1}) = a_n^{c+1}$.\end{lemma}

\begin{proof}Since $H^*(X_n;\Q) = \Q[\iota_n]/(\iota_n^{c+1})$ there is a representative cochain $a_n$ of the cohomology class $\iota_n$. This gives a map of CDGAs $(\Q[\iota_n],d=0) \to C^*(X_n;\Q)$. Because $\iota_n^{c+1}$ is zero, we have that $a_n^{c+1} = db_{(c+1)n-1}$ for some cochain $b_{(c+1)n-1}$. Thus we can extend our map of CDGAs to a map $(\Q[a_n,b_{(c+1)n-1}],d) \to C^*(X_n;\Q)$ with the differential in the former given by $d(a_n) = 0$ and $d(b_{(c+1)n-1}) = a_n^{c+1}$. Here $\Q[a_n,b_{(c+1)n-1}]$ denotes the free graded commutative algebra on the two classes and the subscripts indicate homological dimension. It is easy to check that this map is a quasi-isomorphism and that the former is minimal.\end{proof}

\begin{proposition} Let $n$ be even. All components $F_\lambda(X_n,X_n)$ are $(n-2)$-connected and if $\lambda \neq 0$ are $(2n-2)$-connected. \label{maphighlyconnected}

\end{proposition}

\begin{proof}The minimal model implies that $X_n$ is the homotopy fiber of the map $\phi: K(\Q,n) \to K(\Q,(c+1)n)$ representing the element $\iota_n^{c+1}$ with $\iota_n \in H^n(K(\Q,n))$ a generator. That is, $\phi^*(\iota_{(c+1)n}) = \iota_n^{c+1}$. This implies that $F_\lambda(X_n,X_n)$ is the homotopy fiber of $F_\lambda(X_n,K(\Q,n)) \to F(X_n,K(\Q,(c+1),n))$.

For $i \geq 1$ the $i$th homotopy groups of $F_\lambda(X_n,K(\Q,n))$ and $F(X_n,K(\Q,(c+1)n))$ are $H^{n-i}(X_n;\Q)$ and $H^{(c+1)n-i}(X_n;\Q)$ respectively. The long exact sequence of homotopy groups for a fibration and our calculation of the rational cohomology of $X_n$ immediately proves the first claim, i.e. that $F_\lambda(X_n,X_n)$ is $(n-2)$-connected.

Let $\varphi: F_\lambda(X_n,K(\Q,n)) \to F(X_n,K(\Q,(c+1)n))$ be the map induced by postcomposition with $\phi$. To prove the second claim, it suffices to prove by Hurewicz that the map \[\varphi^*: H^n(F(X_n,K(\Q,(c+1)n));\Q) \to H^n(F_\lambda(X_n,K(\Q,n)))\]
is non-zero if $\lambda$ is non-zero. To do this consider first the evaluation map $e: F_\lambda(X_n,K(\Q,n)) \times X_n \to K(\Q,n)$. This satisfies $e^*(\iota_n) = b \otimes 1 + 1 \otimes \lambda$, where we have used that $F_\lambda(X_n,K(\Q,n)) = K(\Q,n)$ and $b$ is the generator in degree $n$ of the cohomology of this $K(\Q,n)$. The adjoint of the map is $\phi \circ e$. From this we conclude that 
\[(\phi \circ e)^*(\iota_{(c+1)n}) = (b \otimes 1 + 1 \otimes \lambda)^{c+1}.\] The arguments of Section 1.2 of \cite{haefliger} imply that $\varphi^*$ applied to the generator of $H^n(F(X_n,K(\Q,(c+1)n));\Q)$ is $((b \otimes 1 + 1 \otimes \lambda)^{c+1}) \cap \lambda^c = (c+1) b$. Note we have used that $\lambda \neq 0$.
\end{proof}

\subsection{Homological stability for bounded symmetric powers of closed manifolds} Our goal is prove homological stability for bounded symmetric powers of closed manifolds. Theorem \ref{thmscanisoclosed} implies that the scanning map 
\[s: \mr{Sym}_k^{\leq c}(M) \to \Gamma^k_{M}(\mr{Sym}_c(\dot{T}M))\]
is a homology equivalence for $* < \lfloor \frac{k-c}{2c} \rfloor$. Let $l:\mr{Sym}_c(\dot{T}M) \to \mr{Sym}_c(\dot{T}M)_{\Q}$ denote the fiberwise rational localization map (see \cite{Su} for a construction of fiberwise localization). By Theorem 5.3 \cite{Mo}, composition with $l$ induces an isomorphism on rational homology group $\Gamma^k_{M}(\mr{Sym}_c(\dot{T}M)) \to \Gamma^k_{M}(\mr{Sym}_c(\dot{T}M)_{\Q})$. Thus $\mr{Sym}_k^{\leq c}(M)$ has the same rational homology as $\Gamma^k_{M}(\mr{Sym}_c(\dot{T}M)_{\Q})$ through a range. Following the ideas of \cite{BMi}, we will prove rational homological stability for the spaces  $\mr{Sym}_k^{\leq c}(M)$ by proving that most components of $\Gamma_{M}(\mr{Sym}_c(\dot{T}M)_{\Q})$ have the same homotopy type. 

First we consider the case of odd-dimensional manifolds. The following is  Proposition 3.2 of \cite{BMi}.

\begin{proposition}[Bendersky-Miller] \label{oddComponents}
Let $M$ be an odd-dimensional manifold. The rational homotopy type of $\Gamma^k_{M}(\dot{T}M_{\Q})$ is independent of $k$.
\end{proposition}

This directly implies the following result.

\begin{theorem}\label{odd} Let $M$ be closed of odd dimension $n$, then we have isomorphisms for $* < \lfloor \frac{k-c}{2c} \rfloor$
\[H_*(\mr{Sym}_k^{\leq c}(M);\Q) \cong H_*(\mr{Sym}_{k+1}^{\leq c}(M);\Q)\]\end{theorem}

\begin{proof} Because $\mr{Sym}_c(S^n)_{\Q} \simeq S^n_{\Q}$ for $n$ is odd, we have that $\dot{T}M_{\Q}$ and $\mr{Sym}_c(\dot{T}M)_{\Q}$ are fiberwise homotopy equivalent. Proposition \ref{oddComponents} implies that all components of  $\Gamma_{M}(\mr{Sym}_c(\dot{T}M))_{\Q})$ are homotopy equivalent. We conclude that the scanning maps
\[s: \mr{Sym}_k^{\leq c}(M) \to \Gamma^k_{M}(\mr{Sym}_c(\dot{T}M)) \text{ and } s: \mr{Sym}_{k+1}^{\leq c}(M) \to \Gamma^{k+1}_{M}(\mr{Sym}_c(\dot{T}M))\]
are isomorphisms on homology in the range $* < \lfloor \frac{k-c}{2c} \rfloor$ and the codomains have isomorphic rational homology.\end{proof}

\begin{remark}\label{rembetterrangebound}
In fact one can show that the inclusion $\mr{Sym}^{\leq c}_k(M) \to \mr{Sym}_k(M)$ is a rational homology isomorphism for all $c$ if $M$ is odd-dimensional. To prove this, first note that a standard transfer argument shows that that the scanning map $\mr{Sym}^{\leq c}(\R^n) \to \Omega^n B^n \mr{Sym}^{\leq c}(\R^n)$ is injective in homology (see \cite{Do}, or \cite{N}). The reduced rational homology of each connected component of $\Omega^n B^n \mr{Sym}^{\leq c}( \R^n)$ is trivial and thus the same is true for the connected components of $\mr{Sym}^{\leq c}(\R^n)$. Therefore, the inclusion \[\mr{Sym}^{\leq c}(\R^n) \to \mr{Sym}(\R^n) \simeq \N \] is a rational homology equivalence. Using other models of topological chiral homology (see e.g. \cite{Fr2} or \cite{Mi2}), it is clear that a homology equivalence between framed $E_n$-algebras induces a homology equivalence on topological chiral homology. Thus \[\int_M \mr{Sym}^{\leq c}(\R^n) \simeq \mr{Sym}^{\leq c}_k(M) \to \mr{Sym}_k(M) \simeq \int_M \mr{Sym}(\R^n)\] is a rational homology equivalence. In \cite{Srod}, Steenrod proved integral homological stability for the spaces $\mr{Sym}_k(M)$ with a range $* \leq k$ and this range can be improved by the results of \cite{milgram} if $M$ is highly connected. Thus the spaces $\mr{Sym}^{\leq c}_k(M)$ have rational homological stability in this larger range when the dimension is odd. 
\end{remark}

We will now consider even-dimensional oriented closed manifolds. In Lemma 3.4 \cite{BMi} it was proven that if $M$ is orientable then $\dot{T}M_{\Q}$ is a trivial bundle. If $n$ is even, this choice of trivialization is unique up to homotopy. We will prove this implies that a compatible trivialization of $\mr{Sym}_c(\dot{T}M)_{\Q}$ exists.

\begin{lemma}If $M$ is orientable then $\mr{Sym}_c(\dot{T}M)_{\Q}$ is a trivial bundle and a trivialization can be chosen such that the following diagram commutes:
\[\xymatrix{\Gamma^k_{M}(\dot{T}M_{\Q}) \ar[d] \ar[r]^(.45)\simeq & F_{k-\chi(M)/2}(M,S^n_{\Q}) \ar[d]\\
\Gamma^k_{M}(\mr{Sym}_c(\dot{T}M)_{\Q}) \ar[r]_(.42)\simeq & F_{k-\chi(M)/2}(M,\mr{Sym}_c(S^n)_{\Q}).}\]
\label{chiovertwo}
\end{lemma}

\begin{proof} Consider the natural inclusion map $C_k(M) \to \mr{Sym}_k^{\leq c}(M)$. The following diagram commutes:
\[\xymatrix{C_k(M) \ar[d] \ar[r] & \Gamma^k_{M}(\dot{T}M) \ar[d] \\
\mr{Sym}_k^{\leq c}(M) \ar[r] & \Gamma^k_{M}(\mr{Sym}_c(\dot{T}M)).}\] The bundle $\dot{T}M$ is classified by a map $M \to BF_1(S^n,S^n)$, while the bundle $\mr{Sym}_c(\dot{T}M)$ is classified by the map $M \to BF_1(\mr{Sym}_c(S^n),\mr{Sym}_c(S^n))$ obtained as follows: one composes the previous map by the map induced on classifying spaces by the homomorphism $F_1(S^n,S^n) \to F_1(\mr{Sym}_c(S^n),\mr{Sym}_c(S^n))$ given by $f \mapsto \mr{Sym}_c(f)$.


We can fiberwise rationalize to extend the commutative diagram to
\[\xymatrix{C_k(M) \ar[d] \ar[r] & \Gamma^k_{M}(\dot{T}M) \ar[d] \ar[r] & \Gamma^k_{M}(\dot{T}M_{\Q}) \ar[d]\\
\mr{Sym}_k^{\leq c}(M) \ar[r] & \Gamma^k_{M}(\mr{Sym}_c(\dot{T}M)) \ar[r] & \Gamma^k_{M}(\mr{Sym}_c(\dot{T}M)_{\Q})}\]
and the fiberwise rationalizations are respectively classified by maps $M \to BF_1(S^n_{\Q},S^n_{\Q})$ and $M \to BF_1(\mr{Sym}_c(S^n)_{\Q},\mr{Sym}_c(S^n)_{\Q})$, obtained by applying the rationalization functor to the previous maps. The trivialization of $\dot{T}M_{\Q}$ is equivalent to the existence of a dotted lift up to homotopy in the top triangle: \[\xymatrix{ & \ast \ar[d] \\
M \ar[r]^{\dot{T}M_{\Q}} \ar[rd]_(.3){\mr{Sym}_c(\dot{T}M))_{\Q}} \ar@{.>}[ur] & BF_1(S^n_{\Q},S^n_{\Q}) \ar[d] \\
& BF_1(\mr{Sym}_c(S^n)_{\Q}),\mr{Sym}_c(S^n)_{\Q}) .}\]
Since the bottom triangle commutes, we automatically get a trivialization of $\mr{Sym}_c(\dot{T}M)_{\Q}$ from this as well. More importantly, these trivializations are compatible in the sense that the following diagram commutes:
\[\xymatrix{C_k(M) \ar[d] \ar[r] & \Gamma^k_{M}(\dot{T}M_{\Q}) \ar[d] \ar[r]^\simeq & F_{f(k)}(M,S^n_{\Q}) \ar[d]\\
\mr{Sym}_k^{\leq c}(M) \ar[r] & \Gamma^k_{M}(\mr{Sym}_c(\dot{T}M)_{\Q}) \ar[r]_\simeq & F_{f'(k)}(M,\mr{Sym}_c(S^n)_{\Q})}\]
where $f(k)$ and $f'(k)$ are functions to be determined, $F_l(X,Y)$ is the degree $l$ mapping space, for $l \in [X,Y]$ and the right-hand vertical map is induced by the map $S^n_{\Q} \to \mr{Sym}_c(S^n)_{\Q}$ induced by the inclusion $S^n \to \mr{Sym}_c(S^n)$. 

When $M$ is oriented of dimension $n$ and connected, $[M,S^n_{\Q}] \cong \Q$. The rational sphere $S^n_{\Q}$ has $H_n(S^n_{\Q}) = \Q$ with canonical generator $[S^n]$ coming from the fundamental class of $S^n$. An identification $[M,S^n_{\Q}] = \Q$ is given by sending a homotopy class $[g]$ to the rational number $k$ such that $g_*([M]) = k[S^n]$. Our calculation shows that the map $S^n_{\Q} \to \mr{Sym}_c(S^n_{\Q})$ is an isomorphism on $H_n$, so we can use the same generator to define degree here. This makes it clear that the map $S^n_{\Q} \to \mr{Sym}_c(S^n)_{\Q}$ induces the identity on the degrees $[M,S^n_{\Q}] = \Q \to \Q = [M,\mr{Sym}_c(S^n)_{\Q}]$. This calculation shows that $f(k) = f'(k)$. Furthermore, Proposition 3.5 of \cite{BMi} tells us that $f(k) = k-\chi(M)/2$.\end{proof}

Our computation of the connectivity of $BF_1(\mr{Sym}_c(S^n)_{\Q},\mr{Sym}_c(S^n)_{\Q})$ tells us that when $n$ is even, a trivialization of $\mr{Sym}_c(\dot{T}M)_{\Q}$ always exists and is unique up to homotopy by obstruction theory. The main point of the proof in the previous lemma is that this trivialization is compatible with the trivialization of $\dot{T}M_{\Q}$, allowing us to determine the bijection between the components of the section space and the components of the mapping space.

We will next construct bundle automorphisms of $\mr{Sym}_k^{\leq c}(M)_{\Q}$ which induce homotopy equivalences between connected components of $\Gamma_M(\mr{Sym}_k^{\leq c}(M))$.

\begin{lemma}
Let $E \to M$ be a $\mr{Sym}_c(S^n)_{\Q}$-bundle with $n$ even. For all $d \in \Q$ with $d \neq 0$, there exists a bundle map $f_{d}:E \to E$ which induces a degree $d$ map on each fiber. Moreover, this map is unique up to fiberwise homotopy.  
\label{mapexists}
\end{lemma}

\begin{proof}
Let $F_{d}(E,E) \to M$ be the bundle with fiber $F_{d}(E_m,E_m)$ over a point $m \in M$. A bundle map $f_{d}:E \to E$ inducing a degree $d$ map on each fiber can be viewed as a section of  $F_{d}(E,E)$. The obstructions to finding such a section lie in \[H^i(M,\pi_{i-1}(F_{d}(\mr{Sym}_c(S^n)_{\Q},\mr{Sym}_c(S^n)_{\Q})).\] For $i > n$, these groups vanish since $M$ is $n$-dimensional. For $i \leq n$, these groups vanish by Proposition \ref{maphighlyconnected} since $F_{d}(\mr{Sym}_c(S^n)_{\Q},\mr{Sym}_c(S^n)_{\Q})$ is $2n-2$-connected. A similar obstruction theory argument also shows uniqueness.\end{proof}

\begin{lemma}
Let $M$ be an even-dimensional manifold, $\xi \in \Gamma^k_{M}(\mr{Sym}_c(\dot{T}M)_{\Q})$, and $f_d: \mr{Sym}_c(\dot{T}M)_{\Q} \to \mr{Sym}_c(\dot{T}M)_{\Q}$ a fiberwise degree $d$ map. Then $f_d \circ \xi$ is degree $dk-(1-d)\chi(M)/2$. 
\end{lemma}

\begin{proof}
First assume $M$ is orientable. Let $\tau : \mr{Sym}_c(\dot{T}M)_{\Q} \to M \times \mr{Sym}_c(S^n)_{\Q} $ be a trivialization. Let $g_d : M \times \mr{Sym}_c(S^n)_{\Q} \to M \times \mr{Sym}_c(S^n)_{\Q}$ be a bundle map induced by a degree $d$ self map of  $\mr{Sym}_c(S^n)_{\Q}$. If $\xi \in F_k(M,\mr{Sym}_c(S^n)_{\Q})=\Gamma^k_{M}(M \times \mr{Sym}_c(S^n)_{\Q} )$, then the degree of $g_d \circ \xi$ is $dk$. By Lemma \ref{chiovertwo}, composition with $\tau$ gives a map: \[ \Gamma^k_{M}(\mr{Sym}_c(\dot{T}M)_{\Q} ) \to \Gamma^{k-\chi(M)/2}_{M}(M \times \mr{Sym}_c(S^n)_{\Q} ). \] Thus, $\tau^{-1} \circ g_d \circ \tau \circ \xi$ is degree $dk-(1-d)\chi(M)/2$. By Lemma \ref{mapexists}, all fiberwise degree $d$ maps are homotopic to $\tau^{-1} \circ g_d \circ \tau$ and so $f_d \circ \xi$ is degree $dk-(1-d)\chi(M)/2$. 

Now assume $M$ is non-orientable and let $\pi: \tilde M \to M$ be the orientation double cover. Lift $f_d$ to a map $\tilde f_d : \mr{Sym}_c(\dot{T} \tilde M)_{\Q} \to \mr{Sym}_c(\dot{T} \tilde M)_{\Q}$ and lift $\xi$ to a degree $2k$ section of $ \mr{Sym}_c(\dot{T} \tilde M)_{\Q}$. The degree of $\tilde f_d \circ \tilde \xi$ is $2dk-(1-d)\chi(\tilde M)/2=2(dk-(1-d)\chi(M)/2)$. Thus $f_d \circ \xi$ is degree $dk-(1-d)\chi(M)/2$ since lifting a section doubles its degree.
\end{proof}

Using these bundle automorphisms we can compare the components of the section spaces.

\begin{proposition}
Let $M$ be an even-dimensional manifold. For $j,k \neq \chi(M)/2$, $\Gamma^j_{M}(\mr{Sym}_c(\dot{T}M)_{\Q}) \simeq \Gamma^k_{M}(\mr{Sym}_c(\dot{T}M)_{\Q})$.
\label{notorient}
\end{proposition}

\begin{proof}
Since $f_{d} \circ f_{d^{-1}}$ and the identity map are both fiberwise degree one, they are fiberwise homotopic. Thus $f_{d}$ is a fiberwise homotopy equivalence if $d \neq 0$. Also note that for all $j,k \neq \chi(M)/2$, there exists $d \neq 0$ such that $dk-(1-d)\chi(M)/2=j$.
\end{proof}

Finally, using Proposition \ref{notorient} we finish the proof of rational homological stability for bounded symmetric powers of closed manifolds. 

\begin{theorem}\label{thmboundedclosed}
Let $i < \min( \lfloor \frac{k-c}{2c} \rfloor,  \lfloor \frac{j-c}{2c} \rfloor )$ and assume either $n$ is odd or $j,k \neq \chi(M)/2$. If $M$ is of dimension $2$, also assume that $M$ is orientable. Then we have that 
\[H_i(\mr{Sym}_j^{\leq c}(M);\Q) \cong H_i(\mr{Sym}_k^{\leq c}(M);\Q)\]
\end{theorem}

\begin{proof}
Note that the case of $n$ odd was established in Theorem \ref{odd}, so we only need to worry about the case when $n$ is even. By Theorem \ref{thmscanisoclosed}, the scanning maps \[s:\mr{Sym}^{\leq c}_k(M) \to \Gamma^k_M(\mr{Sym}_c(\dot TM)) \qquad \text{and} \qquad s:\mr{Sym}^{\leq c}_j(M) \to \Gamma^j_M(\mr{Sym}_c(\dot TM))\] induce isomorphisms on homology groups $H_i$ since $i <\min( \lfloor \frac{k-c}{2c} \rfloor,  \lfloor \frac{j-c}{2c} \rfloor )$. Since we have that
\begin{align*}H_i(\Gamma^j_M(\mr{Sym}_c(\dot TM)),\Q) &\cong H_i(\Gamma^j_M(\mr{Sym}_c(\dot TM)_{\Q});\Q)\\
H_i(\Gamma^k_M(\mr{Sym}_c(\dot TM)),\Q) &\cong H_i(\Gamma^k_M(\mr{Sym}_c(\dot TM)_{\Q});\Q)\end{align*}
it suffices to show that $\Gamma^j_M(\mr{Sym}_c(\dot TM)_{\Q})$ and $\Gamma^k_M(\mr{Sym}_c(\dot TM)_{\Q})$ are homotopy equivalent. Since $j,k \neq \chi(M)/2$, this follows from Proposition \ref{notorient}. 
\end{proof}

\subsection{Computations of stable homology of bounded symmetric powers} Finally we will compute explicitly the stable rational homology of the bounded powers of manifolds in a few examples. We start however with a calculation of some rational homotopy groups of these spaces. 

Let $M$ be a simply connected manifold of dimension $n$. This implies that $\mr{Sym}^{\leq c}_k(M)$ is also simply connected for $c \geq 2$. In this case, homological stability also implies stability for homotopy groups. If $n$ is odd the situation is simple since $\mr{Sym}_c(S^n)_{\Q} \simeq K(\Q,n)$ by our previous calculations.  If $n$ is even, the map $\mr{Sym}_c(S^n)_{\Q} \to K(\Q,n)$ coming from the rational Postnikov tower is $((c+1)n-2)$-connected. Since simply connected manifolds are orientable, $\dot{T}M_\Q$ is trivial and we can conclude the following. 


\begin{proposition}If $M$ is of dimension $n$ and simply connected, then we have that
\[\mr{colim}_{k \to \infty}\, \pi_*(\mr{Sym}_{k}^{\leq c}(M) )\otimes \Q = \pi_* (F_1(M,S^n_{\Q})) = H^{n-*}(M;\Q).\]
for all $*$ if $n$ is odd and for $* \leq cn-2$ if $n$ is even.\end{proposition}

Theorem 1 of \cite{KS} gives a related result with integral coefficients with a lower range. These results agree with the Dold-Thom theorem, which says that $\mr{Sym}_\infty(X) \simeq \prod_{i > 0} K(H_i(X),i)$. Indeed, $\mr{Sym}^{\leq c}_k(X)$ for $k$ and $c$ sufficiently large is a good approximation of $\mr{Sym}_\infty(X)$.

Using our computation of the minimal model of $\mr{Sym}_c(S^n)_{\Q}$, one can give a minimal model for $F_{\lambda}(M,\mr{Sym}_c(S^n)_{\Q})$ with $\lambda \neq 0$ for $M$ various manifolds.

\begin{example}First of all, one can use Haefliger's approach to computing a minimal model for a mapping space, as in Section 3 of \cite{MR}. This gives for example that if $S$ is a $n$-dimensional homology sphere of even dimension, then a minimal model for $F_{\lambda}(S,\mr{Sym}_c(S^n)_{\Q})$ with $\lambda \neq 0$ is given by the CDGA with algebra the free graded-commutative algebra generated by $b_n$ and $v_{cn-1}$ (the subscripts denote degree) and differential $d(b_n) = 0$ and $d(v_{cn-1}) = b_n^c$. From this one concludes the stable homology $\mr{colim}_{k \to \infty} H_*(\mr{Sym}_{k}^{\leq c}(S);\Q)$ is equal to $\Q$ if $* = kn$ for $0 \leq k \leq c-1$ and $0$ otherwise.\end{example}

\begin{example}Secondly, one can use Brown-Szczarba's approach to computing a minimal model for a mapping space, described in \cite{BS}. This simplifies if $M$ is formal, giving for example the following result in the case $M = \C P^2$. A minimal model for $F_{1}(\C P^2,\mr{Sym}_c(S^n)_\Q)$ is given by the CDGA with algebra the free graded-commutative algebra generated by \[x_2,\,x_4,\,y_{4(c+1)-1},\,y_{4(c+1)-3} \text{ and }y_{4(c+1)-5}\]
where the subscripts denote degree and differential $d$ given by
\begin{align*}&d(x_2) = d(x_4) = 0 \\
&d(y_{4(c+1)-1}) = x_4^{c+1} \\
&d(y_{4(c+1)-3}) = (c+1)x_2x_4^c \\
&d(y_{4(c+1)-5}) = (c+1) x_4^c + \binom{c+1}{2} x_2^2 x_4^{c-1} \end{align*} This recovers known results, in particular the case $c \to \infty$ coming from the Dold-Thom theorem and the case $c=1$ discussed in \cite{kupersmillernote}, which computes the stable homology of $w_{1^j}(\C P^2) = C_j(\C P^2)$ to answer Conjecture G of \cite{VW}.
\end{example}

\bibliography{thesis4}{}
\bibliographystyle{amsalpha}

\end{document}